\documentclass[reqno]{amsart}
\usepackage{amssymb}
\usepackage{amsmath}
\usepackage{changes}
\usepackage[mathscr]{euscript}
\usepackage[small]{caption}
\usepackage{mathtools}
\usepackage{graphicx}
\usepackage{xcolor}
\usepackage{comment}
\usepackage{changes}

\makeatletter
\@addtoreset{equation}{section}
\makeatother

\newtheorem{theorem}{Theorem}[section]
\newtheorem{lemma}[theorem]{Lemma}
\newtheorem{proposition}[theorem]{Proposition}
\newtheorem{corollary}[theorem]{Corollary}

\newtheorem{definition}[theorem]{Definition}
\newtheorem{remark}[theorem]{Remark}

\newtheorem{thm}[theorem]{Theorem}
\newtheorem{lem}[theorem]{Lemma}
\newtheorem{prop}[theorem]{Proposition}
\newtheorem{cor}[theorem]{Corollary}

\newtheorem{rem}[theorem]{Remark}

\newcounter{as}[section]

\newcommand{\mc}[1]{{\mathcal #1}}
\newcommand{\mf}[1]{{\mathfrak #1}}
\newcommand{\mb}[1]{{\mathbf #1}}
\newcommand{\bb}[1]{{\mathbb #1}}

\newcommand{\ms}[1]{{\mathscr #1}}

\newcommand{\<}{\langle}
\renewcommand{\>}{\rangle}
\renewcommand{\Cap}{{\rm cap}}

\definecolor{bblue}{rgb}{.2,0.2,.8}
\definechangesauthor[color=orange]{DM}

\begin{document}

\title{A Resolvent Approach to Metastability}

\begin{abstract}
We provide a necessary and sufficient condition for the metastability
of a Markov chain, expressed in terms of a property of the solutions
of the resolvent equation. As an application of this result, we prove
the metastability of reversible, critical zero-range processes
starting from a configuration.
\end{abstract}

\subjclass[2000]{Primary 60K35 60K40 60J35; Secondary 60F99 60J45.}

\author{C. Landim, D. Marcondes, I. Seo}

\allowdisplaybreaks
\address{IMPA, Estrada Dona Castorina 110, CEP 22460 Rio de Janeiro, Brasil
and CNRS UMR 6085, Universit\'e de Rouen, France. \\
 e-mail: \texttt{landim@impa.br} }
\address{Institute of Mathematics and Statistics, Universidade de São Paulo,
Brazil. \\
 e-mail: \texttt{dmarcondes@ime.usp.br}}
\address{Department of Mathematical Sciences, Seoul National University and
Research Institute of Mathematics, Republic of Korea. \\
 e-mail: \texttt{insuk.seo@snu.ac.kr} }

\maketitle
\section{Introduction}
\label{sec1}

Metastability is a physical phenomenon ubiquitous in first order phase
transitions. A tentative of a precise description can be traced back,
at least, to Maxwell \cite{M75}.

In the mid-1980s, Cassandro, Galves, Olivieri and Vares \cite{CGOV84},
in the sequel of Lebowitz and Penrose \cite{LP71}, proposed a first
rigorous method for deducing the metastable behavior of Markov
processes, based on the theory of large deviations developed by
Freidlin and Wentsel \cite{fw98}. This method, known as the {\it
  pathwise approach to metastability}, was successfully applied to
many models in statistical mechanics \cite{OV05}.

In the following years, different approaches were put forward. In the
beginning of the century, Bovier, Eckhoff, Gayrard and Klein
\cite{BEGK02, BEGK04, BEGK05}, replaced the large deviations tools
with potential theory to derive sharp estimates for the transition
times between wells, the so-called Eyring-Kramers law. We refer to
\cite{BH15} for a comprehensive review of this method, known as the {\it
  potential theoretic approach to metastability}.

More recently, Beltr\'an and Landim \cite{BL1, BL2} characterized the
metastable behavior of a process as the convergence of the order
process, a coarse-grained projection of the dynamics, to a Markov
chain. Inspired by \cite{BEGK04} and based on the martingale
characterization of Markov processes, they provided different sets of
sufficient conditions for metastability.  We refer to \cite{l-review}
for a review of the {\it martingale approach to metastability}.

In this article, we show that the metastable behavior of a sequence of
Markov chains can be read from a property of the solutions of the
resolvent equation associated to the generator of the process.  It
turns out that this property is not only sufficient, but also {\it
necessary} for metastability. This is the content of Theorem
\ref{1t_main2}.

As these conditions for metastability do not rely on the explicit
knowledge of the stationary state, they can, in principle, be
employed to derive the metastable behavior of non-reversible dynamics
whose stationary states are not known.

To emphasize the strength of our method, we show that the necessary
and sufficient conditions for metastability can be derived from the
ones introduced in \cite{BL1, BL2}, which have been proved to hold for
all models whose metastable behavior has been derived through the
potential theoretic method \cite{BEGK04} or the martingale method
\cite{BL1, BL2}.  Moreover, the recent articles \cite{Lee-Seo, lls23a,
lls23b} successfully apply the approach introduced here to
non-reversible overdamped Langevin dynamics.

We further illustrate the extent of possible applications by proving
that the conditions for metastability required in this article hold
for a dynamics with poor mixing properties: reversible condensing
critical zero-range processes. This is a model which does not satisfy
the conditions in \cite{BL1}, and whose metastable behavior could only
be derived so far when the process starts from measures spread over a
well \cite{LMS}. This new approach permits to extend this result to
reversible dynamics in which the process starts from a configuration.

We leave for the future the investigation of metastability of critical
asymmetric zero-range processes. For this model the mixing condition
$\mf M$, introduced in Subsection \ref{sec61}, is very delicate in
that the mixing time is slightly smaller than the escape time. In the
reversible situation considered here, we verify condition $\mf M$
through a careful construction of a sub-harmonic function. It seems
difficult to extend this construction to the non-reversible
case. Beyond condition $\mf M$, all other steps are identical to the
reversible case.

  \subsection*{Recent advancements}

  Before providing a more detailed statement of the main results, we
  review recent progress in the theory of metastability.

  Markov Chain Monte Carlo algorithms have been widely
  used in order to sample from a given Gibbs measure. Their efficiency
  is expressed by the speed of convergence to equilibrium. It has been
  shown in several different contexts that nonreversible dynamics
  converge faster to equilibrium than their reversible counter
  part. This is derived by Kaiser, Jack and Zimmer \cite{kjz17} for
  the large deviations from the hydrodynamic limit of interacting
  particle systems described by the Macroscopic Fluctuation
  Theory. Bouchet and Reygner \cite{br} show that the transition time
  between two wells in overdamped Langevin dynamics is faster in the
  nonreversible case. A similar result appears in \cite{LS1} for
  random walks in potential fields.

  These results raise the problem of finding the non-reversible
  perturbation of a reversible dynamics that does not alter the
  invariant distribution and optimizes the rate of convergence.
  Leli\`evre, Nier and Pavliotis \cite{lnp13} solve this problem for
  overdamped Langevin equations with quadratic potential.  Guillin and
  Monmarch\'e \cite{gm16} show that the asymptotic rate of convergence
  of generalized Ornstein-Uhlenbeck processes is maximized by
  non-reversible hypoelliptic ones.

There are only a few other results on metastability for nonreversible
dynamics.  Le Peutrec and Michel \cite{lm20} obtain by semiclassical
analysis the Eyring-Kramers formula for the exponentially small
eigenvalues of the generator of a nonreversible overdamped Langevin
dynamics associated to a potential which is a Morse function satisfying
additional regularity properties.

In the last years, the close connection between quasi-stationary
states and exponential exit laws have been exploited in many different
directions. Bianchi, Gaudilli\`ere and Milanesi \cite{BG16, BGM18}
expressed the mean transition time in terms of soft capacities and
derived sufficient conditions for metastability in terms of local and
global mixing characteristics of the dynamics. Miclo \cite{M20}
provided an estimate on the distance between the exit time of a set
and an exponential law.  Di Ges\`u, Leli\`evre, Le Peutrec and Nectoux
\cite{GLLN19, LLN20} investigated the distribution of the exit point
from a domain. Berglund \cite{Berg} reviews analytical methods to
derive metastability.   Di Ges\'u \cite{G21} derived,
recently, the Eyring-Kramers formula for the exponentially small
eigenvalues of the generator of reversible discrete diffusions with
semiclassical analysis, an expansion obtained in \cite{LMT, LS1} by
stochastic methods.

We turn to a precise description of the results.

\smallskip\noindent \textbf{The model.}  Consider a sequence of
countable sets $\mc{H}_{N}$, $N\in \bb{N}$, and a collection of
$\mc{H}_{N}$-valued, irreducible, continuous-time Markov chains
$(\xi_{N}(t):t\ge0)$.  To fix ideas, one may think that the sets
$\mc{H}_{N}$ are finite with cardinality increasing to infinity, but
this is not necessary.

Let $S$ be a fixed finite set and $\Psi_{N}:\mc{H}_{N}\to S$ a
projection in the sense that the cardinality of $S$ is much smaller
than that of $\mc{H}_{N}$. Elements of $\mc{H}_{N}$ are represented by
Greek letters $\eta$, $\xi$, while the ones of $S$ by $x$, $y$. The
problem we address is under what conditions the \textit{order process}
$(Y_N (t): t\ge 0)$, defined by $Y_{N}(t)=\Psi_{N}(\xi_{N}(t))$, is
close to a Markovian dynamics which mimics the dynamics of
$\xi_{N}(t)$.

Denote by $\mc{E}_{N}^{x}$ the inverse image of $x\in S$ by
$\Psi_{N}$, $\mc{E}_{N}^{x}=\Psi_{N}^{-1}(x)$, and by $\ms{L}_{N}$ the
generator of the Markov chain $\xi_{N}(t)$. The sets $\mc{E}_N^x$ are
called \textit{wells}.  The following condition plays a central role
in the article.

\smallskip\noindent \textbf{Resolvent condition.}  Fix a function
$g\colon S \to \bb R$, and denote by $G_N\colon \mc H_N\to \bb R$ its
lifting: $G_N (\eta) = \sum_{x\in S} g(x)\, \chi_{\mc E^x_N} (\eta)$,
where $\chi_{\mc A}$, $A\subset\mc H_N$, stands for the indicator of
the set $\mc A$. For $\lambda>0$, denote by $F_{N}$ the solution of
the resolvent equation
\begin{equation}
\label{f01}
(\,\lambda\,-\,\ms{L}_{N}\,)\,F_{N}\;=\; G_N\;.
\end{equation}

Assume that for each $\lambda>0$, $F_N$ is asymptotically constant on
each set $\mc E^x_N$: there exists a function $f:S\to \bb R$ such that
\begin{equation}
\label{f02}
\lim_{N\to\infty}\max_{x\in S}\sup_{\eta\in\mc{E}_{N}^{x}}
\,\big|\,F_{N}(\eta)\,-\,f(x)\,\big|\,\;=\;0\;.
\end{equation}
Of course, $f$ depends on $\lambda$ and on $g$.

Assume, furthermore, that there exists a generator $\ms{L}$ of an
$S$-valued continuous-time Markov chain such that
\begin{equation}
\label{f02b}
(\,\lambda\,-\,\ms{L} \,)\, f\;=\; g\;.
\end{equation}
for all $\lambda>0$, $g:S\to\bb{R}$.

We claim that, under the resolvent conditions \eqref{f02},
\eqref{f02b}, any limit point of the sequence of processes
$Y_{N}(\cdot)=\Psi_{N}(\xi_{N}(\cdot))$ is the law of the
continuous-time Markov chain whose generator is $\ms{L}$.  The
proof of this claim is so simple that we present it below. It relies
on the martingale characterization of Markovian dynamics.

Assume that the sequence $Y_{N}(\cdot)$ converges in law. Fix
$\lambda>0$ and a function $f:S\to\bb{R}$. Denote by $F_{N}$ the
solution of the resolvent equation \eqref{f01} with
$g =(\,\lambda\,-\,\ms{L} \,)\, f$. Since $\xi_{N}(\cdot)$
is a Markov process,
\begin{equation}
M_{N}(t)\;=\; e^{-\lambda t}\, F_{N}(\xi_{N}(t))\,
-\,{F}_{N}(\xi_{N}(0))
\,+\,\int_{0}^{t}e^{-\lambda r}\,
[\, \,( \lambda \,-\, {\ms{L}}_{N}) \,{F}_{N}\,]
(\xi_{N}(r))\,dr
\label{07}
\end{equation}
is a martingale. As $F_{N}$ solves the resolvent equation \eqref{f01},
$\lambda F_{N} - \ms{L}_{N}\,{F}_{N}=G_{N}$. By \eqref{f02},
$F_{N}(\eta)$ is close to $f(\Psi_{N}(\eta))$. Hence, since
$\Psi_{N}(\xi_{N}(t))=Y_{N}(t)$ and $G_N(\eta) = g(\Psi_{N}(\eta))$, we can write
\begin{equation*}
M_{N}(t)\;=\;e^{-\lambda t}\, f(Y_{N}(t))\,-\,f(Y_{N}(0))\,
+\,\int_{0}^{t}\,e^{-\lambda r}\, g(Y_{N}(r))
\,dr\;+\;o_{N}(1)\;,
\end{equation*}
where $o_{N}(1)$ is a small error which vanishes uniformly as
$N\to\infty$.  As $g =(\lambda -\ms{L})f$,
\begin{equation*}
M_{N}(t)\;=\;e^{-\lambda t}\, f(Y_{N}(t)) \,-\,f(Y_{N}(0))
\,+\,\int_{0}^{t} \, e^{-\lambda r}\,
[\, (\lambda \,-\, \ms{L})\, f\,]\,
(Y_{N}(r))\,dr\;+\;o_{N}(1)\;.
\end{equation*}
Passing to the limit yields that any limit point solves the martingale
problem associated to the generator $\ms{L}$. To complete the argument
it remains to recall the uniqueness of solutions of martingale
problems in finite state spaces.

\smallskip
\noindent{\bf The resolvent condition is also necessary.}
The previous approach provides a general method to describe a complex
system, a Markovian dynamics evolving in a large space $\mc{H}_{N}$,
in terms of a much simpler one, an $S$-valued Markov chain. This
abridgement has been named \textit{Markov chain model reduction} or
\textit{metastability}, see \cite{l-review} and references therein.

The point here is that the existence of this synthetic description of
the dynamics can be read from a simple property of the generator.  It
is in force if the resolvent operator
$\ms{U}_{\lambda,N}:=(\lambda-\ms{L}_{N})^{-1}$ sends functions which
are constant on the sets $\mc{E}_{N}^{x}$ to functions which are
asymptotically constant.

The second main point of the article is that conditions \eqref{f02},
\eqref{f02b} are not only sufficient for the convergence of the order
process $Y_N(\cdot)$, but also necessary.

\smallskip
\noindent{\bf Applications.}
The last claim of the article is that this method to derive the
metastable behavior, in the sense of the model reduction described
above, of a sequence of Markov processes can be applied to a wide
range of dynamics. We support this assertion by providing sufficient
conditions for assumptions \eqref{f02}, \eqref{f02b} to hold. These
conditions rely on mixing properties of the dynamics and have been
derived in several different contexts in previous papers.
Furthermore, in the last part of the article, we show that these
conditions are in force for reversible, critical zero-range
dynamics. In particular, we are able to extend the results presented
in \cite{LMS} to the case in which the process starts from a
configuration instead of a measure spread over a well
$\mc{E}_{N}^{x}$.

\smallskip
\noindent{\bf Comments.}
In concrete examples, one has first to find the time-scale $\theta_N$
at which a metastable behavior is observed. Then, one speed-up the
evolution by this quantity and prove all properties of the dynamics in
this new time-scale. Speeding-up the process by $\theta_N$ corresponds
to multiplying the generator by the time-scale $\theta_N$. In the
previous discussion we started from a generator which has already been
speeded-up, which means that the metastable behavior is observed in
the time-scale $\theta_N=1$.

This approach to metastability, inspired from techniques in PDE to
study the asymptotic behavior of solutions of reaction-diffusion
equations \cite{ET, ST}, appeared in the context of Markov processes
in \cite{OR, LS3, RS}.  In these articles, for different models, it is
proved that the solutions of the Poisson equation
$\ms{L}_{N}F_{N}= G_N$ are asymptotically constant in each well.

Replacing the Poisson equation with resolvent equations has a significant
advantage, as the solutions of the later equation are bounded. It
permits, in particular, to prove $L^{\infty}$ estimates instead of
the $L^{2}$ estimates derived in \cite{LMS}. This, in turn, allows
to start the process from a fixed configuration instead of a measure
spread over the sets $\mc{E}_{N}^{x}$.

The existing methods to derive the metastable behavior of a Markov
processes rely on explicit computations involving the stationary state
\cite{BH15, l-review}. In contrast, as already pointed out at the
beginning of this introduction, the deduction of \eqref{f02} and
\eqref{f02b} does not appeal to the stationary state.

\smallskip
\noindent{\bf Introducing a transition region.}
Condition \eqref{f02} is expected to hold only in very special cases,
where the jump rates between configurations belonging to different
sets $\mc{E}_{N}^{x}$ vanish asymptotically. Only in such a case, one
can hope for a discontinuity of the solution of the resolvent equation
\eqref{f01} at the boundary of the set $\mc{E}_{N}^{x}$ [an aftermath
of condition \eqref{f02b}].

To surmount this problem, we introduce a transition set $\Delta_{N}$
which separates the wells $\mc{E}_{N}^{x}$. The set $\Delta_{N}$ has
to be sufficiently large to isolate the wells, but small enough to be
irrelevant from the point of view of the dynamics.

In this new set-up, $\Delta_{N}$, $\mc{E}_{N}$ forms a partition of
the state space $\mc{H}_{N}$, where
$\mc{E}_{N}=\cup_{x\in S}\mc{E}_{N}^{x}$.  To bypass the set
$\Delta_{N}$, we focus our attention on the trace of process
$\xi_{N}(\cdot)$ on $\mc{E}_{N}$ and provide sufficient conditions for
the projection of the trace process to converge to a Markovian
dynamics. This result requires conditions \eqref{f02}, \eqref{f02b}
to hold only on the set $\mc{E}_{N}$, as stated in the first equation.

Furthermore, in this framework, Theorem \ref{1t_main2} asserts that
the resolvent conditions \eqref{f02}, \eqref{f02b} hold if, and only
if, (a) the order process converges to the $S$-valued Markov chain
whose generator is $\mc L$ and (b) the process $\xi_N(\cdot)$ spends
only a negligible amount of time outside the wells $\mc E^x_N$.

\smallskip
\noindent{\bf Critical zero-range processes.}
As mentioned above, this approach is applied to a special class of
zero-range processes. This Markovian dynamics describes the evolution
of particles on a finite set $S$. Denote by $N\ge1$ the total number
of particles and by $\eta=(\eta_{x}:x\in S)$ a configuration of
particles.  Here, $\eta_{x}$ represents the number of particles at
site $x$ for the configuration $\eta$. Let
$\mc{H}_{N}=\{\eta\in\bb{N}^{S}:\sum_{x\in S}\eta_{x}=N\}$ be the
state space.

Particles jump on $S$ according to some rates which will be specified
in the next section. It has been shown that a condensation phenomenon
occurs for this family of rates. The precise statement requires some
notation. Fix a sequence $(\ell_{N}:N\ge1)$ such that $\ell_{N}\to\infty$,
$\ell_{N}/N\to0$. Denote by $\mc{E}_{N}^{x}$, $x\in S$, the
set of configurations given by
\begin{equation*}
\mc{E}_{N}^{x}\;=\;\big\{\eta\in\mc{H}_{N}:\eta_{x}\,\ge\,N-\ell_{N}\}\;.
\end{equation*}
For the models alluded to above, $\mu_{N}(\mc{E}_{N}^{x}) \to 1/|S|$,
where $\mu_{N}$ represents the stationary state of the dynamics
\cite{ev1, jmp, gss, BL3, al, al2, AGL1}.

This means that under the stationary state, essentially all particles
sit on a single site. In consequence, in terms of the dynamics, one
expects the zero-range process to evolve as follows.  When it reaches
a set $\mc{E}_{N}^{x}$, it remains there a very long time, performing
short excursions in $\Delta_{N}$. Its sojourn at $\mc{E}_{N}^{x}$
before it hits a new well $\mc{E}_{N}^{y}$, $y\not = x$, is long
enough for the process to equilibrate inside the well
$\mc{E}_{N}^{x}$.  The transition from $\mc{E}_{N}^{x}$ to a new well
$\mc{E}_{N}^{y}$ is abrupt in the sense that its duration is much
shorter compared to the total time the process stayed in
$\mc{E}_{N}^{x}$.

We apply the method presented at the beginning of this introduction
to derive the asymptotic evolution of the position of the condensate
(the site $x$ where almost all particles sit) for critical, reversible
zero-range dynamics.

The metastable behavior of condensing zero-range processes has a long
history \cite{BL3, Lan, AGL, Seo, OR, LMS}. The critical case,
examined here and in \cite{LMS}, presents a major difference with
respect to the super-critical case considered before. While in the
super-critical case, when entering a well, the process visits all its
configurations before visiting a new well, this is no longer true in
the critical case. This difference prevents the use of the martingale
approach, proposed in \cite{BL1,BL2}, to prove the metastable behavior
of a sequence of Markov chains.

To overcome this problem, we show that in the critical case, when
entering a well, the process hits the bottom of this well before
reaching another well. The proof of this result relies on the
super-harmonic functions constructed in \cite{LMS} and on mixing
properties of the process reflected at the boundary of the wells. The
fact that the process visits one specific configuration inside the
well permits to prove its metastable behavior starting from any
configuration inside a well.

Adding together the property that the process hits quickly the bottom
of a well and that it mixes inside the well before it reaches its
boundary permits to prove that the solution of the resolvent equation
fulfills \eqref{f02}. The proof of property \eqref{f02b} relies also
on a computation of capacities between wells. Details are given in
Sections \ref{sec5}--\ref{sec9}.

To our knowledge, this is the first model which does not visit points
and for which one can prove metastability starting from points and
derive explicit formulae for the time-scale at which metastability
occurs and for the generator $\ms{L}$ of the asymptotic
dynamics.

  Along the same lines, Schlichting and Slowik \cite{ss19} extended
  the investigation of metastability to continuous-time Markov chains
  which do not hit single points. They derived asymptotic sharp
  estimates for mean hitting times by generalizing the potential
  theoretic approach to deal with metastable sets, instead of just
  metastable points. This technique has been applied by Bovier, den
  Hollander, Marello, Pulvirenti and Slowik \cite{bhmps22} to
  inhomogeneous mean-field models.

\smallskip\noindent{\bf Directions for future research.}  As observed
above, it is conceivable to derive properties \eqref{f02},
\eqref{f02b} without turning to the stationary state.  In particular,
this approach might permit to deduce the metastable behavior of
non-reversible dynamics for which the stationary measure is not known
explicitly (say, non-reversible diffusions \cite{br}). Furthermore, proving
properties \eqref{f02} and \eqref{f02b} for a generator $\ms L_N$
becomes an interesting problem since they yield (modulo a third
property) the metastable behavior of the associated Markovian
dynamics.


\section{A Resolvent Approach to Metastability}
\label{sec11}

In this section, we provide a set of sufficient conditions for a
sequence of con\-tin\-uous-time Markov chains to exhibit a metastable
behavior.  If the framework below seems too abstract, the reader may
read this section together with the next, where we apply these results
to a concrete example, the critical zero-range process.

We start introducing the general framework proposed in \cite{BL1, BL2}
to describe the metastable behavior of a Markovian dynamics as a
Markov chain model reduction.  Let $\color{blue} (\mc{H}_{N}:N\ge1)$
be a collection of finite sets. Elements of the set $\mc{H}_{N}$ are
designated by the letters $\eta$, $\xi$, and $\zeta$.

Consider a sequence $\color{blue} (\xi_{N}(t):t\ge0)$ of
$\mc{H}_{N}$-valued, irreducible, continuous-time Markov chains, whose
generator is represented by $\color{blue} \ms{L}_{N}$.  Therefore,
for every function $f: \mc H_N \to \bb R$,
\begin{equation*}
(\ms L_N\, f)(\eta) \;=\; \sum_{\xi\in \mc H_N} R_N(\eta,\xi)\,
\big[\, f(\xi) - f(\eta)\, \big]\;,
\end{equation*}
where $\color{blue} R_N(\eta,\xi)$ stands for the jump rates.  Denote
by $\lambda_N(\eta)$ the holding times of the Markov chain,
$\color{blue} \lambda_N(\eta) = \sum_{\xi \not = \eta} R_N(\eta,\xi)$,
and by $\color{blue} \mu_N$ the unique stationary state.

Denote by $D(\bb{R}_{+},\,\mc{H}_{N})$ the space of right-continuous
functions ${\bf x}:\bb{R}_{+}\to\mc{H}_{N}$ with left-limits, endowed
with the Skorohod topology and its associated Borel
$\sigma$-field. Let ${\color{blue}\mb{P}_{\eta}^{N}}$,
$\eta\in\mc{H}_{N}$, be the probability measure on
$D(\bb{R}_{+},\,\mc{H}_{N})$ induced by the process $\xi_{N}(\cdot)$
starting from $\eta\in\mc{H}_{N}$.  Expectation with respect to
$\mb{P}{}_{\eta}^{N}$ is represented by
${\color{blue}\mb{E}_{\eta}^{N}}$.

Fix a finite set $S$, and denote by $\mc{E}_{N}^{x}$, $x\in S$,
a family of disjoint subsets of $\mc{H}_{N}$. Let
\begin{equation*}
\mc{E}_{N}\,=\,\bigcup\limits _{x\in S}\mc{E}_{N}^{x}
\;\;\;\;\text{and\;\;\;\;}
\Delta_{N}\,=\,\mc{H}_{N}\,\setminus\,
\Big(\,\bigcup_{x\in S}\mc{E}_{N}^{x}\,\Big)\;.
\end{equation*}

The sets $\mathcal{E}_N^x$, $x\in S$, represent the metastable sets of
the dynamics $\xi_N(\cdot)$, in the sense that, as soon as the process
$\xi_N(\cdot)$ enters one of these sets, say $\mathcal{E}_N^x$, it
equilibrates in $\mathcal{E}_N^x$ before hitting a new set
$\mathcal{E}_N^y$, $y\not = x$.  The goal of the theory is to describe
the evolution between these sets. To this end, we introduce the order
process.

For $\mc{A}\subset\mc{H}_{N}$, denote by $T^{\mc{A}}(t)$ the total
time the process $\xi_{N}(\cdot)$ spends in $\mc{A}$ in the
time-interval $[0,t]$:
\begin{equation*}
T^{\mc{A}}(t)\;=\;\int_{0}^{t}\,\chi_{\mc{A}}(\xi_{N}(s))\,ds\;,
\end{equation*}
where \textcolor{blue}{$\chi_{\mc{A}}$} represents the characteristic
function of the set $\mc{A}$. Denote by $S^{\mc{A}}(t)$ the
generalized inverse of $T^{\mc{A}}(t)$:
\begin{equation}
\label{15}
S^{\mc{A}}(t)\;=\;\sup\{\,s\ge0\,:\,T^{\mc{A}}(s)\le t\,\}\;.
\end{equation}

The trace of $\xi_{N}(\cdot)$ on $\mc{A}$, denoted by
$\color{blue}(\xi_{N}^{\mc{A}}(t) : t \ge 0)$, is defined by
\begin{equation}
\label{104}
\xi_{N}^{\mc{A}}(t)\;=\;\xi_{N}(\,S^{\mc{A}}(t)\,)\;;\;\;\;t\ge0\;.
\end{equation}
It is an $\mc{A}$-valued, continuous-time Markov chain, obtained by
turning off the clock when the process $\xi_{N}(\cdot)$ visits the set
$\mc{A}^{c}$, that is, by deleting all excursions to $\mc{A}^{c}$. For
this reason, it is called the trace process of $\xi_{N}(\cdot)$ on
$\mc{A}$.

Let ${\color{blue}\Psi_{N}}:\mc{E}_{N}\to S$ be the projection given by
\begin{equation*}
\Psi_{N}(\eta)\;=\;\sum_{x\in S}x\cdot\chi_{\mc{E}_{N}^{x}}(\eta)\;.
\end{equation*}
The order process \textcolor{blue}{$(Y_{N}(t) : t\ge0)$} is defined
as
\begin{equation}
Y_{N}(t)\;=\;\Psi_{N}(\xi_{N}^{\mc{E}_{N}}(t))\;,\;\;\;t\ge0\;.\label{105}
\end{equation}
Denote by $\color{blue} \bb Q^N_\eta$, $\eta\in\mc E_N$, the
probability measure on $D(\bb R_+, S)$ induced by the measure
$\mb P_\eta^N$ and the order process $Y_N$.

The definition of metastability relies on two conditions. Let $\ms L$
be a generator of an $S$-valued, continuous-time Markov chain.  Denote
by $\color{blue} \bb Q^{\mc L}_x$, $x\in S$, the probability measure
on $D(\bb R_+, S)$ induced by the Markov chain whose generator is
$\ms L$ and which starts from $x$.

\medskip
\noindent
\textbf{Condition $\mf{C}_{\ms{L}}$}. For all $x\in S$ and sequences
$(\eta^{N})_{N\in\bb{N}}$, such that $\eta^{N}\in{\mc{E}}_{N}^{x}$ for
all $N\in\bb{N}$, the sequence of laws
$(\bb Q_{\eta^{N}}^{N})_{N\in\bb{N}}$ converges to $\bb Q_{x}^{\ms L}$
as $N\to\infty$.  \smallskip

The next condition asserts that the process $\xi_{N}(\cdot)$ spends a
negligible amount of time on $\Delta_{N}$ on each finite time
interval. It ensures that the trace process does not differ much from
the original one when starting from a well.

\medskip
\noindent
\textbf{Condition $\mf{D}$}. For all $t>0$,
\begin{equation*}
\lim_{N\to\infty}\,\max_{x\in S}\,
\sup_{\eta\in\mc{E}_{N}^{x}}\mb{E}_{\eta}^{N}\,
\Big[\,\int_{0}^{t}\,\chi_{\Delta_{N}}(\xi_{N}(s))\,ds\,\Big]\;=\;0\;.
\end{equation*}
\smallskip

The next definition is taken from \cite{BL1}.

\begin{definition}
\label{def1}
The process $\xi_N(\cdot)$ is said to be $\ms L$-metastable if
conditions $\mf C_{\ms L}$ and $\mf D$ hold.
\end{definition}

The first main result of this article provides sufficient conditions,
expressed in terms of properties of the solutions of resolvent
equations, for condition $\mf C_{\ms L}$ to hold.  The second one
asserts that these sufficient conditions are also necessary.

Fix a function $g\colon S\to \bb R$, and let
$G_{N}: \mc{H}_{N}\rightarrow\bb{R}$ be its lifting to $\mc H_N$ given
by
\begin{equation}
\label{b32}
G_{N}(\eta)\;=\;\sum_{x\in S}
g (x)\,\chi_{\mc{E}_{N}^{x}}(\eta)\;.
\end{equation}
Note that the function $G_{N}$ is constant on each well
$\mc{E}_{N}^{x}$ and vanishes on $\Delta_{N}$.  For $\lambda >0$,
denote by $F_N = F^{\lambda, g}_N$ the unique solution of the
resolvent equation
\begin{equation}
\label{1f01}
(\, \lambda \,-\, \ms L_N\,) \, F_N \;=\; G_N\;.
\end{equation}

\medskip
\noindent
\textbf{Condition $\mf{R}_{\ms{L}}$}. For all $\lambda>0$ and
$g\colon S\to\bb{R}$, the unique solution $F_{N}$ of the resolvent
equation \eqref{1f01} is asymptotically constant in each set
$\mc{E}_{N}^{x}$:
\begin{equation}
\label{1f02}
\lim_{N\to\infty}
\sup_{\eta\in\mc{E}_{N}^{x}}\,\big|\,F_{N}(\eta)\,-\,f(x)\,\big|
\,\;=\;0\;, \quad x\in S\;,
\end{equation}
where $f:S\rightarrow\bb{R}$ is the unique solution of the reduced
resolvent equation
\begin{equation}
\label{nf06}
(\lambda - \ms L) \, f \,=\, g \;.
\end{equation}
\smallskip

\begin{remark}
\label{rm2}
Condition $\mf R_{\ms L}$ is usually proved in two steps. One first
show that for every $\lambda>0$, $g:S\to\bb R$ the solution $F_N$ of
the resolvent equation \eqref{1f01} is asymptotically constant on each
well. In other words, that \eqref{1f02} holds for some $f$. Then, one
proves that $(\lambda - \ms L)f = g$ for some generator $\ms L$.
\end{remark}

The first  main result of the article reads as follows.

\begin{thm}
\label{1t_main2}
The process $\xi_N(\,\cdot\,)$ is $\ms L$-metastable if, and only if,
condition $\mf{R}_{\ms{L}}$ is fulfilled.  In other words,
Conditions $\mf{D}$ and $\mf{C}_{\ms{L}}$ hold if, and only if,
condition $\mf{R}_{\ms{L}}$ is in force.

\end{thm}

\begin{remark}
This result provides a new tool to prove metastability. The existing
methods rely on explicit computations involving the stationary
state. In particular, they can not be applied to non-reversible
dynamics whose stationary states are not known explicitly. For
example, to small perturbations of dynamical systems or to the
superposition of Glauber and Kawasaki dynamics. Proving that the
solution of a resolvent equation is constant on the wells might be
proven without turning to the stationary state.
\end{remark}

\begin{rem}
\label{rm6}
The introduction of the set $\Delta_{N}$ which separates
the wells makes condition \textbf{$\mf{R}_{\ms{L}}$} plausible.
The challenge is to tune correctly $\Delta_{N}$. Sufficiently large
to prove \textbf{$\mf{R}_{\ms{L}}$}, but small enough
for \textbf{$\mf{D}$} to hold.
\end{rem}

\begin{rem}
\label{rm4}
Solving the martingale problem through a resolvent equation, instead
of a Poisson equation, simplifies considerably the proof of the
metastable behavior of the process. As the solutions of the resolvent
equations are bounded (cf. \eqref{1f10}), one can hope to obtain
bounds and convergence in $L^{\infty}$, as we do here, instead of in
$L^{2}$. Moreover, many $L^{1}$-estimates simplify substantially due
to the $L^{\infty}$-bound on the solution of the resolvent equation.
\end{rem}

\begin{remark}
Condition $\mf R_{\ms L}$ being necessary and sufficient for
metastability implies that it holds for all models whose metastable
behavior have been derived so far. The reader will find in \cite{BH15,
l-review} a list of such dynamics.
\end{remark}

\subsection{Applications}
\label{1sec1.1}

To convince the skeptical reader that condition $\mf R_{\ms L}$ is not
too stringent, besides the fact, alluded to before, that it is also
necessary for metastability, we provide two frameworks where this
condition can be proven. First, Theorem \ref{mt1} states that
condition $\mf R_{\ms L}$ follows from properties (H0) and (H1),
introduced in \cite{BL1, BL2}. Then, in the next section, to
illustrate how to prove condition $\mf R_{\ms L}$ when assumption (H1)
is violated, we prove that it holds for critical condensing zero-range
processes.

The statement of Theorem \ref{mt1} requires some notation. Denote by
$r_N(x,y)$ the mean-jump rate between the sets $\mc{E}_{N}^{x}$ and
$\mc{E}_{N}^{y}$:
\begin{equation}
\label{07b}
r_N(x,y) \;=\; \frac{1}{\mu_N(\mc E^x_N)}
\sum_{\eta\in \mc E^x_N} \mu_N(\eta)  \, \lambda_N(\eta)\,
\mb P^N_{\eta} \big[\,
\tau_{\mc{E}_{N}^{y}} \,<\, \tau^+_{\breve{\mc{E}}_{N}^{y}}
\,\big]\;.
\end{equation}
In this formula, $\tau_{\mc A}$, $\tau^+_{\mc A}$,
$\mc A\subset \mc H_N$, stand for the hitting, return time of the set
$\mc A$, respectively:
\begin{equation}
\label{1.05}
\tau_{\mc A} \;=\; \inf\big\{\, t>0 : \xi_N(t) \in \mc A\,\big\}\;.
\end{equation}
\begin{equation*}
\tau^+_{\mc A} \;=\; \inf\{\, t\ge \sigma_1 : \xi_N(t) \in\mc
A\,\big\}\;, \quad\text{where}\quad
\sigma_1 \;=\; \inf\{\, t\ge 0 : \xi_N(t) \not = \xi_N(0) \,\big\}\;.
\end{equation*}

\smallskip\noindent{\bf Condition (H0)}.  For all $x\not = y\in S$,
the sequence $r_N(x,y)$ converges. Denote its limit by $r(x,y)$:
\begin{equation*}
r(x,y) \;=\; \lim_{N\to \infty} r_N(x,y)\;.
\end{equation*}

Let $\ms{D}_{N}(F)$ be the Dirichlet form of a function
$F:\mc{H}_{N}\rightarrow\bb{R}$ with respect to the generator $\ms{L}_N$:
\begin{equation*}
\ms{D}_{N}(F) \;=\; \<\, F\,,\,
(-\ms{L}_{N}F)\,\>_{\mu_N} \;.
\end{equation*}
A summation by parts yields that
\begin{equation*}
\ms{D}_{N}(F)=\frac{1}{2}\,
\sum_{\eta,\,\eta'\in\mc{H}_{N}}\mu_{N}(\eta)\,
R_{N}(\eta,\,\eta')\,[F(\eta')-F(\eta)]^{2}\;.
\end{equation*}

Fix two disjoint and non-empty subsets $\mc{A}$ and $\mc{B}$ of
$\mc{H}_{N}$. The equilibrium potential between $\mc{A}$ and $\mc{B}$
with respect to the process $\xi_{N}(\cdot)$ is denoted by
$h_{\mc{A},\,\mc{B}}$ and is given by
\begin{equation}
\label{03}
h_{\mc{A},\,\mc{B}}(\eta) \,=\, \mb{P}_{\eta}^{N}
[\, \tau_{\mc{A}}<\tau_{\mc{B}}\, ] \;,\quad
\eta\in\mc{H}_{N}\;.
\end{equation}
The capacity between $\mc{A}$ and $\mc{B}$ is given by
\begin{equation*}
\textup{cap}_{N}(\mc{A},\,\mc{B})
\;=\; \ms{D}_{N}(h_{\mc{A},\,\mc{B}})\;.
\end{equation*}

\medskip\noindent{\bf Condition (H1)}. For each $x\in S$, there exists
a sequence of configurations $(\xi_{N}^x : N\ge 1)$ such that in
$\xi_{N}^x \in \mc E^x_N$ for all $N\ge 1$ and
\begin{equation*}
\lim_{N\to \infty} \max_{\eta\in \mc E^x_N}
\frac{\Cap_N(\mc E^x_N,\breve{\mc E}^x_N)}
{\Cap_N(\xi_N^x, \eta)}\;=\;0\; ,
\end{equation*}
in which $\breve{\mc{E}}_{N}^{x}\;=\;\bigcup_{y\in S\setminus\{x\}}\mc{E}_{N}^{y}\;, x\,\in\,S\;.$

\begin{theorem}
\label{mt1}
Assume that condition {\rm (H0)}, {\rm (H1)} are in force.  Then, the
solution $F_N$ of the resolvent equation \eqref{1f01} is
asymptotically constant on each well $\mc E^x_N$ in the sense that
\begin{equation*}
\lim_{N\rightarrow\infty} \, \max_{x\in S}\,
\max_{\eta, \zeta \in\mc{E}_{N}^{x}}
\,|\,{F}_{N}(\eta)\,-\, F_{N}(\zeta)\,|\;=\;0\;.
\end{equation*}
Furthermore, let $f_N:S \to \bb R$ be the function given by
\begin{equation*}
f_{N}(x)\;=\;
\frac{1}{\mu_{N}(\mc{E}_{N}^{x})}\,
\sum_{\eta\in\mc{E}_{N}^{x}}\,F_{N}(\eta)\,\mu_{N}(\eta)\;,
\quad x\,\in\, S\;,
\end{equation*}
and let $f$ be a limit point of the sequence $f_N$. Then,
\begin{equation}
\label{f011}
\big[\, (\, \lambda \,-\, \ms L_Y\,) \, f\,\big](y) \;=\; g(y)
\end{equation}
for all $y\in S$ such that $\mu_N(\Delta_N)/\mu_N(\mc E^y_N)\to 0$, in which $g$ is the function in equality \eqref{b32}.
In this formula, $\ms L_Y$ is the generator of the continuous-time
Markov process whose jump rates are given by $r(x,y)$, introduced in
{\rm (H0)}.
\end{theorem}

\begin{remark}
Under assumptions {\rm (H0)}, {\rm (H1)} and $\mf D$, condition
$\mf R_{\ms L_Y}$ is in force. Indeed, by \cite[Theorem 2.1]{BL2},
$\mf C_{\ms L_Y}$ holds. Hence, by Theorem \ref{1t_main2},
$\mf R_{\ms L_Y}$ is fulfilled. A direct proof is also possible.
\end{remark}

\begin{remark}
Conditions {\rm (H0)} and {\rm (H1)} have been proved in many
different contexts such as condensing zero-range models \cite{BL3,
Lan, AGL, Seo, LMS}, inclusion processes \cite{BDG, K, KS}, Ising,
Potts and Blume-Capel models at low temperature \cite{LS0, LL, LLM2,
KS2}, random walks and diffusions in potential fields \cite{LMT, LS1,
LS3, RS} and many others \cite{LLM, l-review}.
\end{remark}

\begin{remark}
Lemma \ref{l13} provides a sufficient condition for the identity
\eqref{f011} to hold in the case where
$\mu_N(\Delta_N)/\mu_N(\mc E^y_N)$ does not vanish asymptotically.
\end{remark}

The rest of the article is organized as follows. In Section
\ref{sec2}, we introduce the critical zero-range process and state, in
Theorem \ref{t_main2}, that it fulfills conditions $\mf{R}_{\ms{L}}$
and $\mf{D}$. In Section \ref{sec25}, we prove Theorem \ref{1t_main2}.
In Sections \ref{sec8}--\ref{s7}, we prove Theorem \ref{mt1} and
provide further different sets of sufficient conditions, namely,
conditions $\mathfrak{V}$ and $\mathfrak{M}$, for $\mf{D}$ or
$\mf{R}_{\ms{L}}$ to hold. These families of sufficient conditions
were designed to encompass most dynamics whose metastable behavior
have been derived so far.  Sections \ref{sec5}-\ref{sec9} of the
article are devoted to the proof of Theorem \ref{t_main2}.

\section{Critical Zero-range Dynamics\label{sec2}}

In this section, we introduce the critical condensing zero-range process
to which we apply the resolvent approach described in the previous
section. Fix a finite set $S$ with ${\color{blue}|S|=\kappa\ge2}$
elements, and consider a continuous-time Markov chain on $S$ with
generator $L_{X}$ acting on functions $f:S\rightarrow\bb{R}$
as
\begin{equation*}
(L_{X}f)(x)\;=\;\sum_{y\in S}\,r(x,\,y)\,[\,f(y)\,-\,f(x)\,]
\end{equation*}
for some jump rate $r:S\times S\rightarrow\bb{R}_{+}$ assumed to be
symmetric {[}$r(x,\,y)=r(y,\,x)$ for all $x,\,y\in S${]}.  Set
$r(x,\,x)=0$ for all $x\in S$ for convenience. Denote by
\textcolor{blue}{$(X(t))_{t\geq0}$} the Markov chain generated by
$L_{X}$ and assume that this chain is irreducible. Note that the
process $X(\cdot)$ is reversible with respect to the uniform measure
$m(\cdot)$ on $S$ {[}$m(x)=1/\kappa$ for all $x\in S${]}.

The zero-range process describes the evolution of particles on $S$.
A configuration $\eta\in\bb{N}^{S}$ of particles is written as
$\eta=(\eta_{x})_{x\in S}$ where $\eta_{x}$ represents the number
of particle at $x$ under the configuration $\eta$. For $N\in\bb{N}$
and $S_{0}\subset S$, denote by $\mc{H}_{N,\,S_{0}}\subset\bb{N}^{S_{0}}$
the subset of configurations on $S_{0}$ with exactly $N$ particles:
\begin{equation}
\mc{H}_{N,\,S_{0}}\;=\;\Big\{\,\eta\in\bb{N}^{S_{0}}\,:\,\sum_{x\in S_{0}}\eta_{x}\,=\,N\,\Big\}\;.\label{e_Hn}
\end{equation}
Let ${\color{blue}\mc{H}_{N}=\mc{H}_{N,\,S}}$. The critical
zero-range process is the continuous-time Markov chain
\textcolor{blue}{$\{\eta_{N}(t)\}_{t\ge0}$} on $\mc{H}_{N}$ with
generator acting on functions $F:\mc{H}_{N}\rightarrow\bb{R}$ as
\begin{equation}
(\ms{L}_{N}F)(\eta)\;=\;\sum_{x,\,y\in S}\,g(\eta_{x})\,r(x,\,y)\,[\,F(\sigma^{x,\,y}\eta)\,-\,F(\eta)\,]\;,\;\;\;\;\eta\in\mc{H}_{N}\;,\label{e_genZRP}
\end{equation}
where
\begin{equation*}
g(0)\,=\,0\;,\;\;\;\;g(1)\,=\,1\;,\text{\;\;\; and \;\;\;\;}g(n)\,=\,\frac{n}{n-1}\text{ for }n\ge2\;.
\end{equation*}
In this equation, $\sigma^{x,\,y}\eta$, $x,\,y\in S$, stands for
the configuration obtained from $\eta$ by moving a particle from
$x$ to $y$, when there is at least one particle at $x$:
\begin{equation*}
(\sigma^{x,\,y}\eta)_{z}=\begin{cases}
\eta_{x}-1 & \text{ if }z=x\;,\\
\eta_{y}+1 & \text{ if }z=y\;,\\
\eta_{z} & \text{ otherwise\;,}
\end{cases}
\end{equation*}
if $\eta_{x}\ge1$. Otherwise, $\sigma^{x,\,y}\eta=\eta$.

\subsection{Condensation of particles\label{sec22}}

It is elementary to check that the unique invariant measure for the
irreducible Markov chain $\eta_{N}(\cdot)$ is given by
\begin{equation*}
\mu_{N}(\eta)\;=\;\frac{N}{Z_{N,\,\kappa}
  \,(\log N)^{\kappa-1}}\,\frac{1}{\mb{a}(\eta)}\;,
\end{equation*}
where
\begin{equation*}
\mb{a}(\eta)\;=\;\prod_{x\in S}a(\eta_{x})\text{ \;\;\;with\;\;\; }a(n)\;=\;\max\{n,\,1\}\text{ for }n\ge0\;,
\end{equation*}
and where the partition function $Z_{N,\,\kappa}$ is defined by
\begin{equation}
Z_{N,\,\kappa}\;=\;\frac{N}{(\log N)^{\kappa-1}}\,\sum_{\eta\in\mc{H}_{N}}\frac{1}{\mb{a}(\eta)}\;.\label{e_partition}
\end{equation}
The factor $N/(\log N)^{\kappa-1}$ was introduced so that $Z_{N,\,\kappa}$
has a non-degenerate limit when $N$ tends to infinity: By \cite[Proposition 4.1]{LMS},
$\lim_{N\to\infty}Z_{N,\,\kappa}=\kappa$. Furthermore, the zero-range
process $\eta_{N}(\cdot)$ is reversible with respect to $\mu_{N}(\cdot)$.

Define the metastable well as
\begin{equation*}
\mc{E}_{N}^{x}\;=\;\{\eta\in\mc{H}_{N}\,:\,\eta_{x}\,\geq\,N-\ell_{N}\}\;;\;\;\;x\in S\;,
\end{equation*}
where $\ell_{N}$ is any sequence satisfying
\begin{equation*}
\lim\limits _{N\to\infty}\frac{\ell_{N}}{N}\,=\,0\;\;\;\;\text{and\;\;\;\;}\lim\limits _{N\to\infty}\frac{\log\ell_{N}}{\log N}\,=\,1\;.
\end{equation*}
Assume that \textcolor{blue}{$\ell_{N}=N/\log N$ }for simplicity.
The set $\mc{E}_{N}^{x}$ can be regarded as a collection of
configurations in which almost all particles are sitting at site $x$.
As defined previously, let
\begin{equation*}
\mc{E}_{N}\,=\,\bigcup\limits _{x\in S}\mc{E}_{N}^{x}\;\;\;\;\text{and\;\;\;\;}\Delta_{N}\,=\,\mc{H}_{N}\,\setminus\,\Big(\,\bigcup_{x\in S}\mc{E}_{N}^{x}\,\Big)
\end{equation*}
so that $\mc{H}_{N}=\mc{E}_{N}\cup\Delta_{N}$ gives a partition
of $\mc{H}_{N}.$ The following result is \cite[Theorem 2.3]{LMS}.
\begin{thm}
\label{t_cond}
For all $x\in S$, $\lim_{N\rightarrow\infty}\mu_{N}(\mc{E}_{N}^{x})=1/\kappa$
. In particular,
\begin{equation*}
\lim_{N\rightarrow\infty}\mu_{N}(\mc{E}_{N})\,=\,1\;\;\;\;\text{and\;\;\;\;}\lim_{N\rightarrow\infty}\mu_{N}(\Delta_{N})\,=\,0\;.
\end{equation*}
\end{thm}

Hence, as $N\rightarrow\infty$, under the invariant measure, almost
all particles are condensed at a single site. In this sense, the critical
zero-range process $\eta_{N}(\cdot)$ condensates. The main result
of the article describes the evolution of the condensate.

This model is said to be ``critical'' for the following reason.
Suppose that we replace $g(\eta_{x})$ in \eqref{e_genZRP} by $[g(\eta_{x})]^{\alpha}$
for some $\alpha>0$. It is known that the condensation phenomenon
occurs if $\alpha\ge1$, while a diffusive behavior without condensation
is observed if $\alpha<1$. For this reason the zero-range process
$\eta_{N}(\cdot)$ is said to be critical at $\alpha=1$.

\subsection{Order process}
\label{sec23}

Let ${\color{blue}\theta_{N}=N^{2}\log N}$ be the time-scale at which
the condensate moves, and denote by
\textcolor{blue}{$\xi_{N}(\cdot)$} the process obtained by speeding
up the zero-range process $\eta_{N}(\cdot)$ by $\theta_{N}$, i.e.,
$\xi_{N}(t)=\eta_{N}(t\,\theta_{N})$ for all $t\ge0$. Note that the
process \textcolor{blue}{$\xi_{N}(\cdot)$ }is the $\mc{H}_{N}$-valued,
continuous-time Markov chain whose generator is given by
${\color{blue}\ms {L}_{N}^{\xi}=\theta_{N}\,\ms{L}_{N}}$.

Denote by \textcolor{blue}{ $\mb{P}_{\eta}^{N}$ }the
probability measure on $D(\bb{R}_{+},\,\mc{H}_{N})$ induced
by the process $\xi_{N}(\cdot)$ starting from $\eta\in\mc{H}_{N}$
and by \textcolor{blue}{$\mb{E}_{\eta}^{N}$ }the expectation
with respect to $\mb{P}_{\eta}^{N}$.

Recall from \eqref{104}, \eqref{105} the definition of the trace
process ${\color{blue}(\xi_{N}^{\mc{E}_{N}}(t))_{t\ge0}}$,
of the projection $\Psi_{N}:\mc{E}_{N}\to S$, and of the order
process $(Y_{N}(t))_{t\ge0}$. For critical zero-range processes,
the order process $Y_{N}(\cdot)$ specifies the position of the condensate
for the trace process $\xi_{N}^{\mc{E}_{N}}(t)$. Recall that
{$\bb Q_{\eta}^{N}$}, $\eta\in\mc{H}_{N}$, denotes the
probability law on $D(\bb{R}_{+},\,S)$ induced by the order process
$Y_{N}(\cdot)$ when the underlying zero-range process $\xi_{N}(\cdot)$
starts from $\eta$.

\subsection{Main result\label{sec24}}

We first introduce the $S$-valued Markov chain $(Y(t))_{t\ge0}$
describing the evolution of the condensate. Denote by $\tau_{C}^{X}$,
$C\subset S$, the hitting time of the set $C$ with respect to the
random walk $X(\cdot)$ introduced above.
\begin{equation*}
\tau_{C}^{X}\;=\;\inf\big\{\,t>0:X(t)\in C\,\big\}\;.
\end{equation*}

Let ${\color{blue}\bb Q_{x}^X}$, $x\in S$, be the law of the process
$X(\cdot)$ starting at $x$. For two non-empty, disjoint subsets $A$,
$B$ of $S$, the equilibrium potential between $A$ and $B$ with respect
to the process $X(\cdot)$ is the function
$h_{A,\,B}^X:S\rightarrow\bb{R}$ defined by
\begin{equation}\label{epX}
h_{A,\,B}^X(x)\;=\;\bb Q_{x}^X \,[\,\tau_{A}^{X}<\tau_{B}^{X}\,]\;,
\quad x\in S \;.
\end{equation}

The capacity between $A$ and $B$ is given by
\begin{equation}
{\rm {cap}}_{X}(A,\,B)\;=\;D_{X}(h_{A,\,B}^{X})\;,\label{e_capX}
\end{equation}
where $D_{X}(\cdot)$ stands for the Dirichlet form associated to
the process $X(\cdot)$, which can be written as
\begin{equation}
D_{X}(f)\;=\;\frac{1}{2}\sum_{x,\,y\in S}\,m(x)\,r(x,\,y)\,[\,f(y)-f(x)\,]^{2}\label{e_DirX}
\end{equation}
for $f:S\rightarrow\bb{R}$. If the sets $A$, $B$ are singletons,
we write ${\rm {cap}}_{X}(x,\,y)$ instead of $\textrm{cap}_{X}(\{x\},\,\{y\})$.

Denote by \textcolor{blue}{$(Y(t))_{t\ge0}$} the $S$-valued,
continuous-time Markov chain associated to the generator $L_{Y}$
acting on $f:S\rightarrow\bb{R}$ as
\begin{equation}
(L_{Y}f)(x)\;=\;6\,\kappa\,\sum_{y\in S}
\,{\rm cap}_{X}(x,\,y)\,[\,f(y)-f(x)\,]\;.
\label{e_genZ}
\end{equation}

Recall from the previous section that we denote by $\bb Q_{x}^{L_Y}$,
$x\in S$, the probability measure on $D(\bb{R}_{+},\,S)$ induced by
the Markov chain $Y(\cdot)$ starting from $x$. Sometimes, we represent
$\bb Q_{x}^{L_Y}$ by ${\color{blue}\bb Q_{x}^{Y}}$.

The next theorem is the third main result of the article.

\begin{thm}
\label{t_main2}
Conditions $\mf{R}_{L_{Y}}$ and $\mf{D}$ hold for the critical
zero-range process, where $L_{Y}$ is given by \eqref{e_genZ}. In
particular, the critical zero-range process is $L_{Y}$-metastable.
\end{thm}

In view of Theorem \ref{1t_main2}, this result establishes that, in
the time-scale $\theta_{N}=N^{2}\log N$, the condensate evolves as the
Markov chain $Y(\cdot)$ and that, with the exception of time intervals
whose total length is negligible, almost all particles sit on a single
site.

\begin{rem}
\label{rm}
The so-called martingale approach developed in \cite{BL1,BL2} to
derive the metastable behavior of a Markov process, based on potential
theory, does not apply here because the process does not visit all
points of a well before jumping to a new one, and condition (H1) of
\cite{BL1} is violated. This characteristic is the main difference
between the critical zero-range process and the super-critical ones.
\end{rem}

\begin{rem}
\label{rm11}
By using the so-called Poisson equation approach developed in
\cite{LS3, OR, RS}, we proved in \cite{LMS} a weaker version of
Theorem \ref{t_main2}.  Denote by $\mu_{N}^{x}(\cdot)$, $x\in S$, the
measure on $\mc{E}_{N}^{x}$ obtained by conditioning $\mu_{N}$ on
$\mc{E}_{N}^{x}$:
\begin{equation*}
\mu_{N}^{x}(\eta)\;=\;\frac{\mu_{N}(\eta)}{\mu_{N}(\mc{E}_{N}^{x})}
\;,\quad\eta\,\in\,\mc{E}_{N}^{x}\;.
\end{equation*}
We assumed in \cite{LMS} that the initial distribution is a measure
$\nu_{N}$ concentrated on a set $\mc{E}_{N}^{x}$ for some $x$, and
satisfying the following $L^{2}$-condition: there exists a finite
constant $C_{0}$ such that
\begin{equation}
E_{\mu_{N}^{x}}\,\Big[\,\Big(
\frac{d\nu_{N}}{d\mu_{N}^{x}}\Big)^{2}\,\Big]
\;=\;\sum_{\eta\in\mc{E}_{N}^{x}}
\frac{\nu_{N}(\eta)^{2}}{\mu_{N}^{x}(\eta)}
\;\le\;C_{0}\;\;\;\;\;\text{for all }N\in\bb{N}\;.
\label{e_L2con}
\end{equation}
The main novelty of Theorem \ref{t_main2} lies in the fact that it
removes assumption \eqref{e_L2con} and allows the process to start
from a fixed configuration inside some well.
\end{rem}

\begin{rem}
\label{rm5} The proof of Theorem \ref{t_main2} relies on many estimates
obtained in \cite{LMS}. In particular, on the construction of a super-harmonic
function inside the wells.
\end{rem}

\begin{rem}
\label{rm10}
The equilibration inside the well, or the loss of memory, is obtained
in two different manners. First, we derive a sharp bound on the
relaxation time of the process reflected at the boundary of a
well. This relaxation time is shown to be much smaller than the
metastable time-scale $\theta_{N}$.

Then, we show that the process visits the bottom of the well before
visiting a new well. This crucial property is derived with the help
of the super-harmonic function alluded to above. Thus, although the
process does not visit all configurations in a well before reaching
a new one, it visits a specific configuration.
\end{rem}

\begin{rem}
\label{rm1}
The symmetry of the jump rates $r$ of the chain $X$ is used in the
construction of the super-harmonic function. Theorem \ref{t_main2}
should be in force without this assumption, but a proof is missing.
\end{rem}

The proof of Theorem \ref{t_main2} relies on Theorem \ref{1t_main2}.
The strategy is presented in Section \ref{sec5} and the details in
Sections \ref{sec3}--\ref{sec9}.

\section{Proof Theorem \ref{1t_main2}}
\label{sec25}

In the first part of this section, we show that condition
$\mf{R}_{\ms{L}}$ implies assertions $\mf{C}_{\ms{L}}$ and
$\mf{D}$. In the second part, we prove the reverse statement.

\subsection{Condition $\mf{R}_{\ms{L}}$ entails $\mf{C}_{\ms{L}}$ and
$\mf{D}$}

We first show that the solution of the resolvent equation is bounded.
Fix a function $g:S\to\bb{R}$ and $\lambda>0$. It is well known that
the solution of the resolvent equation \eqref{1f01} can be represented
as
\begin{equation}
{F}_{N}(\eta)\;=\;\mb{E}_{\eta}^{N}\,
\Big[\,\int_{0}^{\infty}\,e^{-\lambda s}\,
G_{N}(\xi_{N}(s))\,ds\,\Big]\;.
\label{1f11}
\end{equation}
In particular, there exists a finite constant $C_{0}=C_{0}(\lambda,g)$
such that
\begin{equation}
\max_{\eta\in\mc{H}_{N}}|\,{F}_{N}(\eta)\,|\;\le\;C_{0}\;.
\label{1f10}
\end{equation}

The next result asserts that condition $\mf{R}_{\ms{L}}$ implies assertion
$\mf{D}$.

\begin{lem}
\label{nl1}
Assume that condition $\mf{R}_{\ms{L}}$ holds. Then, assertion
$\mf{D}$ is in force.
\end{lem}

\begin{proof}
We first claim that for all $\lambda>0$,
\begin{equation}
\label{nf01}
\lim_{N\to\infty}\,
\max_{\eta\in\mc{E}_{N}}\mb{E}_{\eta}^{N}\,
\Big[\,\int_{0}^{\infty}\,e^{-\lambda s}\,
\chi_{\Delta_{N}}(\xi_{N}(s))\,ds\,\Big]\;=\;0\;.
\end{equation}
Indeed, fix $\lambda>0$ and $g:S \to \bb R$ given by $g(x)=1$ for all
$x\in S$. Let $G_N$, $F_N$ be given by \eqref{b32} and \eqref{1f01},
respectively. By \eqref{1f11} and since $G_N = \chi_{_{\mc E_N}}$, for
all $\eta\in \mc H_N$,
\begin{equation*}
F_N(\eta) \;-\; \frac{1}{\lambda}\;=\;
\mb{E}_{\eta}^{N}\,
\Big[\,\int_{0}^{\infty}\,e^{-\lambda s}\,
[\, G_{N}(\xi_{N}(s)) \,-\, 1\,] \,ds\,\Big]
\;=\; -\,
\mb{E}_{\eta}^{N}\,
\Big[\,\int_{0}^{\infty}\,e^{-\lambda s}\,
\chi_{\Delta_N} (\xi_{N}(s)) \,ds\,\Big]\;.
\end{equation*}
Since the solution $f$ of the reduced resolvent equation
$(\lambda - \ms L) f = g$ is $f(x) = 1/\lambda$ for all $x\in S$,
claim \eqref{nf01} follows from \eqref{1f02}.

Fix $t>0$, $\lambda>0$, and observe that
\begin{equation*}
\begin{aligned}
\mb{E}_{\eta}^{N}\,
\Big[\,\int_{0}^{t}\,
\chi_{\Delta_N} (\xi_{N}(s)) \,ds\,\Big] \; & \le\;
e^{\lambda t} \, \mb{E}_{\eta}^{N}\,
\Big[\,\int_{0}^{t}\,e^{-\lambda s}\,
\chi_{\Delta_N} (\xi_{N}(s)) \,ds\,\Big] \\
& \le\;
e^{\lambda t} \, \mb{E}_{\eta}^{N}\,
\Big[\,\int_{0}^{\infty}\,e^{-\lambda s}\,
\chi_{\Delta_N} (\xi_{N}(s)) \,ds\,\Big]
\end{aligned}
\end{equation*}
for all $\eta\in\mc H_N$. Hence, condition $\mf D$ follows from
\eqref{nf01}.
\end{proof}

We prove some consequences of condition $\mf{R}_{\ms{L}}$.  The next
result asserts that the process $\xi_{N}(\cdot)$ can not jump from one
well to another quickly. The proof of this result is similar to the
one of \cite[Proposition 5.2]{RS}. Recall from \eqref{1.05} that we
denote by $\tau_{\mc A}$, $\mc A\subset \mc H_N$, the hitting time of
the set $\mc A$.  Let
\begin{equation}
\label{breveE}
\breve{\mc{E}}_{N}^{x}\;=\;
\bigcup_{y\in S\setminus\{x\}}\mc{E}_{N}^{y}\;,\quad x\,\in\,S\;.
\end{equation}

\begin{lem}
\label{R3}
Assume that condition $\mf{R}_{\ms{L}}$ holds.
Then, for all $x\in S$,
\begin{equation}
\limsup_{t\rightarrow0}\limsup_{N\rightarrow\infty}
\sup_{\eta\in\mc{E}_{N}^{x}}
\mb{P}_{\eta}^{N}\,[\,\tau_{\breve{\mc{E}}_{N}^{x}}\,<\,t\,]\;=\;0\;.
\label{r3}
\end{equation}
\end{lem}

\begin{proof}
Fix $\lambda>0$, $x\in S$ and $\eta^N\in\mc{E}_{N}^{x}$. Let
$f:S\to\bb{R}$ be the function given by $f(y)=1-\delta_{x,y}$. Set
$g = (\lambda - \ms L) f$, and denote by $F_{N}$ the solution of the
resolvent equation \eqref{1f01}. Let $M_N(t)$ be the martingale
defined by
\begin{equation*}
M_N(t) \;=\;  F_{N}(\xi_{N}(t))\,
-\,{F}_{N}(\xi_{N}(0))
\,-\,\int_{0}^{t} (\, {\ms{L}}_{N} \,{F}_{N}\,)
(\xi_{N}(r))\,dr\;.
\end{equation*}
As $\ms{L}_{N}F_{N}=\lambda F_{N}-G_{N}$, for every $t>0$,
\begin{equation}
\mb{E}_{\eta}^{N}\big[\,F_{N}(\xi_{N}(t\wedge\tau))\,\big]
\;=\;{F}_{N}(\eta^N)\,+\,
\mb{E}_{\eta}^{N}\Big[\,\int_{0}^{t\wedge\tau}\,
(\lambda{F}_{N}-G_{N})(\xi_{N}(r))\,dr\,\Big]\;,
\label{08}
\end{equation}
where $\tau=\tau_{\breve{\mc{E}}_{N}^{x}}$.

By condition $\mf{R}_{\ms{L}}$ and the definition of $f$,
$\lim_{N\to\infty}F_{N}(\eta^N)=0$. By \eqref{1f10} and by definition
of $G_{N}$, $\lambda{F}_{N}-G_{N}$ is bounded. The right-hand side of
\eqref{08} is thus bounded by $a_{N}+C_{0}t$ for some finite constant
$C_{0}$ and a sequence $a_{N}$ such that $a_{N}\to0$.

We turn to the left-hand side of \eqref{08}. Since $f\ge0$, by
condition $\mf R_{\ms L}$, there exists a constant $c_{N}\ge0$ such
that $c_{N}\to0$ and $\widetilde{F}_{N}(\zeta)=F_{N}(\zeta)+c_{N}\ge0$
for all $\zeta\in\mc{E}_{N}$.

\smallskip \noindent
\textit{Claim A: $\widetilde{F}_{N}(\zeta)\ge0$ for all $\zeta\in\mc{H}_{N}$.}
\smallskip

\noindent To prove this claim, let $\zeta$ be a configuration at which
$\widetilde{F}_{N}$ achieves its minimum value so that
$(\ms{L}_{N}F_{N})(\zeta)=(\ms{L}_{N}\widetilde{F}_{N})(\zeta)\ge0$.
If $\zeta\in\mc{E}_{N}$, there is nothing to prove. If
$\zeta\in\Delta_{N}$, since $G_{N}$ vanishes on $\Delta_{N}$ and since
$(\lambda-\ms{L}_{N})F_{N}=G_{N}$,
$F_{N}(\zeta)=\lambda^{-1}(\ms{L}_{N}F_{N})(\zeta)\ge0$ so that
$\widetilde{F}_{N}(\zeta)=F_{N}(\zeta)+c_{N}\ge0$, as
claimed. \smallskip

The left-hand side of \eqref{08} is equal to
$\mb{E}_{\eta}^{N}[\,\widetilde{F}_{N}(\xi_{N}(t\wedge\tau))\,]-c_{N}$.
By $\mf R_{\ms L}$, for $N$ sufficiently large,
$\widetilde{F}_{N}(\zeta)\ge1/2$ on $\breve{\mc{E}}_{N}^{x}$. Hence,
the left-hand side of \eqref{08} is bounded below by
$(1/2)\mb{P}_{\eta}^{N}[\,\xi_{N}(t\wedge\tau)\in\breve{\mc{E}}_{N}^{x}]-c_{N}$.

Putting together the previous estimates yields that
\begin{equation*}
\mb{P}_{\eta}^{N}\big[\,\xi_{N}(t\wedge\tau)
\in\breve{\mc{E}}_{N}^{x}\,\big]\;
\le\;2\,\big(c_{N}\,+\,a_{N}\,+\,C_{0}t\,\big)\;.
\end{equation*}
To complete the proof of the lemma, it remains to remark that
\begin{equation*}
\mb{P}_{\eta}^{N}\,[\,\tau_{\breve{\mc{E}}_{N}^{x}}\,<\,t\,]
\;\le\;\mb{P}_{\eta}^{N}\big[\,\xi_{N}(t\wedge\tau)
\in\breve{\mc{E}}_{N}^{x}\,\big]\;.
\end{equation*}
\end{proof}

The next result states that the sequence
$(\bb Q_{\eta^{N}}^{N})_{N\in\bb{N}}$ is tight.

\begin{prop}
\label{1t_tight}Assume that condition $\mf{D}$ and \eqref{r3}
hold. Then, the sequence $(\bb Q_{\eta^{N}}^{N})_{N\in\bb{N}}$
is tight, and any limit point $\bb Q^{*}$ of this sequence is
such that
\begin{equation}
\bb Q^{*}\,[\,Y(t)\not=Y(t-)\,]\,=\,
0\text{ for all }t>0\;.\label{1e_leftcont}
\end{equation}
\end{prop}

\begin{proof}
This result follows from conditions $\mf{D}$, \eqref{r3} and
Aldous' criterion. We refer to \cite[Theorem 5.4]{LMS} for a proof.
\end{proof}

Recall from \eqref{15} the definition of the time-change
$S^{\mc A}(t)$, $\mc A \subset\mc H_N$.  Clearly, for all $t\ge0$,
\begin{equation}
\label{11}
T^{\mc{A}}\big(\,S^{\mc{A}}(t)\,\big)\;=\;t\;.
\end{equation}
In contrast, we only have
$S^{\mc{A}}\big(\,T^{\mc{A}}(t)\,\big)\ge t$ and a strict
inequality may occur. Furthermore, for all $t>0$ and $\epsilon>0$,
\begin{equation}
\label{12}
\Big\{\,S^{\mc{A}}(t)\,-\,t\;\ge\;
\epsilon\,\Big\}\;\subset\;
\Big\{\,\int_{0}^{t+\epsilon}\chi_{\mc{A}^{c}}(\xi_{N}(s))\,ds
\;\ge\;\epsilon\,\Big\}\;.
\end{equation}
Indeed, if $S^{\mc{A}}(t)\ge t+\epsilon$, applying $T^{\mc{A}}$
on both sides of this inequality, as $T^{\mc{A}}$ is an increasing
function, by \eqref{11},
\begin{equation*}
t\;\ge\;T^{\mc{A}}(\,t+\epsilon\,)
\quad\text{so that}\quad
t\;+\;\epsilon\;-\;T^{\mc{A}}(\,t+\epsilon\,)\;\ge\;\epsilon\;.
\end{equation*}
This last relation corresponds exactly to the right-hand side of
\eqref{12}.

Denote by ${\color{blue}\{{\ms{F}}_{t}^{\,0}\}_{t\ge0}}$ the natural
filtration of $D(\bb{R}_{+},\,\mc{H}_{N})$ generated by the process
$\xi_{N}(\cdot)$, $\ms{F}_{t}^{\,0}=\sigma(\xi_{N}(s):s\in[0,\,t])$,
and by ${\color{blue}\{{\ms{F}}_{t}\}_{t\ge0}}$ its usual
augmentation.  Let $\ms{G}_{t}^{N}$ be the filtrations defined by
\begin{equation}
\ms{G}_{t}^{N}\;:=\;\ms{F}_{S^{\mc{E}_{N}}(t)}\;\;\;\text{for }t>0\;.
\label{1e_filt}
\end{equation}

\begin{lem}
\label{l01}
Assume that condition $\mf{D}$ is in force. Then, for all $\lambda>0$
and $t>0$,
\begin{equation*}
\lim_{N\to\infty}\sup_{\eta\in\mc{E}_{N}}\mb{E}_{\eta}^{N}\,
\Big[\,e^{-\lambda t}\,-\,e^{-\lambda S^{\mc{E}_{N}}(t)}\,\Big]
\;=\;0\;,
\end{equation*}
and
\begin{equation*}
\lim_{N\to\infty}\sup_{\eta\in\mc{E}_{N}}\mb{E}_{\eta}^{N}\,
\Big[\,\int_{0}^{t}\,\Big\{ e^{-\lambda r}\,
-\,e^{-\lambda S^{\mc{E}_{N}}(r)}\,\Big\}\,dr\,\Big]\;=\;0\;.
\end{equation*}
Note that these expressions are positive since
$S^{\mc{E}_{N}}(r)\ge r$ for all $r\ge0$.
\end{lem}

\begin{proof}
To prove the first assertion, note that the expectation is bounded
by
\begin{equation*}
\mb{E}_{\eta}^{N}\,\big[\,K_\lambda\big(\,S^{\mc{E}_{N}}(t)-t\,\big)\,\big] \;,
\end{equation*}
where $K_\lambda(a)=1\,-\,e^{-\lambda a}$.

Fix $\epsilon>0$. As $K_\lambda$ is continuous, there exists
$\delta>0$ such that $K_\lambda(a)\le\epsilon$ for all
$0\le a\le\delta$. Since $K_\lambda$ is bounded by $1$, the previous
expectation is less than or equal to
\begin{equation*}
\epsilon\;+\;\mb{P}_{\eta}^{N}\,\big[\,S^{\mc{E}_{N}}(t)-t\,>\,\delta\,\big]\;.
\end{equation*}
By \eqref{12} and Chebyshev's inequality, this expression is bounded
by
\begin{equation*}
\epsilon\;+\;\frac{1}{\delta}\,\mb{E}_{\eta}^{N}\,
\big[\,\int_{0}^{t+\delta}\chi_{\Delta_{N}}(\xi_{N}(s))\,ds\,\big]\;.
\end{equation*}
At this point, the first claim of the lemma follows from condition
$\mf{D}$ by taking the limit $N\to\infty$ and then $\epsilon\to0$.

The proof of the second assertion is similar. The expectation is equal
to
\begin{equation*}
\mb{E}_{\eta}^{N}\,\Big[\,\int_{0}^{t}\,
e^{-\lambda r}\,K_\lambda\big(\,S^{\mc{E}_{N}}(r)-r\,\big)\,dr\,\Big]\;.
\end{equation*}
As $r\mapsto S^{\mc{E}_{N}}(r)-r$ and $K_\lambda$ are increasing
maps, this expectation is bounded by
\begin{equation*}
\frac{1}{\lambda}\,\mb{E}_{\eta}^{N}\,
\big[\,K_\lambda\big(\,S^{\mc{E}_{N}}(t)-t\,\big)\,\big]\;.
\end{equation*}
At this point, the second assertion of the lemma follows from the
first one.
\end{proof}

The next result establishes the uniqueness of limit points of the sequence
$\bb Q_{\eta^{N}}^{N}$.

\begin{prop}
\label{p-uniq}
Assume that condition $\mf{R}_{\ms{L}}$ is in force. Fix $x\in S$ and
a sequence $(\eta^{N})_{N\in\bb{N}}$ such that
$\eta^{N}\in\mc{E}_{N}^{x}$ for all $N\in\bb{N}$.  Let $\bb Q^{*}$ be
a limit point of the sequence $\bb Q_{\eta^{N}}^{N}$ which satisfies
\eqref{1e_leftcont}. Then, $\bb Q^{*}=\bb Q_{x}^{\ms L}$.
\end{prop}

\begin{proof}
Fix $\lambda>0$, a function $f:S\rightarrow\bb{R}$, and let
$g = (\lambda - \ms L) f$. Denote by ${F}_{N}$ the solution of
\eqref{1f01}. Under the measure $\mb{P}_{\eta^{N}}^{N}$, the process
$M_{N}(t)$ given by
\begin{equation*}
M_{N}(t)\;=\;e^{-\lambda t}\,{F}_{N}(\xi_{N}(t))\,
-\,{F}_{N}(\xi_{N}(0))\,+\,\int_{0}^{t}e^{-\lambda r}
\big[\,(\,\lambda\,-\,\,{\ms{L}}_{N}\,)
\,{F}_{N}\,\big](\xi_{N}(r))\,dr
\end{equation*}
is a martingale with respect to the filtration
$\{\ms{F}_{t}\}_{t\ge0}$ defined above \eqref{1e_filt}. By
\eqref{1f01}, we may replace $(\,\lambda\,-\,{\ms{L}}_{N}\,)\,{F}_{N}$
by $G_{N}$. Thus, since $G_{N}$ vanishes on $\Delta_{N}$,
\begin{align*}
M_{N}(t)\;=\,\; & e^{-\lambda t}\,{F}_{N}(\xi_{N}(t))\,
-\,{F}_{N}(\xi_{N}(0))\,\\
&\qquad +\,\int_{0}^{t}\,e^{-\lambda r}\,G_{N}(\xi_{N}(r))\,
\chi_{\mc{E}_{N}}(\xi_{N}(r))\,dr\;.
\end{align*}

Recall the definition of the filtration
$\{\ms{G}_{t}^{N}\}_{t\ge0}$ from \eqref{1e_filt}.  Since $S^{\mc{E}_{N}}(t)$ is a stopping
time with respect to $\ms{F}_{t}$, the process
$\widehat{M}_{N}(t)=M(S^{\mc{E}_{N}}(t))$ is a martingale with respect
to the filtration $\{\ms{G}_{t}^{N}\}_{t\ge0}$:
\begin{equation*}
\begin{aligned}\widehat{M}_{N}(t)\;=\;\, &
e^{-\lambda S^{\mc{E}_{N}}(t)}\,{F}_{N}(\xi_{N}^{\mc{E}_{N}}(t))\,
-\,{F}_{N}(\xi_{N}^{\mc{E}_{N}}(0))\\
& \qquad +\,\int_{0}^{S^{\mc{E}_{N}}(t)}\,e^{-\lambda r}\,
G_{N}(\xi_{N}(r))\,\chi_{\mc{E}_{N}}(\xi_{N}(r))\,dr\;.
\end{aligned}
\end{equation*}

The presence of the indicator of the set $\mc{E}_{N}$ in the
integral permits to perform the change of variables $r'=T^{\mc{E}_{N}}(r)$.
Hence, by \eqref{11},
\begin{equation*}
\widehat{M}_{N}(t)\;=\;e^{-\lambda S^{\mc{E}_{N}}(t)}\,
{F}_{N}(\xi_{N}^{\mc{E}_{N}}(t))\,-\,{F}_{N}(\,\xi_{N}^{\mc{E}_{N}}(0)\,)\,
+\,\,\int_{0}^{t}\,e^{-\lambda S^{\mc{E}_{N}}(r')}\,\,
G_{N}(\xi^{\mc{E}_{N}}(r'))\,dr'\;.
\end{equation*}

By definitions of $G_{N}$, $Y_{N}(\cdot)$, by condition
$\mf R_{\ms L}$ and by Lemmata \ref{nl1} and \ref{l01},
\begin{equation*}
\widehat{M}_{N}(t)\;=\;e^{-\lambda t}\,f(Y_{N}(t))\,
-\,f(Y_{N}(0))\,-\,\int_{0}^{t}\,e^{-\lambda r'}\,g(Y_{N}(r'))\,dr'\,
+\,R_{N}(t)\;,
\end{equation*}
where, for all $t>0$,
\begin{equation}
\lim_{N\to\infty}\sup_{\eta\in\mc{E}_{N}^{x}}\mb{E}_{\eta}^{N}\,
\big[\,R_{N}(t)\,\big]\;=\;0\;.
\label{1e_82}
\end{equation}

Fix $0\le s<t$, $p\ge1$, $0\le s_{1}<s_{2}<\cdots<s_{p}\le s$ and a
bounded measurable function $h:S^{p}\to\bb{R}$. Let
\begin{gather*}
\mf{M}_{f}^{s,\,t}(Y(\cdot))\;:=\;e^{-\lambda t}f(Y(t))\,
-\,e^{-\lambda s}f(Y(s))\,+\,\int_{s}^{t}e^{-\lambda r}
\,[\,(\lambda\,-\,\ms{L}_{Y}\,)f\,](Y(r))\,dr\;,\\
\mf{H}(Y(\cdot))\;:=\;h(\,Y(s_{1}),\,\dots,\,Y(s_{p})\,)\;,
\end{gather*}
and let $\bb Q^{*}$ be a limit point of the sequence
$\bb Q_{\eta^{N}}^{N}$ satisfying the hypothesis of the
proposition. As $\widehat{M}_{N}(t)$ is a martingale and
$\eta^{N}\in\mc{E}_{N}^{x}$, by \eqref{1e_82},
\begin{equation*}
E_{\bb Q^{*}}\,\big[\,\mf{M}_{f}^{s,\,t}(Y(\cdot))\;
\mf{H}(Y(\cdot))\,\big]\;=\;\lim_{N\rightarrow\infty}
\mb{E}_{\eta^{N}}^{N}\Big[\,\mf{M}_{f}^{s,\,t}(Y_{N}(\cdot))\;
\mf{H}(Y_{N}(\cdot))\,\Big]\;=\;0\;.
\end{equation*}
To complete the proof, it remains to appeal to the uniqueness of
solutions of martingale problems in finite state spaces.
\end{proof}

We are now in a position to prove that condition $\mf{R}_{\ms{L}}$
entails $\mf{C}_{\ms{L}}$ and $\mf{D}$.

\smallskip\noindent{\bf Proof:}  The statement follows from Lemma
\ref{nl1} and Propositions \ref{1t_tight} and
\ref{p-uniq}. \qed

\subsection{Conditions $\mf{C}_{\ms{L}}$ and $\mf{D}$ imply
$\mf{R}_{\ms{L}}$}

Recall equation \eqref{1f11} for $F_N$. Since $G_N$ vanishes on
$\Delta_N$, we may rewrite this identity as
\begin{equation*}
F_N(\eta) \;=\;
\mb E^N_{\eta} \Big[ \int_0^{\infty} e^{-\lambda t}
\, G_N(\xi(t)) \, \chi_{\mc E_N} (\xi(t))\; dt \,\Big] \,.
\end{equation*}
As the chain $\xi_N(t)$ is irreducible,
$\lim_{t\to\infty} T^{\mc E}(t) = \infty$.  Hence, by the change of
variables $t' = T^{\mc E}(t)$,
\begin{equation*}
F_N(\eta) \;=\;
\mb E^N_{\eta} \Big[ \int_0^{\infty} e^{-\lambda S^{\mc E}(t)}
\, G_N(\xi^{\mc E}(t)) \; dt \,\Big]
\;=\;
\mb E^N_{\eta} \Big[ \int_0^{\infty} e^{-\lambda S^{\mc E}(t)}
\, g(Y_N(t)) \; dt \,\Big]
\end{equation*}
because $G_N(\xi^{\mc E}(t)) = g(Y_N(t))$. Therefore,
\begin{equation*}
F_N(\eta) \;=\;
\mb E^N_{\eta} \Big[ \int_0^{\infty} e^{-\lambda t}
\, g(Y_N(t)) \; dt \,\Big]
\;+\; R^{(1)}_N(\eta) \,,
\end{equation*}
where the absolute value of the remainder $R^{(1)}_N(\eta) $ is
bounded by
\begin{equation*}
\Vert g\Vert_\infty \,
\mb E^N_{\eta} \Big[ \int_0^{\infty}
\big\{ e^{-\lambda t} \,-\, e^{-\lambda S^{\mc E}(t)}\,\big\}
\; dt \,\Big]
\end{equation*}
because $S^{\mc E}(t) \ge t$ for all $t\ge 0$.

By Condition $\mf C_{\ms L}$, for all $x\in S$,
\begin{equation*}
\lim_{N\to\infty}
\sup_{\eta\in\mc{E}_{N}^{x}}\,\Big|\,
\mb E^N_{\eta} \Big[ \int_0^{\infty} e^{-\lambda t}
\, g(Y_N(t)) \; dt \,\Big] \,-\,
\bb Q^{\ms L}_x \Big[ \int_0^{\infty} e^{-\lambda t}
\, g(Y(t)) \; dt \,\Big] \,\Big|
\,\;=\;0\;.
\end{equation*}
Note that the convergence is uniform in $\ms E^x_N$ because we may
consider a subsequence $\eta^N \in \ms E^x_N$ of initial conditions
which attains the maximum and apply condition $\mf C_{\ms L}$ to this
sequence. By \eqref{1f11}, the second term in the previous formula is
$f(x)$, where $f$ is the solution of \eqref{nf06}.

To complete the proof of the theorem, it remains to show that the
remainder $R^{(1)}_N(\eta) $ converges uniformly to $0$. This is a
consequence of the second assertion of Lemma \ref{l01}.

\section{Potential theory}
\label{sec8}

We review below some results on potential theory used in the next
three sections. The notation is the one introduced in Section
\ref{sec11}.  Recall that we represent by
$\color{blue} R_{N}:\mc{H}_{N}\times\mc{H}_{N}\rightarrow
[0,\,\infty)$ the jump rates of the process $\xi_{N}(\cdot)$, and by
$\color{blue} \lambda_{N} (\eta) = \sum_{\zeta\not =\eta} R_{N}
(\eta, \zeta)$ the holding times. We adopt the convention that the
jump rates vanish on the diagonal: $R_{N}(\eta,\,\eta) =0$ for all
$\eta\in\mc{H}_{N}$.  Denote the jump probabilities by
$\color{blue} p_{N} (\eta, \zeta) = R_{N} (\eta, \zeta)/\lambda_{N}
(\eta)$.

We represent by $\<\,\cdot\,,\,\cdot\,\>_{\mu_N}$ the scalar product
in $L^2(\mu_N)$: for $F$, $G:\mc H_N \to\bb R$,
\begin{equation*}
\<\, F\,,\, G\,\>_{\mu_N} \;=\;
\sum_{\eta\in\mc{H}_{N}}F(\eta)\, G(\eta)\,
\mu_{N}(\eta)\;.
\end{equation*}

Denote by $\color{blue} \ms L^\dagger_N$ the adjoint of the generator
$\ms L_N$ in $L^2(\mu_N)$. It is well known that $\ms L^\dagger_N$ is
the generator of a $\mc H_N$-valued, continuous-time Markov chain,
represented by $\color{blue} \xi^\dagger_N(\cdot)$. The jump rates,
holding times and jump probabilities of this process are denoted by
$\color{blue} R^\dagger_{N} (\eta, \zeta)$,
$\color{blue} \lambda^\dagger_{N} (\eta)$ and
$\color{blue} p^\dagger_{N} (\eta, \zeta)$, respectively. For a
probability measure $\nu$ on $\mc H_N$, we denote by
$\mb P^{\dagger, N}_\nu$ the measure on $D(\bb R_+, \mc H_N)$ induced
by $\xi^\dagger_N(\cdot)$ starting from $\nu$. Expectation with
respect to $\color{blue} \mb P^{\dagger, N}_\nu$ is represented by
$\color{blue} \mb E^{\dagger, N}_\nu$.

Fix two disjoint and non-empty subsets $\mc{A}$ and $\mc{B}$ of
$\mc{H}_{N}$. The equilibrium potential between $\mc{A}$ and $\mc{B}$
with respect to the process $\xi_{N}(\cdot)$ has been introduced in
\eqref{03}. The one for the adjoint process $\xi^\dagger_{N}(\cdot)$
is denoted by
$h^\dagger_{\mc{A},\,\mc{B}}:\mc{H}_{N}\rightarrow[0,\,1]$ and is
given by
\begin{equation}
\label{03b}
h^\dagger_{\mc{A},\,\mc{B}} (\eta) \,=\, \mb{P}^{\dagger, N}_{\eta}
[\, \tau_{\mc{A}}<\tau_{\mc{B}}\, ] \;,\quad
\eta\in\mc{H}_{N}\;.
\end{equation}

Recall from \eqref{07b} the definition of the mean-jump rates between
$\mc{E}_{N}^{x}$ and $\mc{E}_{N}^{y}$ for the process
$\xi_N(\cdot)$. The ones for the adjoint process
$\xi^\dagger_N(\cdot)$, represented by
$\color{blue} r^\dagger_N(x,y)$, are defined analogously. Since the
holding times of the adjoint process coincide with the original ones,
$\lambda^\dagger_N(\eta) = \lambda_N(\eta)$, $r^\dagger_N(x,y)$ is
equal to the right-hand side of \eqref{07b} with $\mb P^N_{\eta}$
replaced by $\mb P^{\dagger, N}_{\eta}$.

The first result of this section establishes an elementary identity
between mean jump rates of the process and its adjoint. Recall that
$\breve{\mc{E}}_N^x$ has been introduced in \eqref{breveE}.

\begin{lemma}
\label{l07b}
For all $x \not = y\in S$,
\begin{equation*}
\mu_N(\mc E^x_N) \, r^\dagger_N(x,y) \;=\; \mu_N(\mc E^y_N) \,
r_N(y,x)\;.
\end{equation*}
and
\begin{equation*}
\sum_{z\not = x} r^\dagger_N(x,z) \;=\; \sum_{z\not = x} r_N(x,z)
\;=\; \frac{1}{\mu_N(\mc E^x_N)}\,
\Cap_N(\mc{E}_{N}^{x} \,,\, \breve{\mc{E}}_{N}^{x})\;.
\end{equation*}
\end{lemma}

\begin{proof}
By the definition \eqref{07b} of the jump rates $r_N(x,y)$,
\begin{equation*}
\mu_N(\mc E^x_N)\, r_N(x,y) \;=\;
\sum_{\eta\in \mc E^x_N} \mu_N(\eta)  \, \lambda_N(\eta)\,
\mb P^N_{\eta} \big[\,
\tau_{\mc{E}_{N}^{y}} \,<\, \tau^+_{\breve{\mc{E}}_{N}^{y}} \,\big]\;.
\end{equation*}
Let $\color{blue} M_N(\eta) = \mu_N(\eta) \, \lambda_N(\eta)$. This
measure is invariant for the embedded, discrete-time Markov
chain. With this notation, the right-hand side can be written as
\begin{equation*}
\sum_{\eta\in \mc E^x_N} \sum_{\zeta\in \mc E^y_N}
\, M_N(\eta)\, \mb P^N_{\eta} \big[\,
\tau_{\zeta} \,=\, \tau^+_{\breve{\mc{E}}_{N}} \,\big]\;.
\end{equation*}
Reversing the trajectory, this sum is seen to be equal to

\begin{equation*}
\sum_{\eta\in \mc E^x_N} \sum_{\zeta\in \mc E^y_N}
\, M_N(\zeta)\, \mb P^{\dagger, N}_{\zeta} \big[\,
\tau_{\eta} \,=\, \tau^+_{\breve{\mc{E}}_{N}} \,\big]
\;=\; \mu_N(\mc E^y_N)\, r^\dagger_N(y,x)
\;,
\end{equation*}
which proves the first assertion of the lemma.

To prove the second one, note that
\begin{equation*}
\mu_N(\mc E^x_N)\, \sum_{z\not =x} r_N(x,z) \;=\;
\sum_{\eta\in \mc E^x_N} M_N(\eta)
\mb P^N_{\eta} \big[\, \tau_{\breve{\mc{E}}_{N}^{x}}
\,<\, \tau^+_{ \mc{E}_{N}^{x}} \,\big] \;=\;
\Cap_N(\breve{\mc{E}}_{N}^{x} \,,\,  \mc{E}_{N}^{x}) \;.
\end{equation*}
By equation (2.4) and Lemma 2.3 in \cite{GL},
$\Cap_N(\breve{\mc{E}}_{N}^{x} \,,\, \mc{E}_{N}^{x}) =
\Cap^\dagger_N(\breve{\mc{E}}_{N}^{x} \,,\, \mc{E}_{N}^{x})$, where
this later expression represents the capacity with respect to the
adjoint process. To conclude the proof, it remains to rewrite the same
two identities for the adjoint process.
\end{proof}

\smallskip\noindent{\bf Conditions (H0), (H1)}.  Recall from Section
\ref{sec11} the statement of these conditions.  We present below some
consequences of them. The next result is \cite[Proposition 5.10]{BL1},
which essentially asserts that the process hits every configuration
inside a metastable set before arriving at another metastable set.

\begin{lemma}
\label{l16}
Assume that condition {\rm (H1)} is in force.  Fix $x\in S$ and a
sequence $(\zeta^N:N\ge 1)$ such that $\zeta^N\in\mc E^x_N$ for all
$N\ge 1$. Then,
\begin{equation*}
\limsup_{N\rightarrow\infty} \max_{\eta\in\mc E^x_N}
\mb P_{\eta} \big[ \, \tau_{\zeta^N}
\,>\, \tau_{\breve{\ms E}^x_N}\, \big]\;=\;0\;.
\end{equation*}
\end{lemma}

The next result asserts that, starting from a well $\mc E^x_N$, the
process $\xi_N(\cdot)$ visits any point in $\mc E^x_N$ quickly.

\begin{lemma}
\label{l05}
Assume that conditions {\rm (H0)} and {\rm (H1)} are in force.  Fix
$x\in S$, and let $(\zeta^N: N\ge 1)$ be a sequence of configurations
such that $\zeta^N \in \mc E^x_N$ for all $N\ge 1$.  Then, for all $\delta>0$,
\begin{equation*}
\limsup_{N\rightarrow\infty}
\max_{\eta\in\mc{E}_{N}^{x}}\,
\mb P_\eta \big[ \, \tau_{\zeta^N} \,>\, \delta\, \big]\;=\;0\;.
\end{equation*}
\end{lemma}

\begin{proof}
Fix a sequence $(\eta^N:N\ge 1)$ such that $\eta^N\in\mc E^x_N$ for
all $N\ge 1$.  By \cite[Theorem 2.1]{BL2}, the process $Y_N(\cdot)$
converges to $Y(\cdot)$. The assertion of the lemma follows from this
fact, Lemma \ref{l16} and \cite[Lemma 3.1]{BL1}.
\end{proof}

In the reversible case, the mean jump rate $r_N(\cdot,\,\cdot)$
can be expressed in terms of capacities: By
\cite[Lemma 6.8]{BL1}, $r_{N}(x,\,y)$ is equal to
\begin{equation}
\label{mjrc}
\frac{1}{\mu_{N}(\mc{E}_{N}^{x})}\big[\,
\textup{cap}_{N}(\mc{E}_{N}^{x},\,\breve{\mc{E}}_{N}^{x})\,+\,
\textup{cap}_{N}(\mc{E}_{N}^{y},\,\breve{\mc{E}}_{N}^{y})\,-\,
\textup{cap}_{N}(\mc{E}_{N}^{x}\cup\mc{E}_{N}^{x},\,
\mc{E}_{N}(S\setminus\{x,\,y\}))\,\big]\;,
\end{equation}
in which $\mc{E}_{N}(S\setminus\{x,\,y\}) = \cup_{z \in S\setminus\{x,\,y\}} \mc{E}_{N}^{z}$. Hence, estimating the mean-jump rates boils down to that of the
capacity between metastable wells, which can be achieved by using the
variational characterizations of capacities, known as the Dirichlet
and the Thomson principles \cite{l-review}.  In the non-reversible
case, a robust strategy of estimating mean-jump rates via capacities
between wells has also been developed in \cite{BL2, Lan, LS1}.

We complete this section with a formula for the average of equilibrium
potentials. Fix two disjoint and non-empty subsets $\mc{A}$ and
$\mc{B}$ of $\mc{H}_{N}$. According to \cite[Proposition A.2]{BL2},
\begin{equation}
\label{34}
\sum_{\eta\not\in \mc A \cup \mc B} \mu_N(\eta)  \,
h^\dagger_{\mc A, \mc B} (\eta)
\;=\; \Cap_N(\mc A, \mc B)\,\,
\mb E^{N}_{\nu^\dagger_{\mc A, \mc B}}
\Big[ \, \int_0^{\tau_{\mc B}}
\chi_{[\mc A \cup \mc B]^c} (\xi_N(s))\, ds\,\Big]\;,
\end{equation}
where $\nu^\dagger_{\mc A, \mc B}$ is the equilibrium measure between
$\mc A$ and $\mc B$:
\begin{equation}
\label{38}
\nu^\dagger_{\mc A, \mc B} (\zeta) \;
=\; \frac{1}{\Cap_N(\mc A, \mc B)}
\, \mu_N(\zeta) \, \lambda_N(\zeta)\,
\mb P^{\dagger, N}_\zeta\big[\, \tau_{\mc B} <
\tau^+_{\mc A}\,\big]\;,
\quad \zeta\in \mc A\;.
\end{equation}

\section{The solutions of the resolvent equation}
\label{s5}

Theorem \ref{1t_main2} asserts that a sequence of Markov processes is
metastable if conditions $\mf R_{\ms L}$ and $\mf D$ are fulfilled. In
this section and in the next, we present sufficient conditions for
$\mf R_{\ms L}$ and $\mf D$ to hold.  We start by dividing the
condition $\mf R_{\ms L}$ into two sub-conditions, namely, conditions
$\mf R^{(1)}$ and $\mf R_{\ms L}^{(2)}$.

In this section, we present two mixing properties, assumptions
$\mf{V}$ and $\mf{M}$, which imply condition $\mf R^{(1)}$. As a
by-product, we show that condition $\mf{M}$ implies condition $\mf{D}$
if $\mu_N(\Delta_N)/\mu_N(\ms E^x_N) \to 0$ for all $x\in S$.  We
leave condition $\mf R_{\ms L}^{(2)}$ to the next section.

\medskip\noindent{\bf Condition $\mf{R}^{(1)}$}.  The solution $F_N$
of the resolvent equation \eqref{1f01} is asymptotically constant on
each well $\mc E^x_N$ in the sense that
\begin{equation*}
\lim_{N\rightarrow\infty} \, \max_{x\in S}\,
\max_{\eta, \zeta \in\mc{E}_{N}^{x}}
\,|\,{F}_{N}(\eta)\,-\, F_{N}(\zeta)\,|\;=\;0\;.
\end{equation*}

\begin{remark}
Clearly, the condition $\mf{R}^{(1)}$ is satisfied if the wells
$\mc E^x_N$ are singletons as in the Ising model under Glauber
dynamics \cite{BM} or the simple inclusion process \cite{BDG, KS}.
\end{remark}

\subsection{Visiting condition $\mf{V}$}\label{secV}

The first condition is build upon the existence in each well
of a configuration which is visited in a time-scale much shorter than
the metastable one.

\medskip\noindent{\bf Condition $\mf{V}$}.  There exist configurations
$\zeta^x_N \in \mc E^x_N$, $x\in S$, such that
\begin{equation}
\label{30}
\lim_{N\to\infty}
\max_{\eta\in\mc{E}_{N}^{y}} \mb{P}^N_{\eta}
\,[\, \tau_{\zeta^y_N} \ge s \,]\;=\; 0
\end{equation}
for all $s>0$, $y\in S$.
\medskip

The next result asserts that this property is a sufficient condition for
$\mf R^{(1)}$ to hold.  Its proof is postponed to the end of the
subsection.

\begin{proposition}
\label{p01}
Condition $\mf{V}$ implies condition $\mf R^{(1)}$.
\end{proposition}

\begin{remark}
\label{rm8}
Condition \eqref{30} requires the process to visit the bottom of the
well quickly. It is weaker than (H1), which implies that the process
visits all configurations in a well before jumping to a new
one. Actually, Proposition \ref{p_mixmain} asserts that a stronger
version of condition \eqref{30} holds for reversible, critical
zero-range processes, a model which does not satisfy condition (H1).
\end{remark}

\begin{corollary}
\label{l04}
Assume that condition {\rm (H0)}, {\rm (H1)} are in force.  Then,
$\mf R^{(1)}$ holds.
\end{corollary}

\begin{proof}
By Lemma \ref{l05}, condition \eqref{30} holds under the assumptions
(H0) and (H1). The assertion of the corollary follows, therefore, from
Proposition \ref{p01}.
\end{proof}

\begin{remark}
\label{rm9}
Condition (H0) and (H1) have been derived for super-critical
condensing zero-range processes in \cite{BL3, Lan, Seo} and for many
other dynamics. These results support the introduction of condition
\eqref{30}.
\end{remark}

We turn to the proof of Proposition \ref{p01}.  We start showing that
we may mollify the solution with the semigroup
$\color{blue} (\ms P_N(t) : t\ge 0)$ associated to the generator
$\ms L_N$.

\begin{lemma}
\label{l08b}
For all $T>0$,
\begin{equation*}
\sup_{0\le t\le T} \, \max_{\eta\in \mc{H}_N} \, \big|\,
F_N(\eta) \;-\; (\ms P_N(t)\, F_N) (\eta)\,\big|\;\le\;
2\, T\, \Vert G_N\Vert_\infty \;.
\end{equation*}
\end{lemma}

\begin{proof}
Fix $T>0$ and $0<t\le T$.  By the representation \eqref{1f11} of $F_N$,
\begin{equation*}
(\ms P_N(t) \, F_N) (\eta) \;=\;
\mb E^N_\eta \Big[ \int_0^\infty e^{-\lambda s}
\, G_N(\xi_N(s+t)) \; ds \,\Big]\;.
\end{equation*}
By a change of variables, the right-hand side can be rewritten as
\begin{equation*}
\mb E^N_\eta \Big[ \int_t^\infty e^{-\lambda s}
\, G_N(\xi_N(s)) \; ds \,\Big] \;+\;
\mb E^N_\eta \Big[ \int_t^\infty \Big\{\, e^{-\lambda (s-t)} \,-\,
e^{-\lambda s} \,\Big\}
\, G_N(\xi_N(s)) \; ds \,\Big]\;.
\end{equation*}
The first term is equal to $F_N(\eta) + R_N$, where the absolute value
of the remainder $R_N$ is bounded by $t\, \Vert G_N\Vert_\infty$. As
$1-e^{-a} \le a$, $a\ge 0$, the second term is bounded by $t\, \Vert
G_N\Vert_\infty$.
\end{proof}

\begin{proof}[Proof of Proposition \ref{p01}]
Fix $x\in S$, $\eta\in\mc E^x_N$, $s>0$, and write
$(\ms P_{N}(s) F_N)(\eta)$ as
\begin{equation*}
\mb E^N_\eta \big[\, F_N(\xi_N(s)) \,,\,
\tau_{\zeta_{N}^{x}} \,\le \, s \,\big] \;+\; R_N^{(1)}\;,
\end{equation*}
where the remainder $R_N^{(1)}$ is bounded by
$\max_{\zeta\in \mc E^x_N} \mb P^N_\zeta [\, \tau_{\zeta_{N}^{x}} \,> \,
s \,]\, \Vert F_N\Vert_\infty$. By the strong Markov property, the
previous expression is equal to
\begin{equation*}
\mb E^N_\eta \Big[\, \big[\,
\ms P_N(s - \tau_{\zeta_{N}^{x}}) \, F_N\, \big] (\zeta^x_N) \,,\,
\tau_{\zeta_{N}^{x}} \,\le \, s \,\Big] \;+\; R_N^{(1)}\;.
\end{equation*}
By Lemma \ref{l08b}, this expression is equal to $F_N(\xi_{N}^{x}) +
R^{(2)}_N$, where
\begin{equation*}
\big|\, R^{(2)}_N \,\big| \;\le\;
2\, s \, \Vert G_N\Vert_\infty \;+\;
2\, \max_{\zeta\in \mc E^x_N} \mb P^N_\zeta [\, \tau_{\zeta_{N}^{x}} \,> \, s
\,]\, \Vert F_N\Vert_\infty\;.
\end{equation*}
Hence, by Lemma \ref{l08b} once more,
\begin{equation*}
\max_{\eta\in \mc E^x_N} \, \big|\,
F_N(\eta) \;-\; F_N (\zeta^x_N)\,\big|\;\le\;
4\, s \, \Vert G_N\Vert_\infty \;+\;
2\, \max_{\zeta\in \mc E^x_N}
\mb P^N_\zeta [\, \tau_{\zeta_{N}^{x}} \,> \, s\,]\, \Vert F_N\Vert_\infty\;.
\end{equation*}
By \eqref{1f10}, the sequence $F_N$ is uniformly bounded. The same
property holds for the sequence $G_N$ by definition. To complete the
proof of the assertion, it remains to let $N\to\infty$ and then
$s \to 0$ and to recall the hypothesis \eqref{30}.
\end{proof}

\subsection{Mixing condition $\mf{M}$}\label{sec61}

The second set of assumptions requires the mixing time of the
reflected process on a well to be much smaller than the hitting time
of the boundary.

Denote by $\color{blue} \mc V_{N}^{x}$, $x\in S$, a set of large
wells which contains the wells $\mc E^x_N$:
$\mc E^x_N \subset \mc V_{N}^{x}$. Let
$\color{blue}(\xi^{R,x}_N(t) : t \ge 0)$ be the continuous-time Markov
chain on $\mc{V}_{N}^{x}$ obtained by reflecting the process
$\xi_N(\cdot)$ at the boundary of this set. In other words, in the
discrete setting, the process $\xi^{R,x}_N(\cdot)$ behaves as the
original process inside the well $\mc{V}_{N}^{x}$, but its jumps to
the set $(\mc{V}_{N}^{x})^{c}$ are suppressed.

Denote by $d^x_{\rm TV}(\mu, \nu)=d^{x,N}_{\rm TV} (\mu, \nu) $ the total variation distance
between two probability measures $\mu$, $\nu$ on $\mc V^x_N$:
\begin{equation}\label{dtv}
d^x_{\rm TV} (\mu, \nu) \;=\; \frac{1}{2}\, \sup_J
\Big|\, \int J(\eta)\, \mu(d\eta) \,-\,
\int J(\eta)\, \nu(d\eta) \, \Big|\;,
\end{equation}
where the supremum is carried over all measurable functions
$J: \mc V^x_N \to \bb R$ bounded by $1$, $\sup_{\xi\in \mc V^x_N} |\,
J(\xi)\,|\le 1$.

Assume that the reflected process $\xi^{R,x}_N(\cdot)$ is ergodic.
Denote by \textcolor{blue}{$(\ms{P}^{R,x}_N (t) : t\ge 0)$} its
semigroup, by $\color{blue} \pi^{R,x} = \pi^{R,x}_N$ its stationary state, and by
$\color{blue} t^{x}_{\rm mix} (\epsilon) = t^{x}_{{\rm mix}, N} (\epsilon)$, $0<\epsilon<1$, its mixing
time:
\begin{equation*}
t^{x}_{\rm mix} (\epsilon) \;=\;
\inf\big\{t> 0 : \sup_{\eta\in  \mc E^x_N} d^x_{\rm TV} (\delta_\eta \ms
P^{R,x}_N(t) \,,\,  \pi^{R,x}) \,\le\, \epsilon \,\big\}\;.
\end{equation*}

The next result asserts that the following mixing properties entail condition
$\mf R^{(1)}$.

\smallskip
\noindent{\bf Condition $\mf M$.} The process $\xi_N(\cdot)$ starting
from a well $\mc{E}_{N}^{x}$ cannot escape from the well
$\mc{V}_{N}^{x}$ within a time scale $\mb h_N \ll 1$: For all
$x\in S$,
\begin{equation}
\label{23}
\lim_{N\to\infty} \sup_{\eta\in \mc{E}_{N}^{x}}
\mb{P}^N_{\eta}\,[\, \tau_{(\mc{V}_{N}^{x})^{c}}
\le \mb h_{N} \,]\;=\;0\;.
\end{equation}
Furthermore, for every $x\in S$, the reflected process
$\xi^{R,x}_N(\cdot)$ is ergodic, and for all $\epsilon>0$,
\begin{equation}
\label{33}
t^{x}_{\rm mix} (\epsilon) \;\le\; \mb h_N
\end{equation}
for all $N$ sufficiently large.
\begin{proposition}
\label{p03}
If the mixing property $\mf M$ is satisfied, then condition
$\mf R^{(1)}$ holds.
\end{proposition}

\begin{remark}
\label{rm10b}
Barrera and Jara \cite{BJ20} proved that the mixing time of small
random perturbations of dynamical systems satisfying certain
regularity assumptions, is of polynomial order. Since the hitting time
of the boundary is exponentially large \cite{fw98}, the previous
result applies to this setting.
\end{remark}

The proof of Proposition \ref{p03} relies on a simple estimate between
the semigroup of the original process and the semigroup of the
reflected one.

\begin{lemma}
\label{l09b}
For each $x\in S$, $\eta\in \mc E^x_N$ and $t>0$,
\begin{equation*}
\big|\, (\ms P_N(t)\, F_N) (\eta) \,-\,
(\ms P^{R,x}_N(t)\, F_N) (\eta) \,\big|\;\le\;
2\, \Vert F_N\Vert_\infty \,
\mb P^N_\eta [\, \tau_{(\mc{V}_{N}^{x})^{c}} \,\le \, t \,]\;.
\end{equation*}
\end{lemma}

\begin{proof}
Fix $x$ in $S$, $\eta$ in $\mc E^x_N$, and write
$(\ms P_N(t)\, F_N) (\eta)$ as
\begin{equation*}
\mb E^N_\eta \big[\, F_N(\xi_N(t)) \,,\,
\tau_{(\mc{V}_{N}^{x})^{c}} \,>\, t \,\big] \;+\;
\mb E^N_\eta \big[\, F_N(\xi_N(t)) \,,\,
\tau_{(\mc{V}_{N}^{x})^{c}} \,\le \, t \,\big] \;.
\end{equation*}
In the first term, we may replace the process $\xi_N(\cdot)$ by the
reflected one since the process remained in the set $\mc{V}_{N}^{x}$
in the time-interval $[0,t]$.  The second term is bounded by
$\mb P^N_\eta [\, \tau_{(\mc{V}_{N}^{x})^{c}} \,\le \, t \,]\, \Vert
F_N\Vert_\infty$. Writing the indicator function of the set
$\{\tau_{(\mc{V}_{N}^{x})^{c}} \,>\, t \}$ as $1$ minus the indicator of
the complement, we conclude the proof of the lemma.
\end{proof}

\begin{proof}[Proof of Proposition \ref{p03}]
By Lemmata \ref{l08b} and \ref{l09b} with $T=t=\mb h_N$, and
hypothesis \eqref{23}
\begin{equation*}
\lim_{N\to\infty} \, \sup_{\eta\in\mc E^x_N}
\,\big|\,{F}_{N}(\eta)\,
-\, (\ms P^{R,x}_N(\mb h_N)\, F_N) (\eta) \,\big|
\;=\; 0\;.
\end{equation*}

Fix $x\in S$, $\eta\in \mc E^x_N$ and $\epsilon>0$. By definition of
the total variation distance,
\begin{equation}\label{mix1}
\,\big|\, (\ms P^{R,x}_N(\mb h_N)\, F_N) (\eta)
\,-\, E_{\pi^{R,x}} \big[\, F_N \,\big]\,   \,\big|
\;\le\; 2\, \Vert\, F_N\, \Vert_\infty\,
d_{\rm TV}^x \big(\, \delta_\eta\, \ms P^{R,x}_N(\mb h_N)
\,,\, \pi^{R,x})\;.
\end{equation}
 By the contracting property of the semigroup,
$$
 d_{\rm TV}^x \big(\, \delta_\eta\, \ms P^{R,x}_N(t)
\,,\, \pi^{R,x}) \;=\; d_{\rm TV}^x \big(\, \delta_\eta\, \ms P^{R,x}_N(t)
\,,\,   \pi^{R,x} \,\ms P^{R,x}_N(t))
$$
is decreasing in $t$, and thus by  \eqref{33}, the right-hand side of \eqref{mix1} is bounded from above by
\begin{equation*}
2\, \Vert\, F_N\, \Vert_\infty\, \sup_{\xi \in \mc{E}_N^x}
d_{\rm TV}^x\big(\, \delta_\xi \, \ms P^{R,x}_N(t^{x}_{\rm mix}
(\epsilon)) \,,\, \pi^{R,x}) = 2 \Vert\, F_N\, \Vert_\infty\, \epsilon
\end{equation*}
by definition of the mixing time. This completes the proof of
the proposition because the sequence $(F_N)$ is uniformly bounded in $N$.

\end{proof}


\subsection{Local equilibration and condition $\mf{D}$}\label{secD}

The same argument shows that condition $\mf M$ yields a fast local
equilibration inside each well. In particular, condition $\mf{D}$
results from assumption $\mf M$ and the property that
$\mu_N(\Delta_N)/\mu_N(\ms E^x_N) \to 0$ for all $x\in S$.

Consider a uniformly bounded sequence of functions
$(Q_{N})_{N\in\bb{N}}$, $Q_{N}:\mc{H}_{N}\rightarrow\bb{R}$: There
exists a finite constant $M>0$ such that
\begin{equation}
\sup_{\eta\in\mc{H}_{N}}|\,Q_{N}(\eta)\,|\;\le\;
M\text{\;\;\;\;\;for all }N\in\bb{N}\;.
\label{e_Qbd-1}
\end{equation}
Recall the definition of the probability measure $\mu^x_N$, introduced
in Remark \ref{rm11}.

\begin{prop}
\label{prop_eqbl}
Assume that condition $\mf M$ is in force. Then, for all $x\in S$,
$T>0$,
\begin{equation*}
\sup_{\eta\in\mc{E}_{N}^{x}}\Big|\,
\mb{E}_{\eta}^{N}\,\Big[\,\int_{0}^{T}Q_{N}(\xi_{N}(t))dt\,\Big]\,
-\,\mb{E}_{\mu_{N}^{x}}^{N}\,\Big[\,\int_{0}^{T}
Q_{N}(\xi_{N}(t))dt\,\Big]\,\Big|\;\le \; 6M(T+1)\,o_{N}(1)\;,
\end{equation*}
where the error term $o_{N}(1)$ at the right-hand side is uniform
over $N$, $M$ and $T$.
\end{prop}

\begin{proof}
Fix $x\in S$, $\eta\in\mc{E}_{N}^{x}$, and let
\begin{equation*}
q_{N}(\eta) \;=\; \mb{E}_{\eta}^{N}\Big[\,
\int_{0}^{T}Q_{N}(\xi_{N}(s))ds\, \Big]\;.
\end{equation*}
Note that
\begin{equation}
|q_{N}(\zeta)|\le TM\;\;\;\;
\text{for all }N\in\bb{N}\;
\text{and\;}\zeta\in\mc{H}_{N}\;.
\label{bdqn}
\end{equation}

By \eqref{33}, there exists a sequence $(\epsilon_N : N\ge 1)$ such
that $\lim_N \epsilon_N =0$ and $t^{x}_{\rm mix} (\epsilon_N) \;\le\;
\mb h_N$ for all $N\ge 1$. Let $\mb{s}_{N} = t^{x}_{\rm mix}
(\epsilon_N) $. Since $Q_N$ is uniformly bounded by $M$,
\begin{equation*}
q_{N}(\eta) \;=\; \mb{E}_{\eta}^{N}
\Big[\int_{\mb{h}_{N}}^{T+\mb{h}_{N}}
Q_{N}(\xi_{N}(s))ds\Big] \;+\; M\,O(\mb{h}_{N})\;.
\end{equation*}
By \eqref{23}, this expectation is equal to
\begin{equation*}
\mb{E}_{\eta}^{N}\Big[\,
\int_{\mb{h}_{N}}^{T+\mb{h}_{N}}Q_{N}(\xi_{N}(s))ds\,,\,
\tau_{(\mc V_{N}^{x})^{c}} \,>\, \mb{h}_{N}\, \Big]
\;+\; M\, T\,o_{N}(1) \;.
\end{equation*}
By the Markov property and the definition of $q_N$, we may write the
previous expectation as
\begin{equation*}
\mb{E}_{\eta}^{N}\Big[ \,q_{N}(\xi_{N}(\mb{h}_{N}))\,,\,
\tau_{(\mc V_{N}^{x})^{c}} > \mb{h}_{N}\, \Big]\;.
\end{equation*}

Recall that we denote by $\xi^{R,x}_N(\cdot)$ the reflected process at
the boundary of $\mc V_{N}^{x}$. Denote by
\textcolor{blue}{$\mb{P}_{\eta}^{R,x}$} the law of the reflected
process $\xi^{R,x}_N(\cdot)$, and by
\textcolor{blue}{$\mb{E}_{\eta}^{R,x}$} the expectation with respect
to $\mb{P}_{\eta}^{R,x}$.

Due to the presence of the indicator of the set
$\{\tau_{(\mc V_{N}^{x})^{c}} > \mb{h}_{N}\}$, we may replace in the
previous expectation $\xi_{N}(\mb{h}_{N})$ by
$\xi^{R,x}_{N}(\mb{h}_{N})$ and then remove the indicator of that
set. After these modifications the previous expression becomes
\begin{equation*}
\mb{E}^{R,x}_{\eta} \Big[ \,q_{N}(\xi_{N}(\mb{h}_{N}))\, \Big]
\;+\; M\, T\,o_{N}(1) \;.
\end{equation*}
By definition of $\mb s_N$ and since $\mb s_N \le \mb h_N$, the
expectation is equal to
\begin{equation*}
E_{\pi^{R,x}} \big[ \,q_{N}\, \big] \;+\; M\, T\,o_{N}(1) \;.
\end{equation*}

We just proved that
\begin{equation*}
\sup_{\eta, \zeta\in\mc{E}_{N}^{x}}
\big|\, q_{N}(\eta) \,-\, q_{N}(\zeta) \,\big|
\;\le \;6\, M\, (T+1)\,o_{N}(1)\;.
\end{equation*}
The assertion of the proposition follows from this bound by averaging
$\zeta$ according to $\mu^x_N$.
\end{proof}

\begin{corollary}
\label{cor01}
Assume that condition $\mf M$ is in force. Then, for all $x\in S$,
$T>0$,
\begin{equation*}
\sup_{\eta\in\mc E^x_N}
\mb{E}_{\eta}^{N}\,\Big[\,\int_{0}^{T} \chi_{\Delta_N}
(\xi_{N}(t))\, dt\,\Big]\, \le\;
\frac{\mu_N(\Delta_N)}{\mu_N(\mc E^x_N)}\, T
\; + \; 6(T+1)\,o_{N}(1)\;.
\end{equation*}
In particular, if $\mu_N(\Delta_N)/\mu_N(\mc E^x_N) \to 0$ for all
$x\in S$, then condition $\mf D$ holds.
\end{corollary}

\begin{proof}
By the proposition, for every $x\in S$, $\eta\in \mc E^x_N$ and $T>0$,
\begin{equation*}
\mb{E}_{\eta}^{N}\,\Big[\,\int_{0}^{T} \chi_{\Delta_N}
(\xi_{N}(t))\, dt\,\Big]\, \le\;
\mb{E}_{\mu_{N}^{x}}^{N}\,\Big[\,\int_{0}^{T}
\chi_{\Delta_N} (\xi_{N}(t))\, dt\,\Big]
\;+ \; 6(T+1)\,o_{N}(1)\;.
\end{equation*}
The expectation is bounded by
\begin{equation*}
\frac{1}{\mu_N(\mc E^x_N)} \mb{E}_{\mu_{N}}^{N}\,\Big[\,\int_{0}^{T}
\chi_{\Delta_N} (\xi_{N}(t))\, dt\,\Big]
\;=\; \frac{\mu_N(\Delta_N)}{\mu_N(\mc E^x_N)}\, T\;,
\end{equation*}
where the last identity follows from the fact that $\mu_N$ is the
stationary state.
\end{proof}

\section{Proof of Theorem \ref{mt1}}
\label{s7}

In this section, we examine the possible limits of the average of the
solutions of the resolvent equation \eqref{1f01} in each well and
prove Theorem \ref{mt1}. Most of the notation is borrowed from Section
\ref{sec8}.

Recall from the statement of Theorem \ref{mt1} the definition of the
function $f_N$.  Note that condition $\mf R^{(1)}$ holds if and only
if
\begin{equation*}
\lim_{N\to\infty} \, \max_{z\in S} \, \max_{\eta\in \mc E^z_N} \,
\big|\, F_N(\eta) \,-\, f_N(z)\,\big| \;=\; 0\;.
\end{equation*}

\medskip\noindent{\bf Condition $\mf R^{(2)}_{\ms L}$}. Let $\ms L$ be
the generator of an $S$-valued, continuous-time Markov chain. For all
$x\in S$,
\begin{equation*}
\lim_{N\to\infty} f_N(x) \;=\; f(x)  \;,
\end{equation*}
where $f\colon S\rightarrow\bb{R}$ is the solution of the reduced
resolvent equation
\begin{equation*}
(\, \lambda \,-\, \ms L\,) \, f \;=\; g\;.
\end{equation*}

\begin{remark}
\label{r1r2}
It is clear that $\mf R^{(1)}$ and $\mf R^{(2)}_{\ms L}$ together
imply condition $\mf R_{\ms L}$.
\end{remark}

By \eqref{1f10} and the definition of $f_N$, there exists a finite
constant $C_0 = C_0(\lambda, g)$ such that
\begin{equation*}
\sup_{N\ge 1}\, \max_{x\in S}
\big|\, f_N(x)\,\big| \;\le\; C_0 \;.
\end{equation*}

Let $\ms L$ be the generator of the $S$-value Markov chain induced
by the rates $r$ introduced in condition (H0):
\begin{equation*}
(\ms L f)(x) \;=\; \sum_{y\in S} r(x,y)\, [\, f(y)\,-\, f(x)\,]\;,
\end{equation*}
and let
$\color{blue} \nu^\dagger_y = \nu^\dagger_{\mc{E}_N^y,
  \breve{\mc{E}}_N^y}$, $y\in S$, be the equilibrium measure between
$\mc E^y_N$ and $\breve{\mc E}^y_N$, as defined in \eqref{38}.

\begin{proposition}
\label{l01b}
Assume that conditions {\rm (H0)} and $\mf R^{(1)}$ are in force.
Let $f$ be a limit point of the sequence $f_N$. Then,
\begin{equation*}
\big[\, (\, \lambda \,-\, \ms L\,) \, f\,\big](y) \;=\; g(y)
\end{equation*}
for all $y\in S$ such that
\begin{equation}
\label{36}
\lim_{N\to\infty} \Big( \sum_{z\not = y}  r_N(y,z) \Big) \,
\mb E^N_{\nu^\dagger_y} \Big[ \, \int_0^{\tau_{\breve{\mc{E}}_{N}^{y}}}
\chi_{\Delta_N} (\xi_N(s))\, ds\,\Big] \;=\;0\;.
\end{equation}
\end{proposition}

\begin{proof}
Fix $y\in S$, and denote by
$h^\dagger_y= h^\dagger_{\mc{E}_N^y, \breve{\mc{E}}_N^y}$ the
equilibrium potential between $\mc{E}_{N}^{y}$ and
$\breve{\mc{E}}_{N}^{y}$ for the adjoint process, as defined in
\eqref{03}.  Multiply the resolvent equation \eqref{1f01} by
$h^\dagger_y$ and integrate with respect to the stationary measure
$\mu_N$ to get that
\begin{equation}
\label{02}
\lambda\, \< F_N\,,\, h^\dagger_y\>_{\mu_N} \;-\;
\< \ms L_N F_N\,,\, h^\dagger_y\>_{\mu_N} \;=\;
\< G_N\,,\, h^\dagger_y\>_{\mu_N} \;.
\end{equation}

Consider the right-hand side of this equation. Since $G_N$ vanishes on
$\Delta_N$ and is equal to $g(z)$ on $\mc E^z_N$, $z\in S$, and since
on the set $\mc E_N$, $h^\dagger_y$ is equal to the indicator of the set
$\mc E^y_N$,
\begin{equation*}
\< G_N\,,\, h^\dagger_y\>_{\mu_N} \;=\; g(y)\, \mu_N(\mc E^y_N) \;.
\end{equation*}

We turn to the first term on the left-hand side of \eqref{02}. By
similar reasons, it is equal to
\begin{equation*}
\lambda\,  \sum_{\eta\in \mc E^y_N} \mu_N(\eta)  \, F_N(\eta)
\;+\; \sum_{\eta\in \Delta_N} \mu_N(\eta)  \, F_N(\eta) \,
h^\dagger_y (\eta)\;.
\end{equation*}
By \eqref{1f10}, the sequence $F_N$ is uniformly bounded. As
$h^\dagger_y $ is bounded by $1$, the second term is bounded by
$C_0(\lambda, g) \, \sum_{\eta\in \Delta_N} \mu_N(\eta) \, h^\dagger_y
(\eta)$. On the other hand, by definition of $f_N$, the first term is
equal to $\lambda\, \mu_N( \mc E^y_N) \, f_N(y)$.

We turn to the second term on the left-hand side of \eqref{02}.
Since $\ms L^\dagger_N h^\dagger_y =0$ on $\Delta_N$,
\begin{equation*}
\< \, \ms L_N F_N\,,\, h^\dagger_y\, \>_{\mu_N} \;=\;
\<\,   F_N\,,\, \ms L^\dagger_N \, h^\dagger_y\, \>_{\mu_N}
\;=\; \sum_{x\in S}
\sum_{\eta\in \mc E^x_N} \mu_N(\eta)  \, F_N(\eta) \,
(\ms L^\dagger_N h^\dagger_y) (\eta)\;.
\end{equation*}
Since the equilibrium potential $h^\dagger_y$ vanishes on
$\breve{\mc{E}}_{N}^{y}$ and is equal to $1$ on $\mc E_{N}^{y}$, for
$\eta\in \mc E_{N}^{x}$, $x\not =y$, as
$\lambda^\dagger_N(\eta) = \lambda_N(\eta)$,
\begin{equation*}
\begin{aligned}
& (\ms L^\dagger_N h^\dagger_y) (\eta) \;=\; \sum_{\zeta\in\mc H_N}
R^\dagger_N(\eta,\zeta)\, \big[ \, h^\dagger_y(\zeta)
\,-\, h^\dagger_y(\eta) \,\big] \\
& \;=\; \lambda_N(\eta)\, \sum_{\zeta\in\mc H_N}
p^\dagger_N(\eta,\zeta)\,
\mb P^{N,\dagger}_{\zeta} \big[\,
\tau_{\mc{E}_{N}^{y}} \,<\, \tau_{\breve{\mc{E}}_{N}^{y}} \,\big]
\;=\; \lambda_N(\eta)\,
\mb P^{N,\dagger}_{\eta} \big[\,
\tau_{\mc{E}_{N}^{y}} \,<\, \tau^+_{\breve{\mc{E}}_{N}^{y}} \,\big]\;.
\end{aligned}
\end{equation*}
Similarly, as $h^\dagger_y(\zeta) \,-\, 1 = -\,
\mb P^{N,\dagger}_{\zeta} \big[\,
\tau_{\breve{\mc{E}}_{N}^{y}} \,<\, \tau_{\mc{E}_{N}^{y}}   \,\big]$, for
$\eta\in \mc E_{N}^{y}$,
\begin{equation*}
(\ms L^\dagger_N h^\dagger_y) (\eta) \;=\; -\, \lambda_N(\eta)\,
\mb P^{N,\dagger}_{\eta} \big[\, \tau_{\breve{\mc{E}}_{N}^{y}} \,<\,
\tau^+_{\mc{E}_{N}^{y}}  \,\big]\;.
\end{equation*}
Therefore,
\begin{equation*}
\begin{aligned}
\< \, \ms L_N F_N \,,\, h^\dagger_y \, \>_{\mu_N} \;  =& \;
\sum_{x\not = y}
\sum_{\eta\in \mc E^x_N} \mu_N(\eta)  \, \lambda_N(\eta)\,
F_N(\eta) \, \mb P^{N,\dagger}_{\eta} \big[\,
\tau_{\mc{E}_{N}^{y}} \,<\, \tau^+_{\breve{\mc{E}}_{N}^{y}} \,\big] \\
& \qquad -\;
\sum_{\eta\in \mc E^y_N} \mu_N(\eta)  \, \lambda_N(\eta)\,
F_N(\eta) \, \mb P^{N,\dagger}_{\eta} \big[\, \tau_{\breve{\mc{E}}_{N}^{y}} \,<\,
\tau^+_{\mc{E}_{N}^{y}}  \,\big] \;.
\end{aligned}
\end{equation*}
Recall from \eqref{07b} the definition of $r^\dagger_N(z,z')$. Add and
subtract $f_N$ to rewrite the right-hand side as
\begin{equation}
\label{19b}
\sum_{x\not = y} \mu_N(\mc E^x_N) \, f_N(x) \, r^\dagger_N(x,y)
\; -\; \mu_N(\mc E^y_N) \, f_N(y) \,
\sum_{x\not = y}  r^\dagger_N(y,x) \;+\; R_N\;,
\end{equation}
where the absolute value of the remainder $R_N$ is bounded by
\begin{equation*}
\max_{z\in S} \max_{\eta\in \mc E^z_N} \,
\big|\, F_N(\eta) \,-\, f_N(z)\,\big| \;
\Big\{\, \sum_{x\not = y} \mu_N(\mc E^x_N) \, r^\dagger_N(x,y)
\; +\; \mu_N(\mc E^y_N) \, \sum_{x\not = y}  r^\dagger_N(y,x) \,\Big\}\;.
\end{equation*}
By Lemma \ref{l07b}, this expression can be rewritten as
\begin{equation*}
2\, \mu_N(\mc E^y_N) \, \max_{z\in S} \max_{\eta\in \mc E^z_N} \,
\big|\, F_N(\eta) \,-\, f_N(z)\,\big| \;
\sum_{x\not = y}  r_N(y,x)\;.
\end{equation*}
By the same reasons, the sum of the first two terms in \eqref{19b} is
equal to
\begin{equation*}
\mu_N(\mc E^y_N) \,
\sum_{x\not = y}  r_N(y,x)\, [\, f_N(x) \, -\, f_N(y) \,]\;.
\end{equation*}

Recollecting all previous calculations and dividing by
$\mu_N(\mc E^y_N)$ permits to rewrite \eqref{02} as
\begin{equation*}
\lambda\, f_N(y) \;-\; \sum_{x\not = y}  r_N(y,x)\, [\, f_N(x) \, -\,
f_N(y) \,] \;=\; g(y) \;+\; R^{(2)}_N\;,
\end{equation*}
where the absolute value of $R^{(2)}_N$ is bounded by
\begin{equation*}
\frac{C_0}{\mu_N(\mc E^y_N)}\,
\sum_{\eta\in \Delta_N} \mu_N(\eta)  \, h^\dagger_y (\eta)
\;+\; 2\, \max_{z\in S} \max_{\eta\in \mc E^z_N} \,
\big|\, F_N(\eta) \,-\, f_N(z)\,\big| \;
\sum_{x\not = y}  r_N(y,x)
\end{equation*}
for some finite constant $C_0 = C_0(\lambda, g)$. By \eqref{34}, with
$\mc A = \mc E^y_N$, $\mc B = \breve{\mc{E}}_{N}^{y}$, and the second
assertion of Lemma \ref{l07b}, this expression can be rewritten as
\begin{equation*}
C_0 \sum_{x\not = y}  r_N(y,x)\,
\Big\{\,
\mb E^N_{\nu^\dagger_y} \Big[ \, \int_0^{\tau_{\breve{\mc{E}}_{N}^{y}}}
\chi_{\Delta_N} (\xi_N(s))\, ds\,\Big]
\;+\; \max_{z\in S} \max_{\eta\in \mc E^z_N} \,
\big|\, F_N(\eta) \,-\, f_N(z)\,\big| \,\Big\}\,
\end{equation*}
for a possibly different constant $C_0$.  To conclude the proof, it
remains to recall the statement of conditions (H0), $\mf{R}^{(1)}$,
and the hypotheses of the proposition.
\end{proof}

In the previous proof we used the identity,
\begin{equation}
\label{36b}
\Big( \sum_{z\not = y}  r_N(y,z) \Big) \,
\mb E^N_{\nu^\dagger_y} \Big[ \, \int_0^{\tau_{\breve{\mc{E}}_{N}^{y}}}
\chi_{\Delta_N} (\xi_N(s))\, ds\,\Big] \;=\;
\frac{1}{\mu_N(\mc E^y_N)}\,
\sum_{\eta\in \Delta_N} \mu_N(\eta)  \, h^\dagger_y (\eta)\;.
\end{equation}
In particular, \eqref{36} holds for $y\in S$ if and only if the
right-hand side vanishes as $N\to\infty$.

\begin{corollary}
\label{l11}
Assume that conditions {\rm (H0)} and $\mf R^{(1)}$ are in
force.  Let $f$ be a limit point of the sequence $f_N$. Then,
\begin{equation*}
\big[\, (\, \lambda \,-\, \ms L_Y\,) \, f\,\big](y) \;=\; g(y)
\end{equation*}
for all $y\in S$ such that $\mu_N(\Delta_N)/\mu_N(\mc E^y_N)\to 0$.
In this formula, $\ms L_Y$ is the generator of the continuous-time
Markov process whose jump rates are given by $r(x,y)$, introduced in
{\rm (H0)}.
\end{corollary}

\begin{proof}
The right-hand side of \eqref{36b} is bounded by
$\mu_N(\Delta_N)/\mu_N(\mc E^y_N)$.  Thus, the assertion of the
corollary follows from the statement of Proposition \ref{l01b}.
\end{proof}

\begin{proof}[Proof of Theorem \ref{mt1}]
Theorem \ref{mt1} follows from Corollaries \ref{l04} and \ref{l11}.
\end{proof}

We complete this section with a method to prove condition \eqref{36}
when the hypotheses of Corollary \ref{l11} are not verified. The idea
behind the decomposition below is that $\mc A_N$ is contained in the
basin of attraction of $\breve{\mc{E}}_{N}^{y}$. In particular,
starting from a configuration in $\mc A_N$ the set
$\breve{\mc{E}}_{N}^{y}$ is reached quickly. We refer to Figure
\ref{fig0} for an example of illustration of the set $\mc A_N$.

\begin{figure}
\includegraphics[scale=0.09]{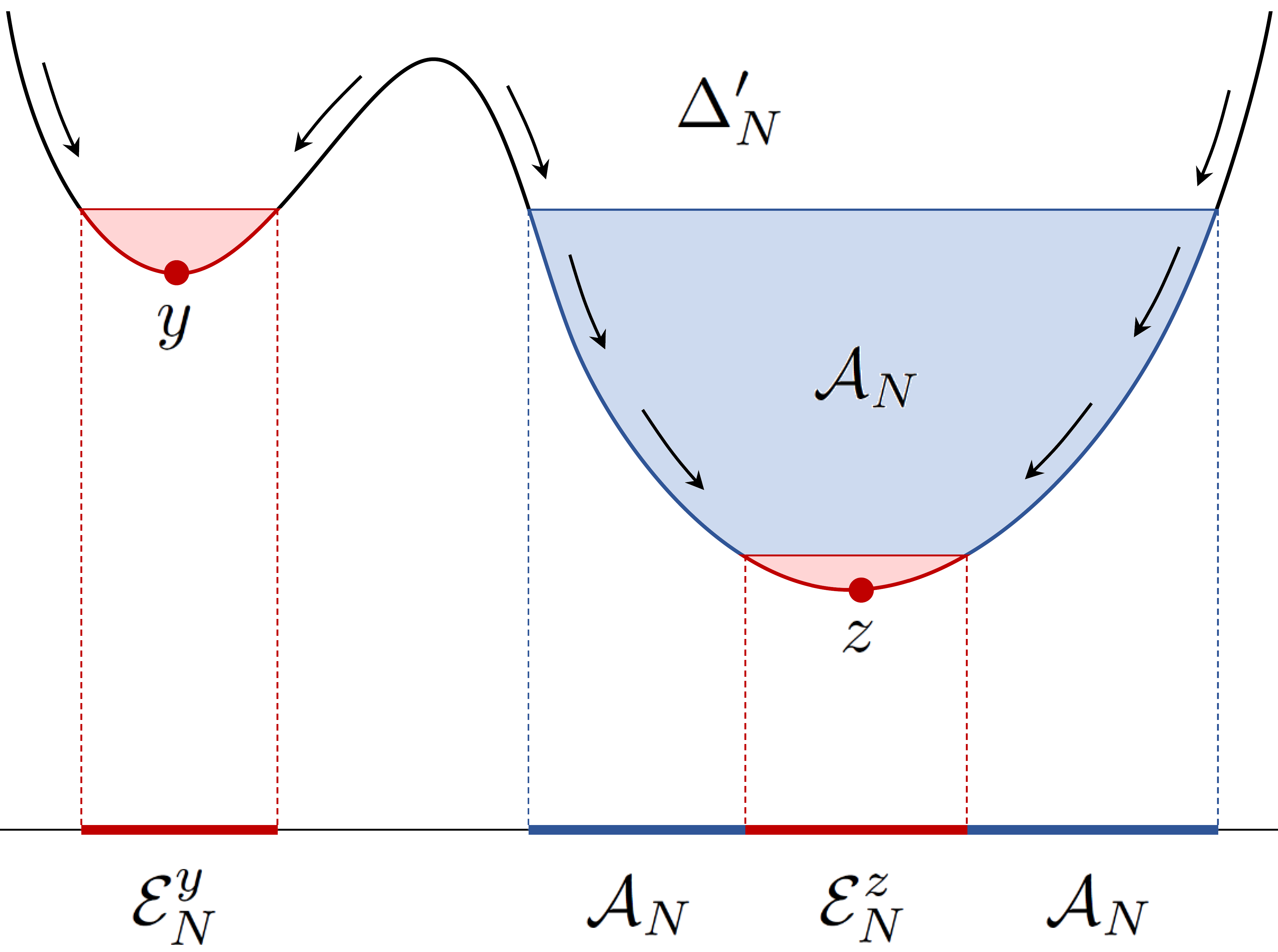}
\caption{ This picture illustrates the idea behind the statement of
  Lemma \ref{l13}. To simplify, we argue in a continuous setting, but
  the same idea applies to the discrete setting. Consider a diffusion
  on the potential field appearing in the picture. The valley
  $\mc{E}_N^y$ is a metastable set and $\mc{E}_N^z$ a stable one. As
  $\mu_N(\Delta_N)/\mu_N(\mc{E}_N^y)$ does not converge to $0$, we
  decompose $\Delta_N$ as $\Delta'_N \cup \mc A_N$, so that
  $\mu_N (\Delta'_N) / \mu_N (\mc E_N^y) \rightarrow 0$. On the other
  hand, as $\mc A_N$ is a subset of the domain of attraction of the
  valley $\mc E_N^z$, we can expect \eqref{an_cond} to hold.}
\label{fig0}
\end{figure}

\begin{lemma}
\label{l13}
Fix $y\in S$, and suppose that $\Delta_N$ may be decomposed as
$\Delta_N = \Delta'_N \cup \mc A_N$,
$\Delta'_N \cap \mc A_N = \varnothing$, where
$\mu_N(\Delta'_N)/\mu_N(\mc E^y_N) \to 0$, and
\begin{equation}
\label{an_cond}
\lim_{N\to\infty} \sup_{\zeta\in \mc A_N}
\mb E^N_{\zeta} \Big[ \, \int_{0}^{\tau_{\breve{\mc{E}}_{N}^{y}}}
\chi_{\mc A_N} (\xi_N(s))\, ds\,\Big] \;=\; 0 \;.
\end{equation}
Then, \eqref{36} holds for $y$.
\end{lemma}

\begin{proof}
In equation \eqref{36}, write $\chi_{\Delta_N}$ as
$\chi_{\mc A_N} + \chi_{\Delta'_N}$.  We estimate the two pieces
separately.  By \eqref{36b} with $\Delta'_N$ instead of $\Delta_N$,
\begin{equation*}
\Big( \sum_{x\not =y} r_N(y,x) \Big)\,
\mb E^N_{\nu^\dagger_y} \Big[ \, \int_0^{\tau_{\breve{\mc{E}}_{N}^{y}}}
\chi_{\Delta'_N} (\xi_N(s))\, ds\,\Big] \;\le\;
\frac{\mu_N(\Delta'_N)}{\mu_N(\ms E^y_N)}\;\cdot
\end{equation*}
By hypothesis, this expression vanishes as $N\to\infty$.

On the other hand, starting the integral from the hitting time of
$\mc A_N$ and applying the strong Markov property, yields that
\begin{equation*}
\begin{aligned}
\mb E^N_{\nu^\dagger_y} \Big[ \, \int_0^{\tau_{\breve{\mc{E}}_{N}^{y}}}
\chi_{\mc A_N} (\xi_N(s))\, ds\,\Big]  \; & =\;
\mb E^N_{\nu^\dagger_y} \Big[ \,
\int_{\tau_{\mc A_N}}^{\tau_{\breve{\mc{E}}_{N}^{y}}}
\chi_{\mc A_N} (\xi_N(s))\, ds\,\Big] \; \\
& \le\; \sup_{\zeta\in \mc A_N}
\mb E^N_{\zeta} \Big[ \, \int_{0}^{\tau_{\breve{\mc{E}}_{N}^{y}}}
\chi_{\mc A_N} (\xi_N(s))\, ds\,\Big] \;.
\end{aligned}
\end{equation*}
By assumption this expression vanishes as $N\to\infty$, which
completes the proof of the lemma.
\end{proof}

\section{Proof of Theorem \ref{t_main2}}
\label{sec5}

In view of Theorem \ref{1t_main2}, to prove Theorem \ref{t_main2},
we have to show that conditions $\mf{D}$ and $\mf{R}_{\ms{L}}$ hold.
The proof of these properties is based on the theory developed in the
previous sections. We proceed as follows.

\smallskip\noindent{\bf Condition $\mf{R}_{\ms{L}}$}. In Proposition
\ref{p_mixmain}, we show that condition $\mf V$ is
fulfilled. Hence, by Proposition \ref{p01}, property $\mf R^{(1)}$
holds.

In Corollary \ref{corH0} we show that condition (H0) holds. Since we
already proved that condition $\mf R^{(1)}$ is fulfilled, and since,
by Theorem \ref{t_cond}, $\mu_N(\Delta_N)/\mu_N(\mc E^x_N) \to 0$ for
all $x\in S$, by Corollary \ref{l11}, property $\mf R^{(2)}_{L_Y}$ is
in force, where $L_Y$ is the generator introduced in \eqref{e_genZ}.

\smallskip\noindent{\bf Condition $\mf D$}.  Recall assumptions
\eqref{23} and \eqref{33} of condition $\mf M$. In Corollary
\ref{cor_m1}, we show that condition \eqref{23} holds for some
enlarged wells $\mc V^x_N$ and a time-scale $\mb h_N \ll 1$. Then, in
Proposition \ref{pro_mix2}, we prove that, for every $\epsilon>0$, the
mixing time $t^x_{\rm mix} (\epsilon)$ of the zero-range process
reflected at the boundary of the set ${\mc V}^x_N$ is bounded by a
sequence $\mb s_N \ll \mb h_N$. This property implies condition
\eqref{33}.  These two results yield condition $\mf M$, which is the
assertion of Corollary \ref{corM}.  Thus, by Theorem \ref{t_cond} and
Corollary \ref{cor01}, property $\mf D$ is fulfilled.

\begin{remark}
To deduce property $\mf R^{(1)}$ one could also invoke Corollary
\ref{corM} and Proposition \ref{p03}.  On the other hand, condition
$\mf{R}^{(2)}_{L_Y}$ has been proven in an alternative way in
\cite[Section 7]{LMS}.
\end{remark}

\section{Escape from large wells}
\label{sec3}

In this section, we prove that condition \eqref{23} holds for the
critical zero-range process for a sequence
$(\mb{h}_{N})_{N\in\bb{N}}$, $\mb h_N \to 0$, and enlarged wells
$(\mc{V}_N^x,\,x\in S)_{N\in\bb{N}}$, $\mc V^x_N\supset \mc E^x_N$.

For $N\in\bb{N}$, set
\begin{equation*}
m_{N}=N/(\log N)^{\delta}\text{\; with}\;\;\delta\in(0,\,1)
\end{equation*}
and let $\mb{h}_{N}$ be the macroscopic time-scales given by
\begin{equation}\label{h_n}
\mb{h}_{N}\;:=\;\frac{m_{N}^{2}\,(\log N)^{1/2}}
{\theta_{N}}\;=\;\frac{1}{(\log N)^{(1/2)+2\delta}}\;\cdot
\end{equation}
For $x\in S$, define a larger well by
\begin{equation}
\mc{V}_N^x\;=\;\{\,\eta\in\mc{H}_{N}:
\eta_{y}\,\le\,m_{N}\text{ for all }y\in S\setminus\{x\}\,\}\;.
\label{large_well}
\end{equation}
As in \cite{LMS}, denote by $\mc{W}_{N}^{x}$, $\mc{D}_{N}^{x}$,
$x\in S$, the wells given by
\begin{equation}
\mc{W}_{N}^{x}\;=\;\{\,\eta\in\mc{H}_{N}\,:\,
\eta_{x}\,\ge\,N-m_{N}\,\}\;,\quad\mc{D}_{N}^{x}
\;=\;\{\,\eta\in\mc{H}_{N}\,:\,\eta_{x}\,\ge\,N-N^{\gamma}\,\}\;.
\label{e_Dnx}
\end{equation}
In this formula, $\gamma\in(0,\,2/\kappa)$ is a fixed constant. The
sets $\mc{D}_{N}^{x}$ are called deep wells and the sets $\mc{W}_{N}^{x}$
shallow wells.

Denote by ${\color{blue}\zeta_N^x}\in\mc{H}_{N}$ the configuration
such that all $N$ particles are located at site $x$, so that
\begin{equation}\label{setorder}
\zeta_N^x\,\in\,\mc{D}_{N}^{x}\,\subset\,\mc{E}_{N}^{x}\,\subset\,\mc{W}_{N}^{x} \,\subset\, \mc{V}_N^x\;.
\end{equation}

The main result of this section asserts that the process $\xi_{N}(\cdot)$
starting from a well $\mc{E}_{N}^{x}$ cannot escape from the
well $\mc{W}_{N}^{x}$ within the time scale $\mb{h}_{N}$.

\begin{prop}
\label{p_mix1} For all $x\in S$,
\begin{equation*}
\lim_{N\to\infty}\sup_{\eta\in\mc{E}_{N}^{x}}
\mb{P}_{\eta}^{N}\,[\,\tau_{(\mc{W}_{N}^{x})^{c}}\le\mb{h}_{N}\,]\;=\;0\;.
\end{equation*}
\end{prop}

By the last inclusion of \eqref{setorder}, the next result is a
straightforward consequence of Proposition \ref{p_mix1}.

\begin{cor}
\label{cor_m1} For all $x\in S$,
\begin{equation*}
\lim_{N\rightarrow\infty}\sup_{\eta\in\mc{E}_{N}^{x}}
\mb{P}_{\eta}^{N}\,[\,\tau_{(\mc{V}_N^x)^{c}}\le\mb{h}_{N}\,]\;=\;0\;.
\end{equation*}
\end{cor}

\subsection{Estimates based on capacity\label{sec31}}

In this subsection, we state several estimates based on the following
bound of the capacity between $\mc{E}_{N}^{x}$ and
$(\mc{W}_{N}^{x})^{c}$ with respect to the critical zero-range
processes.

\begin{lem}
\label{p_capb} There exists a finite constant $C$ such that
\begin{equation*}
\textup{cap}_{N}(\mc{E}_{N}^{x},\,(\mc{W}_{N}^{x})^{c})\;\le\;\frac{C\,\theta_{N}}{m_{N}^{2}\log N}
\end{equation*}
for all $x\in S$, $N\ge1$.
\end{lem}

\begin{proof}
Let ${\color{blue}Q(\eta)}=q(N-\eta_{x})$ for some function $q:\bb{Z}\to\bb{R}$
such that
\begin{equation*}
q(k)\;=\;\begin{cases}
0 & \text{if }k\le\ell_{N}\;,\\
1 & \text{if }k\ge m_{N}\;.
\end{cases}
\end{equation*}
The precise expression for $q$ will be specified below in \eqref{e_defq}.
By the Dirichlet principle and since $Q(\sigma^{y,\,z}\eta)=Q(\eta)$
if $y,\,z\neq x$ or $\eta\notin\mc{W}_{N}^{x}\setminus\mc{E}_{N}^{x}$,
\begin{align*}
 & \textup{cap}_{N}(\mc{E}_{N}^{x},\,(\mc{W}_{N}^{x})^{c})\\
 & \qquad\le\;\ms{D}_{N}(Q)\;=\;\theta_{N}\,\sum_{\eta\in\mc{W}_{N}^{x}\setminus\mc{E}_{N}^{x}}\,\sum_{y\in S}\,\mu_{N}(\eta)\,g(\eta_{x})\,r(x,\,y)\,[\,Q(\sigma^{x,\,y}\eta)\,-\,Q(\eta)\,]^{2}\;.
\end{align*}
By definition of the jump rates and of $Q$, this expression is bounded
by
\begin{align*}
C\,\theta_{N}\,\sum_{\eta\in\mc{W}_{N}^{x}\setminus\mc{E}_{N}^{x}}\mu_{N}(\eta)\,[\,q(N-\eta_{x}+1)\,-\,q(N-\eta_{x})\,]^{2}\;,
\end{align*}
for some finite constant $C$. Let
\begin{equation*}
\mc{R}_{k}\;=\;\{\eta\in\mc{H}_{N}:N-\eta_{x}=k\}
\end{equation*}
so that
\begin{equation*}
\textup{cap}_{N}(\mc{E}_{N}^{x},\,(\mc{W}_{N}^{x})^{c})\;\le\;C\,\theta_{N}\,\sum_{k=\ell_{N}}^{m_{N}-1}\,\mu_{N}(\mc{R}_{k})\,[\,q(k+1)\,-\,q(k)\,]^{2}\;.
\end{equation*}
Define
\begin{equation}
q(k)\;=\;\frac{\sum_{i=\ell_{N}}^{k-1}\,\mu_{N}(\mc{R}_{i})^{-1}}{\sum_{i=\ell_{N}}^{m_{N}-1}\,\mu_{N}(\mc{R}_{i})^{-1}}\;\;\;\;;\;k\,\in\,[\,\ell_{N},\,m_{N}\,]\;.\label{e_defq}
\end{equation}
It follows from the penultimate displayed equation that
\begin{equation}
\textup{cap}_{N}(\mc{E}_{N}^{x},\,(\mc{W}_{N}^{x})^{c})\;\le\;\frac{C\,\theta_{N}}{\sum_{i=\ell_{N}}^{m_{N}-1}\,\mu_{N}(\mc{R}_{i})^{-1}}\;.\label{e_37}
\end{equation}
For $\ell_{N}\le i<m_{N}$,
\begin{equation*}
\mu_{N}(\mc{R}_{i})\;=\;\frac{N}{Z_{N,\,\kappa}\,(\log N)^{\kappa-1}}\,\frac{1}{N-i}\,\frac{Z_{i,\,\kappa-1}\,(\log i)^{\kappa-2}}{i}\;\le\;\frac{C}{\log N}\,\frac{1}{i}\;\cdot
\end{equation*}
Here, we used the facts that $Z_{N,\,\kappa}$, $Z_{i,\,\kappa-1}$
are bounded, that $N/(N-i)\simeq1$ and $\log i/\log N\simeq1$ for
$i\in[\ell_{N},\,m_{N}-1]$. Inserting this into (\ref{e_37}) yields
\begin{equation*}
\textup{cap}_{N}(\mc{E}_{N}^{x},\,(\mc{W}_{N}^{x})^{c})\;\le\;\frac{C\,\theta_{N}}{(\log N)\,\sum_{i=\ell_{N}}^{m_{N}-1}i}\;\le\;\frac{C\,\theta_{N}}{m_{N}^{2}\,\log N}\;,
\end{equation*}
as claimed.
\end{proof}
Based on the previous estimate of the capacity, we can refine
\cite[Propositions 9.4 and 8.6]{LMS}, replacing the set
$\breve{\mc{E}}_{N}^{x}=\mc{E}_{N}\setminus\mc{E}_{N}^{x}$ by the much
larger set $(\mc{W}_{N}^{x})^{c}$.
\begin{lem}
\label{p_visit1} For all $x\in S$,
\begin{equation*}
\lim_{N\rightarrow\infty}\,\inf_{\eta\in\mc{D}_{N}^{x}}\,\inf_{\zeta\in\mc{D}_{N}^{x}}\mb{P}_{\eta}^{N}\,[\,\tau_{\zeta}\,<\,\tau_{(\mc{W}_{N}^{x})^{c}}\,]\;=\;1\;.
\end{equation*}
\end{lem}

\begin{proof}
By \cite[equation (3.3)]{LL} and the monotonicity of the capacity,
\begin{equation*}
\mb{P}_{\eta}^{N}\,[\,\tau_{\zeta}\,>\,\tau_{(\mc{W}_{N}^{x})^{c}}\,]\;\le\;\frac{\textup{cap}_{N}(\eta,\,(\mc{W}_{N}^{x})^{c})}{\textup{cap}_{N}(\eta,\,\zeta)}\;\le\;\frac{\textup{cap}_{N}(\mc{E}_{N}^{x},\,(\mc{W}_{N}^{x})^{c})}{\textup{cap}_{N}(\eta,\,\zeta)}\;.
\end{equation*}
Hence, by \cite[Lemma 9.3]{LMS} and Lemma \ref{p_capb},
\begin{equation*}
\mb{P}_{\eta}^{N}\,[\,\tau_{\zeta}\,>\,\tau_{(\mc{W}_{N}^{x})^{c}}\,]\;\le\;C\,\frac{N^{\gamma\kappa}\,(\log N)^{\kappa-1}}{m_{N}^{2}\,\log N}\;=\;o_{N}(1)\;,
\end{equation*}
where the last equality holds because $\gamma<2/\kappa$. This completes
the proof.
\end{proof}
Recall from \cite[Proposition~9.1]{LMS} that the deep wells $\mc{D}_{N}^{x}$
are attractors, in the sense that
\begin{equation}
\lim\limits _{N\rightarrow\infty}\inf\limits _{\eta\in\mc{E}_{N}^{x}}\mb{P}_{\eta}^{N}[\,\tau_{\mc{D}_{N}^{x}}\,<\,\tau_{(\mc{W}_{N}^{x})^{c}}\,]\;=\;1\;\;\;\;\text{for all}\;x\in S\;.\label{deep_atract}
\end{equation}

\begin{lem}
\label{p_visit2} For all $x\in S$,
\begin{equation*}
\lim_{N\rightarrow\infty}\,\inf_{\zeta\in\mc{D}_{N}^{x}}\,\inf_{\eta\in\mc{E}_{N}^{x}}\mb{P}_{\eta}^{N}\,[\,\tau_{\zeta}\,<\,\tau_{(\mc{W}_{N}^{x})^{c}}\,]\;=\;1\;.
\end{equation*}
\end{lem}

\begin{proof}
Fix $x\in S$, $\eta\in\mc{E}_{N}^{x}$, and $\zeta\in\mc{D}_{N}^{x}$.
Then, by the strong Markov property,
\begin{align*}
\mb{P}_{\eta}^{N}[\,\tau_{\zeta}\,<\,\tau_{(\mc{W}_{N}^{x})^{c}}\,]\; & \ge\;\mb{P}_{\eta}^{N}\,[\,\tau_{\zeta}\,<\,\tau_{(\mc{W}_{N}^{x})^{c}}\,,\;\tau_{\mc{D}_{N}^{x}}\,<\,\tau_{(\mc{W}_{N}^{x})^{c}}\,]\\
 & \ge\;\mb{P}_{\eta}^{N}\,[\,\tau_{\mc{D}_{N}^{x}}\,<\,\tau_{(\mc{W}_{N}^{x})^{c}}\,]\,\inf_{\xi\in\mc{D}_{N}^{x}}\mb{P}_{\xi}^{N}\,[\,\tau_{\zeta}\,<\,\tau_{(\mc{W}_{N}^{x})^{c}}\,]\;.
\end{align*}
Therefore, we have
\begin{align*}
 & \inf_{\zeta\in\mc{D}_{N}^{x}}\,\inf_{\eta\in\mc{E}_{N}^{x}}\mb{P}_{\eta}^{N}\,[\,\tau_{\zeta}\,<\,\tau_{(\mc{W}_{N}^{x})^{c}}\,]\\
\ge\; & \inf_{\eta\in\mc{E}_{N}^{x}}\mb{P}_{\eta}^{N}\,[\,\tau_{\mc{D}_{N}^{x}}\,<\,\tau_{(\mc{W}_{N}^{x})^{c}}\,]\,\times\,\inf_{\zeta\in\mc{D}_{N}^{x}}\,\inf_{\xi\in\mc{D}_{N}^{x}}\mb{P}_{\xi}^{N}\,[\,\tau_{\zeta}\,<\,\tau_{(\mc{W}_{N}^{x})^{c}}\,]\;.
\end{align*}
The first term at the right-hand side is $1-o_{N}(1)$ by \eqref{deep_atract},
and the second one is $1-o_{N}(1)$ by Lemma \ref{p_visit1}. This
completes the proof of the lemma.
\end{proof}

\subsection{Proof of Proposition \ref{p_mix1}\label{sec33}}

The proof of Proposition \ref{p_mix1} is similar to the one of \cite[Theorem 3.2]{LMS}.
First, we establish the following estimate, whose proof is omitted
since it is completely identical to the proof of \cite[Proposition 8.4]{LMS}.
It suffices to replace $\breve{\mc{E}}_{N}^{x}$ by $(\mc{W}_{N}^{x})^{c}$
and $1/\gamma_{N}$ by $t_{N}$.
\begin{lem}
\label{p_BLbd} For all $x\in S$ and probability measure $\upsilon_{N}$
concentrated on $\mc{E}_{N}^{x}$,
\begin{equation*}
\big(\,\mb{P}_{\upsilon_{N}}^{N}\,[\,\tau_{(\mc{W}_{N}^{x})^{c}}\,\le\,\mb{h}_{N}\,]\,\big)\,^{2}\;\le\;e^{2}\,\mb{h}_{N}\,\bb{E}_{\mu_{N}^{x}}\,\bigg[\,\bigg(\,\frac{\upsilon_{N}}{\mu_{N}^{x}}\,\bigg)^{2}\,\bigg]\,\frac{1}{\mu_{N}(\mc{E}_{N}^{x})}\,\textup{cap}_{N}(\mc{E}_{N}^{x},\,(\mc{W}_{N}^{x})^{c})\;.
\end{equation*}
\end{lem}

By inserting $\upsilon_{N}={\color{blue}\pi_{N}^{x}(\cdot)}\coloneqq\mu_{N}(\,\cdot\,|\,\mc{D}_{N}^{x})$
into the previous equation, we obtain the following estimate. The
order of magnitude of $\mb{h}_{N}$ is critically used in the
proof of this result.
\begin{lem}
\label{p_escbd} For all $x\in S$,
\begin{equation*}
\lim_{N\to\infty}\mb{P}_{\pi_{N}^{x}}^{N}\,[\,\tau_{(\mc{W}_{N}^{x})^{c}}\,\le\,\mb{h}_{N}\,]\;=\;0\;.
\end{equation*}
\end{lem}

\begin{proof}
Since
\begin{equation*}
\bb{E}_{\mu_{N}^{x}}\,\Big[\,\Big(\,\frac{\upsilon_{N}}{\mu_{N}^{x}}\,\Big)\,\Big]^{2}\;=\;\frac{\mu_{N}(\mc{E}_{N}^{x})}{\mu_{N}(\mc{D}_{N}^{x})}\;,
\end{equation*}
by Lemmata \ref{p_capb} and \ref{p_BLbd}, we get
\begin{equation*}
\big(\,\mb{P}_{\upsilon_{N}}^{N}\,[\,\tau_{(\mc{W}_{N}^{x})^{c}}\le\mb{h}_{N}\,]\,\big)\,^{2}\;\le\;C\,\mb{h}_{N}\,\frac{1}{\mu_{N}(\mc{D}_{N}^{x})}\,\frac{\theta_{N}}{m_{N}^{2}\log N}\;=\;\frac{C}{(\log N)^{1/2}}\,\frac{1}{\mu_{N}(\mc{D}_{N}^{x})}\;\cdot
\end{equation*}
By \cite[Lemma 4.2]{LMS}, $\mu_{N}(\mc{D}_{N}^{x})\simeq\frac{1}{\kappa}\gamma^{\kappa-1}$,
which completes the proof.
\end{proof}
\smallskip{}

\begin{proof}[Proof of Proposition \ref{p_mix1}]
Fix $\eta\in\mc{E}_{N}^{x}$ and $\zeta\in\mc{D}_{N}^{x}$.
By Lemma \ref{p_visit2},
\begin{align*}
\mb{P}_{\eta}^{N}\,[\,\tau_{(\mc{W}_{N}^{x})^{c}}\le\mb{h}_{N}\,]\; & \le\;\mb{P}_{\eta}^{N}\,[\,\tau_{(\mc{W}_{N}^{x})^{c}}\,\le\,\mb{h}_{N}\,,\,\,\tau_{\zeta}\,<\,\tau_{(\mc{W}_{N}^{x})^{c}}\,]\,+\,\mb{P}_{\eta}^{N}\,[\,\tau_{\zeta}\,>\,\tau_{(\mc{W}_{N}^{x})^{c}}\,]\,\\
 & =\;\mb{P}_{\eta}^{N}\,[\,\tau_{(\mc{W}_{N}^{x})^{c}}\,\le\,\mb{h}_{N}\,,\,\,\tau_{\zeta}\,<\,\tau_{(\mc{W}_{N}^{x})^{c}}\,]\,+\,o_{N}(1)\;.
\end{align*}
By the strong Markov property,
\begin{equation*}
\mb{P}_{\eta}^{N}\,[\,\tau_{(\mc{W}_{N}^{x})^{c}}\,\le\,\mb{h}_{N}\,,\;\tau_{\zeta}\,<\,\tau_{(\mc{W}_{N}^{x})^{c}}\,]\;\le\;\mb{P}_{\zeta}^{N}\,[\,\tau_{(\mc{W}_{N}^{x})^{c}}\,\le\,\mb{h}_{N}\,]\,
\end{equation*}
so that
\begin{equation*}
\mb{P}_{\eta}^{N}\,[\,\tau_{(\mc{W}_{N}^{x})^{c}}\,\le\,\mb{h}_{N}\,]\;\le\;\mb{P}_{\zeta}^{N}\,[\,\tau_{(\mc{W}_{N}^{x})^{c}}\,\le\,\mb{h}_{N}\,]\,+\,o_{N}(1)\;.
\end{equation*}
Multiplying both sides by $\pi_{N}^{x}(\zeta)$ and summing over $\zeta\in\mc{D}_{N}^{x}$,
we get
\begin{equation*}
\mb{P}_{\eta}^{N}\,[\,\tau_{(\mc{W}_{N}^{x})^{c}}\,\le\,\mb{h}_{N}\,]\;\le\;\mb{P}_{\pi_{N}^{x}}^{N}\,[\,\tau_{(\mc{W}_{N}^{x})^{c}}\,\le\,\mb{h}_{N}\,]\,+\,o_{N}(1)\;.
\end{equation*}
Apply Lemma \ref{p_escbd} to complete the proof.
\end{proof}

\section{Condition $\mf{V}$ for critical zero-range processes}
\label{sec10}

In this section, we prove that the process $\xi_{N}(\cdot)$ starting
from a well $\mc{E}_{N}^{x}$ hits quickly the configuration
$\zeta_N^x$.
For $N\ge1$, define the time scale $\mb u_N$ by
\begin{equation*}
\mb{u}_{N}\;:=\;\frac{m_{N}^{2}}{\theta_{N}}
\;=\;\frac{1}{(\log N)^{1+2\delta}}\;.
\end{equation*}

\begin{prop}
\label{p_mixmain}For all $x\in S$,
\begin{equation*}
\lim_{N\to\infty}\sup_{\eta\in\mc{E}_{N}^{x}}\mb{P}_{\eta}^{N}\,[\,\tau_{\zeta_N^x}\ge\mb{u}_{N}\,]\;=\;0\;.
\end{equation*}
In particular, the condition $\mf{V}$ holds for the critical
zero-range processes.
\end{prop}

This result is crucially used in the next section to verify the
requirement \eqref{33} of condition $\mf{M}$.

\subsection{A super-harmonic function\label{sec32}}

We first establish, in Lemma \ref{p_mix2a} below, the estimate stated
in Proposition \ref{p_mixmain} for the process which is reflected at
the boundary of $\mc{W}_{N}^{x}$. The proof of Lemma \ref{p_mix2a} is
based on the construction, carried out in \cite[Section 10]{LMS}, of a
function $G_{N}^{x}:\mc{H}_{N}\rightarrow\bb{R}$, $x\in S$, which is
super-harmonic on $\mc{W}_{N}^{x}\setminus\mc{D}_{N}^{x}$.  For the
sake of completeness, we recall its definition and main properties
below.

Fix $x_{0}\in S$, and let $S_{0}=S\setminus\{x_{0}\}$. For a subset
$\mc{C}$ of $\mc{H}_{N}$, let
\begin{gather*}
\text{int}\,\mc{C}\;=\;\{\,\eta\in\mc{C}:\sigma^{x,\,y}\eta\in\mc{C}\text{ for all }x,\,y\text{ with }r(x,\,y)>0\,\}\;,\\
\partial\mc{C}\;=\;\mc{C}\setminus\text{int }\mc{C}\;,\\
\overline{\mc{C}}\;=\;\{\eta\in\mc{H}_{N}:\eta\in\mc{C}\text{ or }\sigma^{x,\,y}\eta\in\mc{C}\text{ for some }x,\,y\text{ with }r(x,\,y)>0\}\;.
\end{gather*}
With this notation, let
\begin{equation*}
\mc{U}_{N}^{x_{0}}\;=\;\overline{\mc{W}_{N}^{x_{0}}\setminus\mc{D}_{N}^{x_{0}}}\quad\text{so that}\quad\text{int}\,\mc{U}_{N}^{x_{0}}=\mc{W}_{N}^{x_{0}}\setminus\mc{D}_{N}^{x_{0}}\;.
\end{equation*}

Recall, from \eqref{e_capX}, that $h_{A,\,B}(\cdot)$ and $\textup{cap}_{X}(\cdot,\,\cdot)$
represent the equilibrium potential and the capacity, respectively,
associated to the underlying random walk $X$. For each non-empty
subset $A$ of $S_{0}$, consider the sequence $(b_{x,\,y}^{A})_{x,\,y\in S}$
defined by
\begin{equation*}
b_{x,\,y}^{A}\;=\;\frac{1}{\kappa}\,\frac{h_{x,\,A^{c}}(y)}{\textup{cap}_{X}(x,\,A^{c})}\;,\quad x\;,y\,\in\,A\;,
\end{equation*}
and $b_{x,\,y}^{A}=0$ otherwise. By elementary properties of the
capacity and the equilibrium potential, $b_{x,\,y}^{A}=b_{y,\,x}^{A}$
for all $x,\,y\in S$ (cf. \cite[Lemma 10.2]{LMS}). Moreover, by
\cite[Lemma 10.3]{LMS},
\begin{equation}
b_{x,\,y}^{A}\le b_{x,\,y}^{B}\;\;\text{for all }x,\,y\in S\label{increasing}
\end{equation}
if $A\subset B\subset S_{0}$.

For each non-empty subset $A$ of $S_{0}$, define the quadratic function
$P^{A}(\cdot)$ as
\begin{equation*}
P^{A}(\eta)\;=\;\frac{1}{2}\,\sum_{x\in A}b_{x,\,x}^{A}\,\eta_{x}\,(\eta_{x}-1)\;+\;\sum_{\{x,\,y\}\subset A}b_{x,\,y}^{A}\,\eta_{x}\,\eta_{y}\;.
\end{equation*}
By \cite[Lemma 10.8]{LMS},
\begin{equation}
c_{1}\,\Big(\sum_{x\in S_{0}}\eta_{x}\Big)^{2}\;\le\;P^{S_{0}}(\eta)\;\le\;c_{2}\,\Big(\sum_{x\in S_{0}}\eta_{x}\Big)^{2}\;.\label{LMSLem10.8}
\end{equation}

Fix $A\subsetneq S_{0}$. For each constant $c_{A}>0$ and positive
integer $\ell\ge1$, let $P_{\ell}^{A}:\mc{U}_{N}^{x_{0}}\rightarrow\bb{R}$
be given by
\begin{equation*}
P_{\ell}^{A}(\eta)\;=\;P^{A}(\eta)\;-\;c_{A}\,\ell^{2}\;.
\end{equation*}
The dependence of $P_{\ell}^{A}$ on the constant $c_{A}$ is omitted
from the notation. Taking $P_{\ell}^{\varnothing}(\eta)=0$ for all
$\eta\in\mc{U}_{N}^{x_{0}}$, define the corrector function $W_{\ell}:\mc{U}_{N}^{x_{0}}\rightarrow\bb{R}$
by
\begin{equation*}
W_{\ell}(\eta)=\min\left\{ P_{\ell}^{A}(\eta):A\subset S_{0}\,,\,A\not=S_{0}\,\right\} \;.
\end{equation*}
By \cite[Lemma 10.10]{LMS}, there exists a constant $0<C<\infty$
such that
\begin{equation*}
-\,C\,\ell^{2}\le W_{\ell}(\eta)\le0\;.
\end{equation*}
Hence, by \eqref{LMSLem10.8} and the previous bound, $P^{S_{0}}(\eta)-W_{\ell}(\eta)>0$
for all $\eta\in\mc{U}_{N}^{x_{0}}$.

For each positive integer $m>2$, define the function $G_{N}^{x_{0}}:\mc{H}_{N}\rightarrow\bb{R}$
by
\begin{equation}
G_{N}^{x_{0}}(\eta)\;=\;\begin{cases}
{\displaystyle \sum_{\ell=2}^{m}\frac{1}{\ell}\,[\,P^{S_{0}}(\eta)\,-\,W_{\ell}(\eta)\,]^{1/2}\;,} & \eta\,\in\,\mc{U}_{N}^{x_{0}}\;,\\
\vphantom{\Big\{}0\;, & \zeta\,\in\,\mc{H}_{N}\setminus\mc{U}_{N}^{x_{0}}\;.
\end{cases}\label{gnx}
\end{equation}
Here, again, the dependence of $G_{N}^{x_{0}}$ on $m$ is omitted.  The
next result is \cite[Theorem 9.2]{LMS}. Recall that
$\ms{L}_{N}^{\xi} = \theta_{N} \ms{L}_{N}$, introduced in Subsection
\ref{sec23}, is the generator of the speeded-up process.

\begin{thm}
\label{tlms1} For large enough $m$ and a suitable selection of constants
$(c_{A})_{A\subsetneq S_{0}}$, the function $G_{N}^{x_{0}}$ is super-harmonic
in $\mc{W}_{N}^{x_{0}}\setminus\mc{D}_{N}^{x_{0}}$. More
precisely, there exists a positive constant $C>0$ such that
\begin{equation*}
(\ms{L}_{N}^{\xi}\,G_{N}^{x_{0}})(\eta)\;\le\;-\frac{C\,\theta_{N}}{N-\eta_{x_{0}}}\text{\;\;\;\;for all }\eta\,\in\,\mc{W}_{N}^{x_{0}}\setminus\mc{D}_{N}^{x_{0}}\;.
\end{equation*}
Additionally, there exist constants $0<c_{1}<c_{2}<\infty$ such that
\begin{equation}
c_{1}\,(N-\eta_{x_{0}})\;\le\;G_{N}^{x_{0}}(\eta)\;\le\;c_{2}\,(N-\eta_{x_{0}})\label{boundG}
\end{equation}
for all $\eta\,\in\,\mc{W}_{N}^{x_{0}}\setminus\mc{D}_{N}^{x_{0}}$.
\end{thm}

\subsection{Reflected processes}

Denote by ${\color{blue}(\widehat{\xi}_{N}^{x}(t))_{t\ge0}}$ the
continuous-time Markov chain on $\mc{W}_{N}^{x}$ obtained by
reflecting the zero-range process $\xi_{N}(\cdot)$ at the boundary
of this set. In other words, the process $\widehat{\xi}_{N}^{x}(\cdot)$
behaves as the zero-range process inside the well $\mc{W}_{N}^{x}$,
but its jumps to the set $(\mc{W}_{N}^{x})^{c}$ are suppressed.
Denote by \textcolor{blue}{$\widehat{\mb{P}}_{\eta}^{N,\,x}$}
the law of the reflected process $\widehat{\xi}_{N}^{x}(\cdot)$,
and by \textcolor{blue}{$\widehat{\mb{E}}_{\eta}^{N,\,x}$}
the expectation with respect to $\widehat{\mb{P}}_{\eta}^{N,\,x}$.

The next result asserts that the function $G_{N}^{x_{0}}$ is also super-harmonic
at $\mc{W}_{N}^{x_{0}}\setminus\mc{D}_{N}^{x_{0}}$ for
the reflected process. Denote by ${\color{blue}\ms{L}_{N}^{x_{0}}}$
the generator associated to the reflected process $\widehat{\xi}_{N}^{x_{0}}(\cdot)$.
\begin{lem}
\label{lem_superh} Fix $x_{0}\in S$, and let $G_{N}^{x_{0}}$ be
the function given by \eqref{gnx}. Then, there exists $C>0$ such
that
\begin{equation*}
(\ms{L}_{N}^{x_{0}}\,G_{N}^{x_{0}})(\eta)\;\le\;-\frac{C\,\theta_{N}}{N-\eta_{x_{0}}}\text{\;\;\;\;for all }\eta\,\in\,\mc{W}_{N}^{x_{0}}\setminus\mc{D}_{N}^{x_{0}}\;.
\end{equation*}
\end{lem}

The main difference between this lemma and Theorem \ref{tlms1} is
the analysis around the boundary of $\mc{W}_{N}^{x_{0}}$, since
the generator $\ms{L}_{N}^{x_{0}}$ differs from $\ms{L}_{N}$
there, as the jumps to set $(\mc{W}_{N}^{x_{0}})^{c}$ are excluded.
\begin{proof}
Since we possibly have $(\ms{L}_{N}^{x_{0}}G_{N}^{x_{0}})(\eta)\neq(\ms{L}_{N}G_{N}^{x_{0}})(\eta)$
only at the boundary ${\color{blue}\partial\mc{W}_{N}^{x_{0}}}=\{\eta:\eta_{x_{0}}=N-m_{N}\}$,
it suffices to show that
\begin{equation*}
({\ms{L}}_{N}^{x_{0}}\,G_{N}^{x_{0}})(\eta)\;\le\;({\ms{L}}_{N}^{\xi}\,G_{N}^{x_{0}})(\eta)\text{\;\;\;\; for all }\eta\in\partial\mc{W}_{N}^{x_{0}}\;.
\end{equation*}
At $\partial\mc{W}_{N}^{x_{0}}$, the reflected process cannot
decrease the number of particles at site $x_{0}$. Thus,
\begin{equation*}
({\ms{L}}_{N}^{x_{0}}\,G_{N}^{x_{0}})(\eta)\;=\;({\ms{L}}_{N}^{\xi}\,G_{N}^{x_{0}})(\eta)\,-\,\theta_{N}\,\sum_{y\in S}g(\eta_{x_{0}})\,r(x_{0},\,y)\,[\,G_{N}^{x_{0}}(\sigma^{x_{0},\,y}\eta)\,-\,G_{N}^{x_{0}}(\eta)\,]\;,
\end{equation*}
and it is enough to show that
\begin{equation*}
G_{N}^{x_{0}}(\sigma^{x_{0},\,y}\eta)\;\ge\;G_{N}^{x_{0}}(\eta)\text{ \;\;\;\;for all }\eta\,\in\,\partial\mc{W}_{N}^{x_{0}}\;.
\end{equation*}
Actually, by the definition \eqref{gnx} of $G_{N}^{x_{0}}$, it is
enough to show that
\begin{equation}
P^{S_{0}}(\sigma^{x_{0},\,y}\eta)\,-\,W_{\ell}(\sigma^{x_{0},\,y}\eta)\;\ge\;P^{S_{0}}(\eta)\,-\,W_{\ell}(\eta)\label{e_39}
\end{equation}
for all $\ell\ge2$ and $\eta\,\in\,\partial\mc{W}_{N}^{x_{0}}$.

Fix $A\subsetneq S_{0}$, by definitions of $P^{S_{0}}$ and $P^{A}$
along with the increasing property \eqref{increasing},
\begin{equation}
P^{S_{0}}(\sigma^{x_{0},\,y}\eta)\,-\,P^{S_{0}}(\eta)\;=\;\sum_{z\in S_{0}}b_{y,\,z}^{S_{0}}\,\eta_{z}\;\ge\;\sum_{z\in A}b_{y,\,z}^{A}\,\eta_{z}\;=\;P^{A}(\sigma^{x_{0},\,y}\eta)\,-\,P^{A}(\eta)\;.\label{e_310}
\end{equation}
Hence, if $W_{\ell}(\eta)=P_{\ell}^{A}(\eta)$ and $W_{\ell}(\sigma^{x,\,y}\eta)=P_{\ell}^{A}(\sigma^{x,\,y}\eta)$
for the same set $A\subsetneq S\setminus\{x_{0}\}$, (\ref{e_39})
follows from (\ref{e_310}).

On the other hand, if $W_{\ell}(\eta)=P_{\ell}^{A}(\eta)$ and $W_{\ell}(\sigma^{x,\,y}\eta)=P_{\ell}^{B}(\sigma^{x,\,y}\eta)$
for some $A\neq B$, by definition of $W_{\ell}$ and (\ref{e_310}),
\begin{align*}
W_{\ell}(\sigma^{x,\,y}\eta)\,-\,W_{\ell}(\eta) & \;=\;P_{\ell}^{B}(\sigma^{x,y}\eta)\,-\,P_{\ell}^{A}(\eta)\\
 & \le\;P_{\ell}^{A}(\sigma^{x,y}\eta)\,-\,P_{\ell}^{A}(\eta)\;\le\;P(\sigma^{x,y}\eta)\,-\,P(\eta)\;.
\end{align*}
This completes the proof of (\ref{e_39}), and the one of the lemma.
\end{proof}

\subsection{Hitting times of the reflected process\label{sec34}}

In this subsection, we establish, in Lemma \ref{p_mix2a} below, that
the assertion of Proposition \ref{p_mixmain} holds for the reflected
process $\widehat{\xi}_{N}^{x}(\cdot)$. The first result asserts
that the process $\widehat{\xi}_{N}^{x}(\cdot)$ hits the set $\mc{D}_{N}^{x}$
quickly when it starts from a configuration in $\mc{E}_{N}^{x}$.
\begin{lem}
\label{lem38}
There exists $C>0$ such that, for all $x\in S$, $N\ge1$,
\begin{equation*}
\sup_{\eta\in\mc{E}_{N}^{x}}\widehat{\mb{E}}_{\eta}^{N,\,x}\,[\,\tau_{\mc{D}_{N}^{x}}\,]\;\le\;C\,\frac{m_{N}\,\ell_{N}}{\theta_{N}}\;\cdot
\end{equation*}
\end{lem}

\begin{proof}
By the martingale formulation, for every $t>0$,
\begin{equation*}
\widehat{\mb{E}}_{\eta}^{N,\,x}\,[\,G_{N}^{x}(\,\widehat{\eta}_{N}^{x}(\tau_{\mc{D}_{N}^{x}}\wedge t)\,)\,]\;=\;G_{N}^{x}(\eta)\,+\,\widehat{\mb{E}}_{\eta}^{N,\,x}\,\Big[\,\int_{0}^{\tau_{\mc{D}_{N}^{x}}\wedge t}(\ms{L}_{N}^{x}G_{N}^{x})(\widehat{\eta}_{N}^{x}(s))ds\,\Big]\;.
\end{equation*}
By Lemma \ref{lem_superh}, there exists a positive constant
$C$, whose value may change from line to line, such that
\begin{equation*}
(\ms{L}_{N}^{x}\,G_{N}^{x})(\eta)\;\le
\;-\,\frac{C\,\theta_{N}}{N-\eta_{x}}\;\le
\;-\,\frac{C\,\theta_{N}}{m_{N}}\;\;\;\;
\text{for }\eta\,\in\,\mc{W}_{N}^{x}\;.
\end{equation*}
On the other hand, by \eqref{boundG}, $G_{N}^{x}$ is
non-negative. Therefore, by the next to last displayed equation,
\begin{equation*}
\frac{C\,\theta_{N}}{m_{N}}\,\widehat{\mb{E}}_{\eta}^{N,\,x}\,[\,\tau_{\mc{D}_{N}^{x}}\wedge t\,]\;\le\;G_{N}^{x}(\eta)\;.
\end{equation*}
By \eqref{boundG}, there exist a finite constant $C_{1}$ such that
$G_{N}^{x}(\eta)\le C_{1}\,(N-\eta_{x})$. Hence, since $N-\eta_{x}\le\ell_{N}$
for $\eta\in\mc{E}_{N}^{x}$,
\begin{equation*}
\frac{C\,\theta_{N}}{m_{N}}\,\widehat{\mb{E}}_{\eta}^{N,\,x}\,[\,\tau_{\mc{D}_{N}^{x}}\wedge t\,]\;\le\;C_{1}\,\ell_{N}\;.
\end{equation*}
To complete the proof of the lemma, it remains to let $t\to\infty$.
\end{proof}
The next result asserts that $\widehat{\xi}_{N}^{x}(\cdot)$ hits the
configuration $\zeta_N^x$ quickly when it starts from a configuration
in $\mc{D}_{N}^{x}$.
\begin{lem}
\label{lem_hitbot} There exists a finite constant $C$ such that,
\begin{equation*}
\sup_{\eta\in\mc{D}_{N}^{x}}\,\widehat{\mb{E}}_{\eta}^{N,\,x}\,[\,\tau_{\zeta_N^x}\,]\;\le\;C\,\frac{N^{\gamma\kappa}\,(\log N)^{\kappa-1}}{\theta_{N}}
\end{equation*}
for all $x\in S$ and $N\ge1$.
\end{lem}

\begin{proof}
If $\eta=\zeta_N^x$, there is nothing to prove. For
$\eta\neq\zeta_N^x$, we recall the well-known identity
(cf. \cite[Proposition 6.10]{BL1})
\begin{equation}
\widehat{\mb{E}}_{\eta}^{N,\,x}\,[\,\tau_{\zeta_N^x}\,]\;=\;\frac{\bb{E}_{\widehat{\mu}_{N}^{x}}\,[\,\mf{h}_{\eta,\,\zeta_N^x}^{N,\,x}\,]}{\textrm{cap}_{N}^{x}(\eta,\,\zeta_N^x)}\;,\label{e_mf}
\end{equation}
where $\mf{h}_{\eta,\,\zeta_N^x}^{N,\,x}$ and
$\textrm{cap}_{N}^{x}(\eta,\,\zeta_N^x)$ denote the equilibrium
potential and the capacity between $\eta$ and $\zeta_N^x$ with respect
to the reflected process $\widehat{\xi}_{N}^{x}(\cdot)$, respectively,
and where $\widehat{\mu}_{N}^{x}(\cdot)$ denotes the invariant measure
conditioned on $\mc{W}_{N}^{x}$, i.e.,
\begin{equation*}
\widehat{\mu}_{N}^{x}(\cdot)\;=\;\mu_{N}(\cdot\,|\mc{W}_{N}^{x})\;=\;\frac{\mu_{N}(\cdot)}{\mu_{N}(\mc{W}_{N}^{x})}\;.
\end{equation*}
Observe that $\widehat{\mu}_{N}^{x}(\cdot)$ is the invariant measure
of the reflected process $\widehat{\xi}_{N}^{x}(\cdot)$.

Applying the trivial bound $\mf{h}_{\eta,\,\zeta_N^x}^{N,\,x}\le1$
to \eqref{e_mf}, we get
\begin{equation*}
\widehat{\mb{E}}_{\eta}^{N,\,x}\,[\,\tau_{\zeta_N^x}\,]\;\le\;\frac{1}{\textrm{cap}_{N}^{x}(\eta,\,\zeta_N^x)}\;\cdot
\end{equation*}
By \cite[Lemma 9.3]{LMS},
\begin{equation*}
\textrm{cap}_{N}^{x}(\eta,\,\zeta_N^x)\;\ge\;\frac{C\,\theta_{N}}{N^{\gamma\kappa}\,(\log N)^{\kappa-1}}\;\cdot
\end{equation*}
Actually, in \cite[Lemma 9.3]{LMS} this bound is proved for the capacity
with respect to the original zero-range process, but the same proof
applies to the reflected process. To complete the proof, it remains
to combine the previous bounds.
\end{proof}
\begin{lem}
\label{p_mix2a} For all $x\in S$,
\begin{equation*}
\lim_{N\to\infty}\sup_{\eta\in\mc{E}_{N}^{x}}\widehat{\mb{P}}_{\eta}^{N,\,x}\,[\,\tau_{\zeta_N^x}\ge\mb{u}_{N}\,]\;=\;0\;.
\end{equation*}
\end{lem}

\begin{proof}
By Lemmata \ref{lem38}, \ref{lem_hitbot}, and the strong Markov
property,
\begin{equation*}
\widehat{\mb{E}}_{\eta}^{N,\,x}\,[\,\tau_{\zeta_N^x}\,]\;\le\;\frac{C}{\theta_{N}}\,\Big\{\,m_{N}\,\ell_{N}+C\,N^{\gamma\kappa}\,(\log N)^{\kappa-1}\Big\}\;\ll\;\frac{m_{N}^{2}}{\theta_{N}}\;,
\end{equation*}
since we assumed that $\gamma<2/\kappa$. The assertion of the lemma
follows from Chebyshev inequality.
\end{proof}

\subsection{Proof of Proposition \ref{p_mixmain}\label{sec35}}

Consider the canonical coupling of the zero-range process $\xi_{N}(\cdot)$
and the reflected process $\widehat{\xi}_{N}^{x}(\cdot)$ starting
together at $\eta\in\mc{W}_{N}^{x}$. The two processes move
together until $\xi_{N}(\cdot)$ hits $(\mc{W}_{N}^{x})^{c}$.
From this point on, they move independently according to their respective
dynamics. By Proposition \ref{p_mix1}, starting from $\mc{E}_{N}^{x}$,
we can couple the original zero-range process and the reflected process
$\widehat{\xi}_{N}^{x}(\cdot)$ up to time $\mb{h}_{N}$ with
a probability close to $1$.

The joint law of $\xi_{N}(\cdot)$ and $\widehat{\xi}_{N}^{x}(\cdot)$
under this canonical coupling is represented by \textcolor{blue}{
$\widehat{\bf{P}}_{\eta}^{N,\,x}$}. Denote by $\tau_{\mc{A}}$
and $\widehat{\tau}_{\mc{A}}$ the hitting time of a set $\mc{A}$
with respect to $\xi_{N}(\cdot)$ and $\widehat{\xi}_{N}^{x}(\cdot)$,
respectively.

\begin{proof}[Proof of Proposition \ref{p_mixmain}]
Recall the definition of the sequence $\mb{u}_{N}$ introduced at the
beginning of Section \ref{sec10}, and the one of $\mb{h}_{N}$
presented in \eqref{h_n}. Fix $\eta\in\mc{E}_{N}^{x}$. By Proposition
\ref{p_mix1},
\begin{equation*}
\mb{P}_{\eta}^{N}\,[\,\tau_{\zeta_N^x}\,\ge\,\mb{u}_{N}\,]\;=\;\mb{P}_{\eta}^{N}\,[\,\tau_{\zeta_N^x}\,\ge\,\mb{u}_{N}\,,\,\,\tau_{(\mc{W}_{N}^{x})^{c}}\,>\,\mb{h}_{N}\,]\,+\,o_{N}(1)\;.
\end{equation*}

Recall the canonical coupling introduced above. On the event
$\{\tau_{(\mc{W}_{N}^{x})^{c}}>\mb{h}_{N}\}$, the two processes
$\xi_{N}(t)$ and $\widehat{\xi}_{N}^{x}(t)$ move together until
$\mb{h}_{N}$. Since $\mb{u}_{N}\ll\mb{h}_{N}$, on the previous event,
the sets $\{\tau_{\zeta_N^x}\ge\mb{u}_{N}\}$ and
$\{\widehat{\tau}_{\zeta_N^x}\ge\mb{u}_{N}\}$ coincide.  Thus,
\begin{align*}
\mb{P}_{\eta}^{N}\,[\,\tau_{\zeta_N^x}\,\ge\,\mb{u}_{N}\,,\,\,\tau_{(\mc{W}_{N}^{x})^{c}}\,>\,\mb{h}_{N}\,]\;=\; & \widehat{\bf{P}}_{\eta}^{N,\,x}\,[\,\tau_{\zeta_N^x}\,\ge\,\mb{u}_{N}\,,\,\,\tau_{(\mc{W}_{N}^{x})^{c}}\,>\,\mb{h}_{N}\,]\\
=\; & \widehat{\bf{P}}_{\eta}^{N,\,x}\,[\,\widehat{\tau}_{\zeta_N^x}\,\ge\,\mb{u}_{N},\,\tau_{(\mc{W}_{N}^{x})^{c}}\,>\,\mb{h}_{N}\,]\;.
\end{align*}

Since,
\begin{equation*}
\widehat{\bf{P}}_{\eta}^{N,\,x}[\widehat{\tau}_{\zeta_N^x}\,\ge\,\mb{u}_{N}\,,\,\,\tau_{(\mc{W}_{N}^{x})^{c}}\,>\,\mb{h}_{N}]\;\le
\;\widehat{\bf{P}}_{\eta}^{N,\,x}[\widehat{\tau}_{\zeta_N^x}\,\ge\,\mb{u}_{N}]\;=\;\widehat{\mb{P}}_{\eta}^{N,\,x}[\tau_{\zeta_N^x}\ge\mb{u}_{N}]\;,
\end{equation*}
by Lemma \ref{p_mix2a}, this quantity vanishes as $N\to\infty$.
It remains to combine the previous estimates.
\end{proof}

\section{Condition $\mf{M}$ for critical zero-range processes}
\label{sec4}

In this section, we prove condition \eqref{33} for a time-scale
$\mb{s}_{N} \ll \mb h_N$ and the large wells $\mc V^x_N$ introduced in
\eqref{large_well}.  For $N\ge1$, define
\begin{equation}\label{s_N}
\mb{s}_{N}\;=\;(\,1+(\log N)^{1/4}\,)\,\mb{u}_{N}\;.
\end{equation}
Note that $\mathbf{s}_N \ll \mathbf{h}_N$.

Recall from Subsection \ref{sec61} and equation \eqref{dtv} the
definitions of the reflected process $\xi^{R, x}_N(\cdot)$ and of the
total variation distance $d_{\rm TV}^x (\cdot,\,\cdot)$. For $t\ge 0$,
let
\begin{equation*}
D^x_{\rm TV}(t) \;:=\; \sup_{\eta\in\mc V^x_N} d_{\rm TV}^x
(\,\delta_\eta \ms{P}_N^{R,x} (t)\,,\, \,\pi^{R,x}\,)\;.
\end{equation*}
Note here that we prove a stronger version of mixing than the one required in the condition $\mf{M}$ since the supremum in the definition $D^x_{\rm TV}(t)$ is taken over all configurations in $\mc{V}_N^x$.

\begin{prop}
\label{pro_mix2}For all $x\in S$,
\begin{equation*}
\lim_{N\rightarrow\infty}D^x_{\rm TV}(\mb{s}_{N})=0\;.
\end{equation*}
\end{prop}

It follows from this result that for all $\epsilon>0$ the mixing time
$t^x_{\rm mix}(\epsilon)$ is bounded by $\mb s_N$ for $N$ sufficiently
large. In particular, condition \eqref{33} holds because
$\mb{s_N}\ll \mb{h_N}$.

\begin{cor}
\label{corM}
The condition $\mf{M}$ holds for the critical zero-range processes.
\end{cor}
\begin{proof}
The proof follows from Corollary \ref{cor_m1}, Proposition
\ref{pro_mix2}, and the fact that $\mb{s_N}\ll \mb{h_N}$.
\end{proof}

The proof of Proposition \ref{pro_mix2} is divided into several
steps. We first show that the process $\xi^{R,x}_N(\cdot)$ hits the
configuration $\zeta_N^x$ in the time-scale $\mb{u}_{N}$. The
reasoning carried out in the proof of Lemma \ref{p_mix2a} does not
apply to the process $\xi^{R,x}_N(\cdot)$ because Lemma
\ref{lem_superh} does not hold for it. We present below an alternative
argument, based on Propositions \ref{p_mix1} and \ref{p_mixmain}.

Recall from Subsection \ref{sec35} the definition of the canonical
coupling of the zero-range process $\xi_{N}(\cdot)$ and the reflected
process $\widehat{\xi}_{N}^{x}(\cdot)$. The same definition permits to
couple $\xi_{N}(\cdot)$ and $\xi^{R,x}_N(\cdot)$. Denote by
$\color{blue} \widehat{\bf{P}}_{\eta}^{R,\,x}$ the joint law of
$\xi_{N}(\cdot)$ and ${\xi}_{N}^{R,x}(\cdot)$ under the canonical
coupling.

\begin{lem}
\label{lem_mix1}
For all $x\in S$,
\begin{equation*}
\lim_{N\rightarrow\infty}\sup_{\eta\in\mc{E}_{N}^{x}}
\mb P^{R,x}_\eta\,[\,\tau_{\zeta_N^x}
\,\ge\,\mb{u}_{N}\,]\;=\;0\;.
\end{equation*}
\end{lem}

\begin{proof}
By Proposition \ref{p_mixmain},
\begin{equation*}
o_{N}(1)\;=\;\mb{P}_{\eta}^{N}\,[\,\tau_{\zeta_N^x}\,\ge\,\mb{u}_{N}\,]
\;\ge\;\mb{P}_{\eta}^{N}\,[\,\tau_{\zeta_N^x}\,\ge\,\mb{u}_{N},
\,\tau_{(\mc{V}_N^x)^{c}}\,>\,\mb{h}_{N}\,]\;.
\end{equation*}
Let $\tau^R_{\zeta_N^x}$ stand for the hitting time of the
configuration $\zeta_N^x$ with respect to the reflected process
$\xi^{R,x}_N(\cdot)$. Replace the probability measure on the
right-hand side of the previous equation by the coupling measure
$\widehat{\bf P}^{R,x}_\eta$.  Since $\mb{h}_{N}\gg\mb{u}_{N}$,
\begin{equation*}
\widehat{\bf P}^{R,x}_\eta\,[\,\tau_{\zeta_N^x}
\,\ge\,\mb{u}_{N},\,\tau_{(\mc{V}_N^x)^{c}}\,>\,\mb{h}_{N}\,]
\;=\;\widehat{\bf P}^{R,x}_\eta\,[\, {\tau}^R_{\zeta_N^x}
\,\ge\,\mb{u}_{N},\,\tau_{(\mc{V}_N^x)^{c}}\,>\,\mb{h}_{N}\,]
\end{equation*}
By Proposition \ref{p_mix1}, the right-hand side is equal to
\begin{equation*}
\widehat{\bf P}^{R,x}_\eta\,[\,\tau^R_{\zeta_N^x}
\ge\mb{u}_{N}]\,-\,o_{N}(1)\;=\;
\mb P^{R,x}_\eta\,[\,\tau^R_{\zeta_N^x}\,\ge\,\mb{u}_{N}\,]\,-\,o_{N}(1)\;,
\end{equation*}
as claimed.
\end{proof}

We next recall a bound on the spectral gap established in \cite{LMS}.

\begin{thm}
\label{t_gap}
There exists a constant $c_{0}>0$ such that the spectral
gap of the reflected process $\xi^{R,x}_N(\cdot)$ on
$\mc{V}_N^x$ is bounded below by $c_{0}\,\mb{s}_{N}^{-1}$
for all $N\ge1$.
\end{thm}

This result is \cite[Theorem 6.1]{LMS}. One just has to replace
$\ell_{N}$ by $m_{N}$ in the statement and in the proof of that
result. \smallskip{}

Recall from Section \ref{sec61} that $\pi^{R,x}$, $x\in S$, represents
the stationary state of the reflected process
$\xi^{R,x}_N(\cdot)$. Moreover, for $\eta\in\mc{V}_N^x$ and $t>0$, the
measure $\delta_\eta\, \ms P^{R,x}_N(t)$ on $\mc{V}_N^x$ stands for
the distribution of the reflected process $\xi^{R,x}_N(t)$ starting at
$\eta$. Let
\begin{equation*}
\pi_{N}^{x}(\cdot,\,t)\;=\;\delta_{\zeta_N^x}\, \ms P^{R,x}_N(t)\;.
\end{equation*}
The next result asserts that the reflected process $\xi^{R,x}_N(\cdot)$
starting from $\zeta_N^x$ mixes in the time scale $(\log N)^{1/4}\mb{u}_{N}$.

\begin{lem}
\label{lem_mix2}
There exist two constants $C_{1},\,C_{2}>0$ such
that, for all $x\in S$ and $t\ge(\log N)^{1/4}\mb{u}_{N}$,
\begin{equation*}
d_{\rm TV}^x (\, \pi_{N}^{x}(\cdot,\,t) \,,\,\,
\pi^{R,x}  \,)
\;\le\;C_{1}e^{-C_{2}(\log N)^{1/8}}\;.
\end{equation*}
\end{lem}

\begin{proof}
By the Cauchy-Schwarz inequality,
\begin{equation*}
d_{\rm TV}^x (\, \pi_{N}^{x}(\cdot,\,t) \,,\,\,\pi^{R,x}\,)^{2}
\;\le\;\frac{1}{4}\sum_{\eta\in\mc{V}_N^x}
\Big\{\,\frac{\pi_{N}^{x}(\eta,\,t)}{\pi^{R,x}(\eta)}\,-\,1\,\Big\}^{2}
\,\pi^{R,x}(\eta)\;.
\end{equation*}
Since the process $\xi_{N}(\cdot)$ is reversible, the conditioned
measure $\mu_{N}(\,\cdot\,|\, \mc V^x_N)$ is the stationary measure
for the reflected process $\xi^{R,x}_N(\cdot)$. Hence, by the standard
$L^{2}$-contraction inequality (cf. \cite[Lemma 20.5]{LPW}) and
Theorem \ref{t_gap}, the summation at the right-hand side, which is
actually the square of $L^{2}$-distance between
$\pi_{N}^{x}(\cdot,\,t)$ and $\pi^{R,x}(\cdot)$, is less than or equal
to
\begin{equation*}
e^{-c_{0}(t/\mb{u}_{N})}\,\sum_{\eta\in\mc{V}_N^x}\,
\Big\{\,\frac{\pi_{N}^{x}(\eta,\,0)}{\pi^{R,x}(\eta)}\,-\,1\,\Big\}^{2}
\,\pi^{R,x}(\eta)
\end{equation*}
for some constant $c_{0}>0$ independent of $N$. As
$t\ge(\log N)^{1/4}\mb{u}_{N}$ and
$\pi_{N}^{x}(\eta,\,0)=\mb{1}\{\eta=\zeta_N^x\}$, this expression is
bounded by
\begin{equation*}
e^{-c_{0}\,(\log N)^{1/4}}\,\Big(\,
\frac{1}{\pi^{R,x}(\zeta_N^x)}\,-\,1\,\Big)\;,
\end{equation*}

By the explicit formula for the invariant measure $\mu_{N}$,
\begin{equation*}
\pi^{R,x}(\zeta_N^x)\;=\;\frac{1}{\mu_{N}(\mc{V}_N^x)}
\frac{N}{Z_{N,\kappa}(\log N)^{\kappa-1}}
\frac{1}{\mb{a}(\zeta_N^x)}\;\ge\;\frac{c_{1}}{(\log N)^{\kappa-1}}\;
\end{equation*}
for some constant $c_{1}>0$. Putting together the previous estimates
yields that
\begin{equation*}
d_{\rm TV}^x(\, \pi_{N}^{x}(\cdot,\,t) \,,\,\,\pi^{R,x} \,)
\;\le\;c_{1}^{-1}\,e^{-c_{0}(\log N)^{1/4}}(\log N)^{\kappa-1}
\ll c_{1}^{-1}\,e^{-(c_{0}/2)(\log N)^{1/4}}\;.
\end{equation*}
This completes the proof.
\end{proof}

\begin{proof}[Proof of Proposition \ref{pro_mix2}]
The proof relies on Lemmata \ref{lem_mix1} and \ref{lem_mix2}.  Fix
$x\in S$, $\eta\in\mc{E}_{N}^{x}$ and $\mc{A}\subset\mc{V}_{N}^x$.  By
Lemma \ref{lem_mix1}, we can write
\begin{equation}
\mb P^{R,x}_\eta[\,\xi^{R,x}_N(\mb{s}_{N})\in\mc{A}\,]
\;=\;\mb P^{R,x}_\eta[\,\xi^{R,x}_N(\mb{s}_{N})\in\mc{A}
\,|\,\tau^R_{\zeta_N^x}<\mb{u}_{N}\,]+o_{N}(1)\;,
\label{emx3}
\end{equation}
where the error term $o_{N}(1)$ at the right-hand side is bounded by
$2\mb P^{R,x}_\eta[\tau^R_{\zeta_N^x}\ge\mb{u}_{N}]$ and hence is
independent of $\mc{A}$.

Denote by $\alpha_{N}^{x}(t)dt$ the distribution of $\tau^R_{\zeta_N^x}$
conditioned on $\tau^R_{\zeta_N^x}<\mb{u}_{N}$. Then, by the
strong Markov property, we can write the probability at the right-hand
side as
\begin{equation}
\int_{0}^{\mb{u}_{N}} {\mb{P}}_{\zeta_N^x}^{R,\,x
}[\,\xi^{R,x}_N(\mb{s}_{N}-t)\in\mc{A}\,]\,\alpha_{N}^{x}(t)\,dt\;.
\label{emx2}
\end{equation}
Since $\mb{s}_{N}-t\ge(\log N)^{1/4}\mb{u}_{N}$ for all
$t\in[0,\,\mb{u}_{N}]$, by definition of $\mb{s}_{N}$, it follows from
Lemma \ref{lem_mix2} that
\begin{align}
|{\mb{P}}_{\zeta_N^x}^{R,\,x}[\,\xi^{R,x}_N(\mb{s}_{N}-t)
\in\mc{A}\,]\,-\,\pi^{R,x}(\mc{A})| \;
&\le\;
d_{\rm TV}^x(\,\pi_{N}^{x}(\cdot,\,\mb{s}_{N}-t) \,,\,\,
\pi^{R,x}\,) \nonumber \\
& \le\; C_{1}e^{-C_{2}(\log N)^{1/8}}\;,
\label{emx1}
\end{align}
where we used the fact that
\begin{equation}
d_{\rm TV}^x (\,\nu_{1},\,\nu_{2}\,)
\;=\;\sup_{\mc{A}\subset\mc{V}_{N}^x}|\,\nu_{1}(\mc{A})
\,-\,\nu_{2}(\mc{A})\,|
\label{tvn}
\end{equation}
for any probability measures $\nu_{1}$ and $\nu_{2}$ on $\mc{H}_{N}$.
By \eqref{emx1} and \eqref{emx2}, we can assert that the right-hand
side of \eqref{emx3} is $\pi^{R,x}(\mc{A})+o_{N}(1)$.
Thus,
\begin{equation*}
\mb P^{R,x}_\eta\,[\,\xi^{R,x}_N(\mb{s}_{N})\in\mc{A}\,]
\,-\,\pi^{R,x}(\mc{A})\;=\;o_{N}(1)\;,
\end{equation*}
where the error term is independent of $\mc{A}$. Therefore,
by \eqref{tvn}, we can conclude that
\begin{equation*}
D^x_{\rm TV}(\mb{s}_N )\;=\; \sup_{\mc A\subset \mc V_N^x}  | \,\mb P^{R,x}_\eta[\,\xi^{R,x}_N(\mb{s}_{N})\in\mc A\,]\,-\,
\pi^{R,x}(\mc A)\,|\;=\;o_{N}(1)\;,
\end{equation*}
as claimed.
\end{proof}

\section{Condition (H0) for critical zero-range processes}
\label{sec9}

In this section, we verify the condition (H0) by establishing the
following proposition. For each $A\subset S$, write
$$
\mathcal{E}_N(A)\;=\;\bigcup_{x\in A} \mc{E}_N^x\;.
$$

\begin{prop}
\label{p131}Fix a non-empty subset $S_{1}\subsetneq S$, and let
$S_{2}=S\setminus S_{1}$. Then, we have
\begin{equation*}
\lim\limits _{N\rightarrow\infty}\,{\rm {cap}}_{N}(\mc{E}_{N}(S_{1}),\,\mc{E}_{N}(S_{2}))\;=\;6\sum_{x\in S^{1},\,y\in S^{2}}{\rm cap}_{X}(x,\,y)\;.
\end{equation*}
\end{prop}

The proof is similar to that of \cite[Theorem 2.2]{BL3}. As carried
out in \cite{BL3}, we prove the lower and upper bounds separately.
The respective proofs is given in Sections \ref{sec91} and \ref{sec92}.
We prove several technical lemmata in Section \ref{sec93}.

Note that the zero-range dynamics that we are considering now is
reversible, and thus we can express the mean-jump rate $r_{N}(x,\,y)$
as in \eqref{mjrc}. Hence, the following is a immediate consequence of
Theorem \ref{t_cond} and Proposition \ref{p131}.

\begin{cor}
\label{corH0}
The critical zero-range processes satisfies condition (H0) with
\begin{equation*}
r(x,\,y)\;=\;6\,\kappa\,{\rm cap}_{X}(x,\,y)\,,\;\;\;\;x,\,y\in S\;.
\end{equation*}
\end{cor}

Now we turn to the proof of Proposition \ref{p131}.

\subsection{\label{sec91}Lower bound}

We start with a lower bound whose proof is a modification of
\cite[Proposition 4.1]{BL3}.  For the proof, we have to introduce a
notion of tube along which the metastable transition occurs. For
$x,\,y\in S,x\neq y$, define the tube $\mc{I}_{N}^{x,\,y}$ between
$\mc{E}_{N}^{x}$ and $\mc{E}_{N}^{y}$ as,
\begin{align*}
\mc{I}_{N}^{x,\,y}\;=\;\big\{\,\xi\in\mc{H}_{N-1}\,:\, & \xi_{x}\,+\,\xi_{y}\,\geq\,(N-1)\,-\,\ell_{N}/3\\
 & \text{\ensuremath{\qquad} and }\xi_{x},\,\xi_{y}\,\leq\,(N-1)\,-\,\ell_{N}\,\big\}\;.
\end{align*}
Then we can observe that
\begin{equation}
\mc{I}_{N}^{x,\,y}\,=\,\mc{I}_{N}^{y,\,x}\;\;\;\;\;\text{and \;\;\;\;}\mc{I}_{N}^{x,\,y}\,\cap\,\mc{I}_{N}^{z,\,w}\,=\,\emptyset\text{ if }\{x,y\}\,\neq\,\{z,w\}\label{e_I}
\end{equation}
for all large enough $N$. From now on, all the computations implicitly
assume that $N$ is large enough. This is legitimate since we will
send $N$ to $\infty$ in the end. The former one in \eqref{e_I}
is immediate from the definition. For the latter one, the statement
is obvious if $\{x,\,y\}\cap\{z,\,w\}=\emptyset$. To check the other
case, let us suppose that $\xi\in\mc{I}_{N}^{x,\,y}\cap\mc{I}_{N}^{x,\,w}$
for some $x,\,y,\,w\in S$. Then, we must have
\begin{align*}
2(N-1-\ell_{N}/3) & \;\leq\;\xi_{x}\,+\,(\xi_{y}+\xi_{x}+\xi_{w})\;\leq\;\xi_{x}+N-1
\end{align*}
and hence we get $\xi_{x}\geq N-1-2\ell_{N}/3$. This contradicts
with the condition $\xi_{x}\leq N-1-\ell_{N}$ of $\mc{I}_{N}^{x,\,y}$.
\begin{prop}
\label{p_lower} Fix a non-empty subset $S_{1}\subsetneq S$, and
let $S_{2}=S\setminus S_{1}$. Then,
\begin{equation*}
\liminf\limits _{N\rightarrow\infty}\,{\rm {cap}}_{N}(\mc{E}_{N}(S_{1}),\,\mc{E}_{N}(S_{2}))\;\geq\;6\sum_{x\in S^{1},\,y\in S^{2}}{\rm cap}_{X}(x,\,y)\;.
\end{equation*}
\end{prop}

\begin{proof}
Write $\mf{h}=h_{\mc{E}_{N}(S_{1}),\,\mc{E}_{N}(S_{2})}^{N}$
(cf. \eqref{03}) the equilibrium potential between $\mc{E}_{N}(S_{1})$
and $\mc{E}_{N}(S_{2})$ so that ${\rm {cap}}_{N}(\mc{E}_{N}(S_{1}),\,\mc{E}_{N}(S_{2}))=\ms{D}_{N}(\mf{h})$.

Let $\mf{o}_{x}\in\mc{H}_{1}$ denote the configuration
with one particle at site $x$ and no particle at the other sites.
We can write the Dirichlet form as
\begin{align*}
\ms{D}_{N}(\mf{h}) & \;=\;\frac{N\,\theta_{N}}{2\,(\log N){}^{\kappa-1}\,Z_{N,\,\kappa}}\,\sum_{\xi\in\mc{H}_{N-1}}\,\sum_{z,\,w\in S}\,\frac{r(z,\,w)}{\mb{a}(\xi)}\,[\,\mf{h}(\xi+\mf{o}_{w})\,-\,\mf{h}(\xi+\mf{o}_{z})\,]^{2}\;.
\end{align*}
Thus, by \eqref{e_I}, we can bound $\ms{D}_{N}(\mf{h})$
from below by
\begin{equation}
\frac{N\,\theta_{N}}{2\,(\log N){}^{\kappa-1}\,Z_{N,\,\kappa}}\sum_{x\in S_{1},\,y\in S_{2}}\,\sum_{\xi\in\mc{I}_{N}^{x,\,y}}\,\sum_{z,\,w\in S}\,\frac{r(z,\,w)}{\mb{a}(\xi)}\,[\,\mf{h}(\xi+\mf{o}_{w})\,-\,\mf{h}(\xi+\mf{o}_{z})\,]^{2}\label{e_93}
\end{equation}

For $x\in S_{1}$ and $y\in S_{2}$, fix a configurations $\xi\in\mc{I}_{N}^{x,\,y}$
such that $\mf{h}(\xi+\mf{o}_{x})\neq\mf{h}(\xi+\mf{o}_{y})$.
Define a function $f:S\rightarrow\bb{R}$ as
\begin{equation*}
f(v)\;=\;\frac{\mf{h}(\xi+\mf{o}_{v})\,-\,\mf{h}(\xi+\mf{o}_{y})}{\mf{h}(\xi+\mf{o}_{x})\,-\,\mf{h}(\xi+\mf{o}_{y})}\;\;\;\;;\;v\in S\;.
\end{equation*}
Since $f(x)=1$ and $f(y)=0$, we can apply the Dirichlet principle
for the underlying random walk to get (cf. \eqref{e_capX} and \eqref{e_DirX})
\begin{align*}
 & \frac{1}{2}\sum_{z,\,w\in S}r(z,\,w)\,[\,\mf{h}(\xi+\mf{o}_{w})\,-\,\mf{h}(\xi+\mf{o}_{z})\,]^{2}\\
\;=\; & \kappa\,D_{X}(f)\,[\,\mf{h}(\xi+\mf{o}_{x})\,-\mf{\,h}(\xi+\mf{o}_{y})\,]^{2}\;\geq\;\kappa\,{\rm cap}_{X}(x,\,y)\,[\,\mf{h}(\xi+\mf{o}_{x})\,-\,\mf{h}(\xi+\mf{o}_{y})\,]^{2}\;.
\end{align*}
The same inequality obviously holds when $\mf{h}(\xi+\mf{o}_{x})=\mf{h}(\xi+\mf{o}_{y})$.
Hence, we can bound the summation at \eqref{e_93} from below by
\begin{equation}
\sum_{x\in S_{1},\,y\in S_{2}}\,\Big[\,\kappa\,{\rm cap}_{X}(x,\,y)\sum_{\xi\in\mc{I}_{N}^{x,\,y}}\frac{1}{\,\mb{a}(\xi)}\,[\,\mf{h}(\xi+\mf{o}_{x})-\mf{h}(\xi+\mf{o}_{y})\,]^{2}\,\Big]\;.\label{e_94}
\end{equation}

Fix $x_{0}\in S_{1}$ and $y_{0}\in S_{2}$ and denote $S_{0}=S\setminus\{x_{0},\,y_{0}\}$.
For each $\zeta\in\mc{H}_{k,\,S_{0}}$ (cf. \eqref{e_Hn}) with
$k\le\ell_{N}/3$, let $G_{\zeta}:\{0,\,\dots,\,N-1-k\}\rightarrow\bb{R}$
be the function defined by $G_{\zeta}(i)=\mf{h}(\xi)$ in which
$\xi$ is the configuration in $\mc{H}_{N}$ given by $\xi_{v}=\zeta_{v}$
for $v\in S_{0}$, $\xi_{x_{0}}=i$, and $\xi_{y_{0}}=N-k-i$. Then,
we can rewrite the second sum of \eqref{e_94} as
\begin{equation*}
\,\sum_{k=0}^{\ell_{N}/3}\,\sum_{\zeta\in\mc{H}_{k,\,S_{0}}}\,\Big[\,\frac{1}{\mb{a}(\zeta)}\,\sum_{i=\ell_{N}-k}^{N-\ell_{N}-1}\,\frac{1}{i\,(N-k-1-i)}\,[\,G_{\zeta}(i+1)\,-\,G_{\zeta}(i)\,]^{2}\,\Big]\;.
\end{equation*}
By the Cauchy-Schwarz inequality, the third sum above is bounded from
below by
\begin{equation}
\Big[\,\sum_{i=\ell_{N}-k}^{N-\ell_{N}-1}i\,(N-k-1-i)\,\Big]^{-1}\,[\,G_{\zeta}(N-\ell_{N})-G_{\zeta}(\ell_{N}-k)\,]^{2}\;.\label{e_95}
\end{equation}
By an elementary estimate
\begin{equation*}
\lim\limits _{N\to\infty}N^{3}\,\Big[\,\sum_{i=\ell_{N}-k}^{N-\ell_{N}-1}i\,(N-k-1-i)\,\Big]^{-1}\;=\;\Big[\,\int_{0}^{1}\,u(1-u)\,du\,\Big]^{-1}\;=\;6
\end{equation*}
and by the fact that $G_{\zeta}(N-\ell_{N})=1$ and $G_{\zeta}(\ell_{N}-k)=0$,
we can assert that \eqref{e_95} is $\frac{6+o_{N}(1)}{N^{3}}$. Summing
up, we have shown so far that
\begin{align*}
\ms{D}_{N}(\mf{h})\geq\;[\,1+o_{N}(1)\,]\, & \frac{6\,\kappa}{Z_{N,\,\kappa}\,}\,\sum_{x\in S_{1},\,y\in S_{2}}{\rm cap}_{X}(x,\,y)\\
 & \times\,\Big[\,\frac{1}{(\log N){}^{\kappa-2}}\,\sum_{k=0}^{\ell_{N}/3}\,\sum_{\zeta\in\mc{H}_{k,\,S_{0}}}\,\frac{1}{\mb{a}(\zeta)}\,\Big]\;.
\end{align*}
Now it suffices to apply Lemma \ref{lem_aux1} and \cite[Proposition 4.1]{LMS}
to complete the proof.
\end{proof}

\subsection{\label{sec92}Upper bound}

Now we deduce the upper bound of the capacity.
\begin{prop}
\label{p_upper}Fix a non-empty subset $S_{1}\subsetneq S$, and let
$S_{2}=S\setminus S_{1}$. Then
\begin{equation*}
\limsup\limits _{N\rightarrow\infty}{\rm {cap}}_{N}(\mc{E}_{N}(S_{1}),\,\mc{E}_{N}(S_{2}))\;\le\;6\sum_{x\in S_{1},\,y\in S_{2}}{\rm cap}_{X}(x,\,y)\;.
\end{equation*}
\end{prop}

We remark at this moment that Proposition \ref{p131} is an immediate
consequence of Propositions \ref{p_lower} and \ref{p_upper}.

We fix $S_{1}$ and $S_{2}$ throughout this subsection. We prove
this proposition in an identical manner as \cite[Section 5]{BL3}
at which a test function is constructed and then the upper bound is
established via the Dirichlet principle. This test function can be
constructed as a suitable approximation for the equilibrium potential
$\mf{h}_{\mc{E}_{N}(S_{1}),\,\mc{E}_{N}(S_{2})}$.
We repeat the construction of test function in the exactly same way
as \cite[Section 5]{BL3} here.

The set $\ms{D}\subset\bb{R}^{S}$ is defined as
\begin{equation*}
\ms{D}\;=\;\Big\{\,u\,\in\,[\,0,\,1\,]^{S}\,:\,\sum_{x\in S}u_{x}\,=\,1\,\Big\}\;,
\end{equation*}
and for each $x,y\in S$ and $\epsilon\in(0,\,1/6)$ we define
\begin{align*}
 & \ms{D}_{\epsilon}^{x}\;=\;\{\,u\,\in\,\ms{D}\,:\,u_{x}\,>\,1-\epsilon\,\}\;,\\
 & \ms{L}_{\epsilon}^{x,\,y}\;=\;\{\,u\,\in\,\ms{D}\,:\,u_{x}+u_{y}\,\geq\,1-\epsilon\,\}\;.
\end{align*}
Fix from now on $x\in S$ and $\epsilon\in(0,\,1/6)$. Let $\phi:[0,1]\rightarrow[0,1]$
be a smooth bijective function such that $\phi(t)+\phi(1-t)=1$ for
all $t\in[0,1]$ and $\phi\equiv0$ on $[0,3\epsilon]$. Then, define
${H:[0,1]\rightarrow[0,1]}$ as ${\color{blue}{H(t)=6\int_{0}^{\phi(t)}\,u\,(1-u)\,du}}$.
For each $y\in S\setminus\{x\}$, we write $h_{x,\,y}=h^X_{\{x\},\,\{y\}}$ (cf. \eqref{epX} the equilibrium potential between $x$ and $y$ with respect to the
underlying random walk. Then, enumerate the elements of $S$ by $x=z_{1},\,z_{2},\,\dots,\,z_{\kappa}=y$
in such a manner that
\begin{equation*}
1\,=\,h_{x,\,y}(z_{1})\,\ge\,h_{x,\,y}(z_{2})\,\ge\,\cdots\,\ge\,h_{x,\,y}(z_{\kappa})\,=\,0\;.
\end{equation*}
Then, for each $y\in S\setminus\{x\}$, define $F_{xy}^{j}:\mc{H}_{N}\rightarrow\bb{R},1\leq j\leq\kappa-1$,
as $F_{xy}^{1}(\eta)=H(\eta_{x}/N)$ and
\begin{equation*}
F_{xy}^{j}(\eta)\;=\;H\,\Big(\,\frac{\eta_{x}}{N}\,+\,\min\,\Big\{\,\frac{1}{N}\sum_{i=1}^{j}\eta_{z_{j}},\,\epsilon\,\Big\}\,\Big)\,\;\;\;\;\text{for }j\in[2,\,\kappa-1]\;.
\end{equation*}
Define $F_{xy}:\mc{H}_{N}\rightarrow\bb{R}$ as
\begin{equation*}
F_{xy}(\eta)\;=\;\sum_{j=1}^{\kappa-1}\,[\,h_{x,\,y}(z_{j})\,-\,h_{x,\,y}(z_{j+1})\,]\,F_{xy}^{j}(\eta)\;.
\end{equation*}
For $y\neq x$, write $\ms{K}_{y}^{x}=\ms{L}_{\epsilon}^{x,\,y}\setminus\ms{D}_{3\epsilon}^{x}$.
We can observe that $\kappa-1$ sets $\ms{K}_{y}^{x}$, $y\neq x$,
are pairwise disjoint compact subsets of $\ms{D}$. Therefore,
there exists a smooth partition of unity $\Theta_{y}^{x}:\ms{D}\rightarrow[0,1]$,
$y\neq x$, in the sense that,
\begin{equation*}
\sum_{y\in S\setminus\{x\}}\Theta_{y}^{x}\,\equiv\,1\;\;\text{on }\ms{D}\;\;\;\;\text{and\;\;\;\;}\Theta_{y}^{x}\,\equiv\,1\;\;\text{on }\ms{K}_{y}^{x}\text{ for }y\,\in\,S\setminus\{x\}\;.
\end{equation*}
With the constructions above, we define $F_{x}:\mc{H}_{N}\rightarrow\bb{R}$
as
\begin{equation*}
F_{x}(\eta)\;=\;\sum_{y\in S\setminus\{x\}}\Theta_{y}^{x}(\eta/N)\,F_{xy}(\eta)\;.
\end{equation*}
Finally, the test function $F_{S_{1}}:\mc{H}_{N}\rightarrow\bb{R}$
is defined by
\begin{equation*}
F_{S_{1}}(\eta)\;=\;\sum_{x\in S_{1}}F_{x}(\eta)\;.
\end{equation*}
The main property of this test function is the following lemma.

\begin{lem}
\label{lem_upper1} We have that
\begin{equation*}
\limsup\limits _{N\rightarrow\infty}\ms{D}_{N}(F_{S_{1}})\leq6\,(\,1+\epsilon^{1/2}\,)^{3}\,\sum_{x\in S_{1},\,y\in S_{2}}{\rm cap}_{X}(x,\,y)\;.
\end{equation*}
\end{lem}

We omit the proof of this lemma since it is identical to those of
\cite[(5.11), (5.12), Proposition 5.3]{BL3}. Even if they proved
these results for $\alpha>1$, the same argument also holds for $\alpha=1$.
The only different part is \cite[Lemma 5.2]{BL3} which is used in
the proof of \cite[(5.11)]{BL3}. We substitute this lemma by the
following lemma.
\begin{lem}
\label{lem_upper2}For $x,\,y\in S$, define
\begin{equation*}
\ms{I}_{N}^{xy}\;=\;\{\,\eta\,\in\,\mc{H}_{N}\,:\,\eta_{x}+\eta_{y}\,\geq\,N-\ell_{N}\,\}\;.
\end{equation*}
and let $\ms{I}_{N}^{x}=\cup_{y\in S\setminus\{x\}}\ms{I}_{N}^{xy}$.
Then, for all $x\in S$ and $\epsilon\in(0,\,1/6)$, there exists
a constant $C_{\epsilon}>0$ depending only on $\epsilon>0$ such
that
\begin{equation}
\frac{1}{2}\!\sum_{\eta\in\mc{H}_{N}\setminus\ms{I}_{N}^{x}}\!\sum_{z,w\in S}\mu_{N}(\eta)\,\mb{g}(\eta_{z})\,r(z,w)\,[\,F_{x}(\sigma^{zw}\eta)-F_{x}(\eta)\,]^{2}\;\leq\;\frac{C_{\epsilon}\,(\log\log N)^{2}}{Z_{N,\,\kappa}\,N^{2}\,(\log N)^{2}}\;.\label{e_96}
\end{equation}
\end{lem}

\begin{proof}
Basically, we perform the same proof as \cite[Lemma 5.2]{BL3}, but
the fact that $\alpha=1$ makes it slightly different.

Since $F_{x}(\eta)=1$ if $\eta_{x}\geq(1-3\epsilon)N$ and $F_{x}(\eta)=0$
if $\eta_{x}\leq2\epsilon N$, we can restrict the first sum of \eqref{e_96}
to configurations $\eta\in\mc{H}_{N}\setminus\ms{I}_{N}^{x}$
satisfying $\epsilon N\leq\eta_{x}\leq(1-\epsilon)N$. Since there
exists $C_{\epsilon}>0$ such that
\begin{equation*}
\max\limits _{\eta\in\mc{H}_{N}}\,|\,F_{x}(\sigma^{z,\,w}\eta)-F_{x}(\eta)\,|\;\leq\;\frac{C_{\epsilon}}{N}\;,
\end{equation*}
we can bound the left-hand side of \eqref{e_96} from above by
\begin{equation}
\label{e_97}
\begin{aligned}
 & \frac{C_{\epsilon}}{N^{2}}\sum_{\substack{\eta\in\mc{H}_{N}\setminus\ms{I}_{N}^{x}:\\
\epsilon N\leq\eta_{x}\leq(1-\epsilon)N
}
}\mu_{N}(\eta) \\
 & \qquad\leq\;\frac{C_{\epsilon}\,N}{Z_{N,\,\kappa}\,(\log N)^{\kappa-1}\,N^{2}}\sum_{i=\epsilon N}^{(1-\epsilon)N}\!\sum_{\substack{\eta:\eta_{x}=i\\
\max\{\eta_{y}:y\neq x\}\leq N-i-\ell_{N}
}
}\frac{1}{\mb{a}(\eta)} \\
 & \qquad=\;\frac{C_{\epsilon}}{Z_{N,\,\kappa}\,(\log N)^{\kappa-1}\,N}\sum_{i=\epsilon N}^{(1-\epsilon)N}\sum_{\xi\in\mc{H}_{N-i,S\setminus\{x\}}(\ell_{N})}\frac{1}{i}\,\frac{1}{\mb{a}(\xi)}\;,
\end{aligned}
\end{equation}

where, for $S_{0}\subset S$, we define
\begin{equation}
\mc{H}_{N,\,S_{0}}(\ell)\;=\;\{\eta\,\in\,\mc{H}_{N,\,S_{0}}\,:\,\eta_{x}\leq
N-\ell\text{ for all }x\in S_{0}\}\;.
\label{e_Hns}
\end{equation}
Since $\mc{H}_{N-i,\,S\setminus\{x\}}(\ell_{N})$ is a subset
of $\mc{H}_{N-i,\,S\setminus\{x\}}(\ell_{N-i})$, we can further
bound \eqref{e_97} from above by
\begin{align*}
 & \frac{C_{\epsilon}}{Z_{N,S}\log(N)^{\kappa-1}N}\sum_{i=\epsilon N}^{(1-\epsilon)N}\frac{1}{i}\sum_{\xi\in\mc{H}_{N-i,S\setminus\{x\}}(\ell_{N-i})}\frac{1}{\mb{a}(\xi)}\\
 & \leq\frac{C_{\epsilon}\log(\log(N))}{Z_{N,S}N^{2}\log(N)}\left\{ \frac{N}{\log(N)^{\kappa-1}}\sum_{i=\epsilon N}^{(1-\epsilon)N}\frac{\log(N-i)^{\kappa-2}}{i(N-i)}\right\} \\
 & \leq\frac{C_{\epsilon}\log(\log(N))}{Z_{N,S}N^{2}\log(N)}\left\{ \frac{N}{\log(N)}\sum_{i=N/\log(N)}^{N-N/\log(N)}\frac{1}{i(N-i)}\right\} \leq\frac{C_{\epsilon}\log(\log(N))^{2}}{Z_{N,S}N^{2}\log(N)^{2}}
\end{align*}
where the first and last inequality follows from Lemma \ref{lem_aux2}.
\end{proof}
Now we are ready to prove the upper bound.
\begin{proof}[Proof of Proposition \ref{p_upper}]
It is immediate that $F_{S_{1}}$ satisfies
\begin{equation*}
F_{S_{1}}(\eta)\;=\;\begin{cases}
1 & \text{if }\eta\in\mc{E}_{N}(S_{1})\\
0 & \text{if }\eta\in\mc{E}_{N}(S_{2})\;.
\end{cases}
\end{equation*}
Hence, by Lemma \ref{lem_upper1} and the Dirichlet principle, we
get
\begin{equation*}
\limsup\limits _{N\rightarrow\infty}{\rm {cap}}_{N}(\mc{E}_{N}(S_{1}),\,\mc{E}_{N}(S_{2}))\;\leq\;(\,1+\epsilon^{1/2}\,)^{3}\,\frac{6}{\kappa}\,\sum_{x\in S_{1},\,y\in S_{2}}{\rm cap}_{X}(x,\,y)\;.
\end{equation*}
Letting $\epsilon\rightarrow0$ completes the proof.
\end{proof}

\subsection{Auxiliary Lemmata\label{sec93}}

In this subsection we prove two technical lemmata. The first one is
used in the proof of Proposition \ref{p_lower}.
\begin{lem}
\label{lem_aux1}For all $c>0$ and $x,\,y\in S$, we have
\begin{equation}
\lim\limits _{N\rightarrow\infty}\frac{1}{(\log N)^{\kappa}}\,\sum_{n=0}^{c\ell_{N}}\,\sum_{\xi\in\mc{H}_{n,\,S}}\frac{1}{\mb{a}(\xi)}\;=\;1\;.\label{e_99}
\end{equation}
\end{lem}

\begin{proof}
By the definition \eqref{e_partition} of the partition function,
we can rewrite \eqref{e_99} as
\begin{equation*}
\lim_{N\rightarrow\infty}\frac{1}{(\log N)^{\kappa}}\,\sum_{n=1}^{c\ell_{N}}\frac{(\log n)^{\kappa-1}}{n}\,Z_{n,\,\kappa}\;=\;1
\end{equation*}
as the case $n=0$ is negligible. Since $(Z_{n,\,\kappa})_{n\in\bb{N}}$
is bounded by \cite[Proposition 4.1]{LMS}, and since
\begin{equation*}
\frac{1}{(\log\log N)^{\kappa}}\sum_{n=1}^{\log N}\frac{(\log n)^{\kappa-1}}{n}\;\simeq\;\frac{1}{\kappa}\;,
\end{equation*}
it suffices to prove that
\begin{equation*}
\lim_{N\rightarrow\infty}\frac{1}{(\log N)^{\kappa}}\,\sum_{n=\log N}^{c\ell_{N}}\frac{(\log n)^{\kappa-1}}{n}\,Z_{n,\,\kappa}\;=\;1\;.
\end{equation*}
This follows from \cite[Proposition 4.1]{LMS}, the elementary fact
that
\begin{equation*}
\sum_{n=\log N}^{c\ell_{N}}\frac{(\log n)^{\kappa-1}}{n}\;\simeq\;\frac{(\log(c\ell_{N}))^{\kappa}-(\log\log N)^{\kappa}}{\kappa}
\end{equation*}
and that $\lim_{N\rightarrow\infty}\log\ell_{N}/\log N=1$.
\end{proof}
The next one is used in the proof of Proposition \ref{p_upper}. Recall
the definition of $\mc{H}_{N,\,S_{0}}(\ell)$ from \eqref{e_Hns}.
We simply write $\mc{H}_{N}(\ell)=\mc{H}_{N,\,S}(\ell)$.
\begin{lem}
\label{lem_aux2}Define $\ell_{N}^{(a)}=N/(\log N)^{a}$. For $\kappa\geq2$
and $a>0$, there exists a constant $C=C_{\kappa,\,a}>0$ such that
\begin{equation}
\frac{N}{(\log N)^{\kappa-1}}\,\sum_{\eta\in\mc{H}_{N}(\ell_{N}^{(a)})}\frac{1}{\mb{a}(\eta)}\;\leq\;C\,\frac{\log\log N}{\log N}\;.\label{e_910}
\end{equation}
\end{lem}

\begin{proof}
We proceed by induction. For $\kappa=2$, we can rewrite and bound
the left-hand side of \eqref{e_910} as
\begin{align*}
\frac{N}{\log N}\sum_{n=\ell_{N}^{(a)}}^{N-\ell_{N}^{(a)}}\frac{1}{n\,(N-n)} & \;\leq\;\frac{C}{\log N}\,\sum_{n=\ell_{N}^{(a)}}^{N/2}\frac{1}{n}\;\leq\;C\,\frac{\log\log N}{\log N}\;.
\end{align*}

Now we assume the result holds for all $\kappa\in[2,\,\kappa_{0}-1]$
and for all $a>0$. Then, we will show that \eqref{e_910} holds for
$\kappa=\kappa_{0}$ and $a>0$. Define, for $n\le\ell_{N}^{(a+1)}$,
\begin{equation*}
\mc{A}_{N,\,n}^{x,\,a}\;=\;\big\{\,\xi\,\in\,\mc{H}_{N-n,\,S\setminus\{x\}}\,:\,\xi_{y}\,\leq\,N-\ell_{N}^{(a)}\text{ for all }y\,\in\,S\setminus\{x\}\,\big\}\;.
\end{equation*}
We first claim that
\begin{equation}
\mc{A}_{N,\,n}^{x,\,a}\;\subset\;\mc{H}_{N-n,\,S\setminus\{x\}}(\,\ell_{N-n}^{(a+1)}\,)\;.\label{e_911}
\end{equation}
To verify this, it suffices to check that,
\begin{equation*}
(N-n)\,-\,\frac{N-n}{(\log(N-n))^{a+1}}\;\geq\;N\,-\,\frac{N}{(\log N)^{a}}\;\cdot
\end{equation*}
This follows for $n\le\ell_{N}^{(a+1)}$ from the inequality
\begin{equation*}
n\,+\,\frac{N-n}{(\log(N-n))^{a+1}}\;\leq\;\frac{2N}{(\log N)^{a+1}}\;\leq\;\frac{N}{(\log N)^{a}}\;.
\end{equation*}
Observe that we can write
\begin{align}
 & \frac{N}{(\log N)^{\kappa_{0}-1}}\,\sum_{\eta\in\mc{H}_{N}(\ell_{N}^{(a)})}\frac{1}{\mb{a}(\eta)}\nonumber \\
=\; & \Big\{\!\sum_{n=0}^{\ell_{N}^{(a+1)}}\,+\!\sum_{n=\ell_{N}^{(a+1)}+1}^{N-\ell_{N}^{(a)}}\,\Big\}\!\Big[\,\frac{N}{\mb{a}(n)\,(\log N)^{\kappa_{0}-1}}\!\sum_{\xi\in\mc{A}_{N,\,n}^{x,\,a}}\!\frac{1}{\mb{a}(\xi)\,}\,\Big]\;.\label{e_912}
\end{align}
By the induction hypothesis and \eqref{e_911}, the first summation
above is bounded by
\begin{align*}
 & \frac{N}{(\log N)^{\kappa_{0}-1}}\,\sum_{n=0}^{\ell_{N}^{(a+1)}}\,\Big[\,\frac{1}{\mb{a}(n)}\,\sum_{\xi\in\mc{H}_{N-n,S\setminus\{x\}}(\ell_{N-n}^{(a+1)})}\,\frac{1}{\mb{a}(\xi)}\,\Big]\\
\leq\; & \frac{CN}{(\log N)^{\kappa_{0}-1}}\,\sum_{n=0}^{\ell_{N}^{(a+1)}}\frac{1}{\mb{a}(n)}\,\frac{(\log(N-n))^{\kappa_{0}-2}}{N-n}\,\frac{\log\log(N-n)}{\log(N-n)}\\
\leq\; & C\frac{\log\log N}{(\log N)^{2}}\,\sum_{n=0}^{\ell_{N}^{(a+1)}}\frac{1}{\mb{a}(n)}\;\le\;C\,\frac{\log\log N}{\log N}\;.
\end{align*}
On the other hand, for the second summation of \eqref{e_912}, we
can enlarge $\mc{A}_{N,\,n}^{x,\,a}$ to $\mc{H}_{N-n,\,S\setminus\{x\}}$
so that the summation is bounded by
\begin{equation*}
C\,\frac{N}{(\log N)^{\kappa_{0}-1}}\!\sum_{n=\ell_{N}^{(a+1)}}^{N-\ell_{N}^{(a+1)}}\,\frac{Z_{N,\,\kappa_{0}-1}\,(\log(N-n))^{\kappa_{0}-2}}{(N-n)}\,\frac{1}{\mb{a}(n)}\;\leq\;C\,\frac{\log\log N}{\log N}
\end{equation*}
by \cite[Proposition 4.1]{LMS}.
\end{proof}

\noindent \textbf{Acknowledgment.} C. L. has been partially supported
by FAPERJ CNE E-26/201.207/2014, by CNPq Bolsa de Produtividade em
Pesquisa PQ 303538/2014-7, by ANR-15-CE40-0020-01 LSD of the French
National Research Agency. D. M. has received financial support from
CNPq during the development of this paper. I. S. was supported by
the Samsung Science and Technology Foundation
(Project Number SSTF-BA1901-03) and  National Research Foundation(NRF) of Korea grant funded by the Korea government(MSIT) (No. 2022R1F1A106366811 and 2022R1A5A600084012).

\bibliographystyle{authordate1}

\end{document}